\documentclass[11pt,reqno]{amsart}
\usepackage[utf8]{inputenc}
\usepackage[T1]{fontenc}
\usepackage{lmodern}
\usepackage[english]{babel}
\usepackage{amsmath,a4wide}
\usepackage{xfrac}
\usepackage{esint}
\usepackage{stmaryrd,mathrsfs,bm,amsthm,mathtools,yfonts,amssymb,color,braket,booktabs,graphicx,graphics,amsfonts}
\usepackage{latexsym,microtype,indentfirst,hyperref}
\usepackage{xcolor}
\usepackage{courier}
\usepackage{constants}
\usepackage{tikz}
\usetikzlibrary{arrows.meta}

\newtheorem{theorem}{Theorem}[section]
\newtheorem{corollary}[theorem]{Corollary}

\newtheorem{lemma}[theorem]{Lemma}
\newtheorem{proposition}[theorem]{Proposition}

\theoremstyle{definition}
\newtheorem{definition}[theorem]{Definition}
\newtheorem{remark}[theorem]{Remark}

\numberwithin{equation}{section}
\numberwithin{subsection}{section}

\newcommand{\R}{\mathbb{R}} 

\newcommand{\spt}{\mathrm{spt}} 
\newcommand{\dist}{\mathrm{dist}} 
\newcommand{\Lip}{\mathrm{Lip}} 
\newcommand{\Ha}{\mathcal{H}} 
\newcommand{\pa}{\partial} 

\newcommand{\mres}{\mathbin{\vrule height 1.6ex depth 0pt width 
0.13ex\vrule height 0.13ex depth 0pt width 1.3ex}}

\newcommand{\abs}[1]{\lvert#1\rvert} 

\newcommand{\var}{\mathbf{var}} 
\newcommand{\Sing}{\mathrm{Sing}} 

\newcommand{\bC}{\mathbf{C}} 

\newcommand{\bS}{\mathbf{S}} 

\newconstantfamily{eps}{
symbol=\varepsilon}
\newconstantfamily{con}{
symbol=c}

\title[Brakke flows near triple junctions]{The epsilon-regularity theorem for \\ Brakke flows  near triple junctions} 
\date{\today}

\author[S. Stuvard]{Salvatore Stuvard}
\address{Dipartimento di Matematica, Universit\`{a} degli Studi di Milano, Via Saldini 50, I-20133 Milano (MI), Italy}
\email{salvatore.stuvard@unimi.it}

\author[Y. Tonegawa]{Yoshihiro Tonegawa}
\address{Department of Mathematics, Institute of Science Tokyo, 2-12-1 Ookayama, Meguro-ku, Tokyo 152-8551, Japan}
\email{tonegawa@math.titech.ac.jp}

\begin{document}

\begin{abstract}
We establish the $\varepsilon$-regularity theorem for $k$-dimensional, possibly forced, Brakke flows near a static, multiplicity-one triple junction. This result provides the parabolic analogue to L. Simon's foundational work on the singular set of stationary varifolds and confirms that the regular structure of triple junctions persists under weak mean curvature flow. The regularity holds provided the flow satisfies a mild structural assumption on its 1-dimensional slices taken orthogonal to the junction's $(k-1)$-dimensional spine, which prohibits certain topological degeneracies. We prove that this assumption is automatically satisfied by two fundamental classes of flows where such singularities are expected: codimension-one multi-phase flows, such as the canonical $\mathrm{BV}$-Brakke flows constructed by the authors, and flows of arbitrary codimension with the structure of a mod~3 integral current, which arise from Ilmanen's elliptic regularization. For such flows, therefore, the Simon type regularity holds unconditionally.
\end{abstract}

\maketitle

\tableofcontents

\section{Introduction} \label{sec:intro}

A central theme in geometric analysis is the structure of singularities in critical points of geometric variational problems. The regularity theory for stationary varifolds — weak solutions to the Euler–Lagrange equation of the area functional — forms one of the most challenging parts of the field. Following the fundamental monotonicity formula, which yields subsequential convergence of blow-ups of a stationary varifold to stationary cones (tangent cones), three questions became central. First: does the occurrence of a \emph{regular} tangent cone (a plane, possibly weighted with constant multiplicity $Q$) at a point force local regularity of the varifold? Second: more generally, does the occurrence of a given tangent cone imply its \emph{uniqueness} (independence of blow-up sequences) and that the varifold is locally diffeomorphic to that cone? Third: what can be said about the size (Hausdorff dimension) and fine properties (rectifiability, higher regularity) of the singular set? While complete answers remain out of reach, remarkable partial results have been obtained in the last few decades, with various degrees of precision and possibly under various additional assumptions, such as stability or area minimization; see, among others, \cite{Allard,Almgren_brp,DLS_Lp,DLS_cm,DLS_bu,NV_varifolds,Simon_Loj,white-immiscible,Wic}. 

\smallskip

A parallel and natural line of research concerns the parabolic counterpart: weak varifold solutions to the $L^2$-gradient flow of the area functional — the mean curvature flow. The relevant notion of weak solution in this context was given by K. Brakke in \cite{Brakke_mcf}, whence it is typically referred to as Brakke flow; see Section \ref{sec:results} for the relevant definitions. In this parabolic framework, the monotonicity formula of Huisken \cite{Huisken_mono}, originally proved for smooth mean curvature flows and extended  to Brakke flows by Ilmanen in \cite{Ilmanen_mono}, allows one to mirror the elliptic theory. In particular, it establishes subsequential convergence of \emph{parabolic} blow-ups of a Brakke flow at a space-time point of its support to limit tangent flows, and the analogues of the three questions above can then be asked in this framework, too. This paper resolves, in arbitrary dimension and codimension, the analogue of the second question at static multiplicity-one triple junctions, namely when a tangent flow is independent of time and equal to the stationary cone $\bC$ given by the union of three half-planes meeting at $120^\circ$ along a common subspace. Precisely, we establish the following regularity theorem, presented here informally and stated rigorously in Theorem \ref{thm:main}.

\begin{theorem}[Main theorem, informal statement] \label{thm:main-informal}
    Suppose $\{V_t\}_{t \in (-1,0]}$ is a $k$-dimensional Brakke flow in the open ball $U_1(0) \subset \R^n$ satisfying the structural assumption (A6). If a tangent flow at $0$ is a static multiplicity-one triple junction $\bC$ for $t\leq 0$, then the following holds:
    \begin{itemize}
        \item[(i)]  $\bC$ is the \emph{unique} tangent flow at $0$;
        \item[(ii)]  the parabolic blow-ups $V_t^{(0,0),\lambda}$, informally $V_t^{(0,0),\lambda}=\lambda^{-1} V_{\lambda^2t}$, converge to $\bC$ at a rate $\mathrm{O}(\lambda^\alpha)$ as $\lambda \to 0^+$ for every $\alpha \in (0,1)$;
        \item[(iii)] there exists $r > 0$ such that, for every $t \in (-r^2,0)$, $\spt\|V_t\| \cap U_r(0)$ consists of the union of three $k$-dimensional submanifolds-with-boundary meeting at $120^\circ$ along a common boundary. In fact the three submanifolds are normal graphs over the three branches of $\bC$, and the common boundary is a normal graph over the axis (spine) of $\bC$. Furthermore, the boundary is regular of class $C^{1,\alpha}$, and each sheet is a smooth solution to the mean curvature flow in the interior and $C^{1,\alpha}$ up to the common boundary.  
    \end{itemize}

In fact, an $\varepsilon$-regularity statement holds. More precisely, there exists $\varepsilon_0>0$ such that, under (A6), if the flow is $\varepsilon_0$-close to a static multiplicity-one triple junction (as specified in Definition \ref{def:eps-nbd}) and the Gaussian density $\Theta(0,0)\ge\tfrac32$, then the tangent flow at $(0,0)$ is a unique static triple junction $\bC_{(0,0)}$, and the conclusions above hold with $\bC_{(0,0)}$ replacing $\bC$. The statements remain valid for Brakke flows with forcing $u\in L_t^{q}L_X^{p}$ with $p\ge2$, $q>2$, and 
$\alpha_\star:=1-\frac{k}{p}-\frac{2}{q}\in(0,1)$; the rate is $\mathrm{O}(\lambda^\alpha)$ 
for every $\alpha\le \alpha_\star$.
\end{theorem}

\smallskip

Before commenting further on the theorem and its assumptions (particularly the structural assumption (A6) mentioned there), let us first discuss the relevance of studying triple junction singularities. The intrinsic interest of these specific singularities dates back to J. Plateau's experiments with soap films and bubbles in the nineteenth century. From the phenomenological observations, Plateau formulated what are nowadays known as \emph{Plateau's laws} on the shape of soap films. The laws predict that:
\begin{itemize}
    \item[(i)] Soap films are made of entire (unbroken) smooth surfaces.
    \item[(ii)] The mean curvature of a portion of a soap film is everywhere constant on any point on the same piece of soap film.
    \item[(iii)]  Soap films always meet in threes along an edge called a Plateau border, and they do so at an angle of $\arccos(-1/2)=120^\circ$.
    \item[(iv)] These Plateau borders meet in fours at a vertex, at the tetrahedral angle $\arccos(-1/3) \approx 109.47^\circ$.
\end{itemize}
The local analysis of triple junction singularities of minimal surfaces (the elliptic, stationary setting) is then the rigorous study, in mathematical terms, of the local geometry at the Plateau borders described above in (iii). In \cite{Taylor_AlmMin}, J. Taylor demonstrated that two-dimensional Almgren minimal sets in $\R^3$ (see \cite{Alm_minsets}) do satisfy Plateau's laws. Later on, in \cite{Taylor}, she identified triple junctions as the \emph{only} admissible singularities for two-dimensional flat chains in $\R^3$ that minimize the area in the homology class mod~3. In his pioneering paper \cite{Simon_cylindrical}, L. Simon eventually proved an $\varepsilon$-regularity theorem for triple junction singularities of multiplicity-one stationary varifolds in arbitrary dimension and codimension. In particular, his result implies that if a multiplicity-one stationary varifold admits, at a point of its support, a unit density triple junction as tangent cone, then that tangent cone is unique, blow-ups converge towards it at a rate that is a positive power of the blow-up scale, and locally at the point the varifold consists of three smooth minimal surfaces meeting at $120^\circ$ at a common $C^{1,\alpha}$ boundary. The techniques introduced by Simon have proved themselves to be extremely robust, and they have been successfully applied to a variety of elliptic problems concerning the local structure of cylindrical singularities (namely, singularities where a tangent cone splits a Euclidean factor); see, in particular, \cite{CoEdSp,DLHMSS_odd,MDLS,MWic}. 

\smallskip

In contrast, much less progress has been made in this direction at the level of parabolic theory. Prior to the present contribution, the analysis of Brakke flows at triple junction singularities had been addressed in two papers. In \cite{SW20}, F. Schulze and B. White considered the restricted class of mean curvature flows of clusters of smooth $k$-dimensional surfaces in $\R^n$ that meet in triples at equal angles along smooth edges and higher-order junctions on lower-dimensional faces, termed \emph{mean curvature flows with triple edges}. They proved that any such flow that is close to a static triple junction weakly in the sense of Brakke flows is in fact close to it in the smooth topology, and furthermore they showed that any cluster with only triple edges and no higher-order junctions evolves by mean curvature within the class for short time. Instead, in \cite{tone-wic} the second-named author and N. Wickramasekera implemented Simon's blow-up technique to prove an $\varepsilon$-regularity theorem for $1$-dimensional Brakke flows that are $L^2$-close to a static triple junction in the plane, thereby establishing the $k=1$, $n=2$ case of Theorem~\ref{thm:main-informal}, without the structural assumption (A6). 

\smallskip

The essential advantage of the $k=1$ case is the following. Along a mean curvature flow, it is natural to control the space-time integral of the mean curvature squared, since this quantity represents the dissipation of area along the flow. When $k=1$, this natural control implies that for a.e. $t$ the $1$-dimensional varifold $V_t$ has bounded generalized mean curvature in $L^2$. In turn, this information entails strong constraints on the topology (and length) of the rectifiable set supporting $V_t$; see, for instance \cite[Proposition 4.2]{tone-wic}. When the varifolds' dimension is $k\geq 2$, such constraints are unavailable. With such a weak, albeit natural, integrability condition on the mean curvature, it is not clear to the authors whether a complete, unconditional parabolic counterpart to Simon's theorem in \cite{Simon_cylindrical} is to be expected. 

More precisely, while many of the formulas in \cite{Simon_cylindrical} have parabolic counterparts for Brakke flows, one crucial estimate does not. For a stationary integral varifold $V$ of dimension $k$ and a stationary cone $\bC$ of the same dimension, if $0 \in \spt\|V\|$ is a point such that the density $\Theta_V(0) \geq \Theta_{\bC}(0)$ then as a consequence of the monotonicity formula it holds for a.e.~$r>0$ that  \begin{equation}\label{simonineq}
    kr^{k-1}\int_{B_r}\frac{|X^\perp|^2}{|X|^{k+2}}\,d\|V\|(X)\leq \frac{d}{dr}\big(\|V\|(B_r)-\|\bC\|(B_r)\big),
\end{equation}
where $X^\perp$ is the orthogonal projection of the position vector $X$ to $({\rm Tan}_X \|V\|)^\perp$; see \cite[p.~613]{Simon_cylindrical} and \cite[Lemma 8.2 and Appendix E]{DLHMSS_odd}. The above formula allowed Simon to avoid certain derivative terms of cut-off 
functions needed for localization. For Brakke flows, a direct counterpart to \eqref{simonineq} is missing: while the natural attempt would be to try and derive a suitable estimate from Huisken's monotonicity formula, the intrinsic ``non-local'' nature of the latter prevents one to obtain an inequality which can be successfully integrated against radial cut-off functions. The local monotonicity formula of Ecker \cite{Ecker} does not appear to provide the kind of estimates we need either. 

In this paper, we show that the lack of a formula mirroring \eqref{simonineq} for mean curvature flow can be entirely overcome by imposing one single additional structural assumption on the flow, which, in the present paper, is labeled (A6) and is introduced in Section \ref{futh}. Although it is not precise, (A6) roughly requires the following. If in a parabolic cylinder $B_r \times (-r^2,0)$ the flow is sufficiently close, in space-time $L^2$, to a triple junction $\bC$, then at a.e. time $t$ the slices of $V_t$ in the direction perpendicular to the spine $\bS(\bC)$ of the cone, a one-dimensional rectifiable set $M_t^y$ for a.e. $y \in \bS(\bC)$, have scale invariant length
\[
r^{-1}\mathcal H^1(M_t^y \cap B_r)
\]
and scale invariant second moment with respect to $\bS(\bC)$
\[
r^{-3}\int_{M_t^y\cap B_r} \dist^2(X,\bS(\bC))  \,d\mathcal H^1(X)
\]
greater than or equal to the corresponding quantities evaluated on the slices of $\bC$, up to an admissible error which is quadratic in the scale invariant $L^\infty$ distance between the flow and the cone in the annulus $B_r \setminus B_{r/2}$ (see \eqref{cotriple}–\eqref{cotriple2}). 

Essentially, the validity of (A6) prevents the occurrence of topological degeneracies in the slices of the flow such as those depicted in Figure \eqref{fig:A6}. Interestingly, we verify in this paper (see Section \ref{lastapp}) that (A6) is in fact automatically satisfied in the two canonical classes of Brakke flows where triple junction singularities are expected and a robust existence theory is available: (i) \emph{multi-phase Brakke flows} with at least three phases, such as those constructed in \cite{kim-tone,ST_canonical}, and (ii) \emph{flows of currents mod $3$}, which can be obtained, for instance, by elliptic regularization as in \cite{Ilm1}. In these cases, the parabolic counterpart to Simon's regularity theorem holds unconditionally, and we have the following corollary; the precise statements are Theorem \ref{profor1.1} and Theorem \ref{thm:main-mod3}.  

\begin{corollary} \label{cor-applications}
There exists $\varepsilon_0$ with the following property. Suppose that $\{V_t\}_{t\in(-1,0]}$ is a $k$-dimensional (possibly forced) Brakke flow in the open ball $U_1(0)\subset \R^n$ which is $\varepsilon_0$-close, in the sense of Definition \ref{def:eps-nbd}, to a static multiplicity-one triple junction $\bC$. Suppose furthermore that $\{V_t\}_t$ has a multi-phase cluster structure (in which case $n=k+1$) or that it is a flow of currents mod~3. If the Gaussian density $\Theta(0,0) \geq \tfrac32$, then the tangent flow at $(0,0)$ is unique, it is a static triple junction $\bC_{(0,0)}$, and the conclusions (ii) and (iii) of Theorem \ref{thm:main-informal} hold with $\bC_{(0,0)}$ in place of $\bC$.
\end{corollary}

Note that in case $\{V_t\}_{t \in (-1,0]}$ is a flow of currents mod~3 arising as limit of Ilmanen's elliptic regularization scheme, triple junction regularity had already been established by Schulze-White in \cite[Lemma 5.2]{SW20} (also utilizing Krummel's result \cite{Krummel}), and in fact more can be said in that case, as the graphical sheets in point (iii) of Theorem \ref{thm:main-informal} are \emph{smooth} up to the common boundary, and the latter is smooth as well. An advantage of our result, however, is that it shows how the basic $C^{1,\alpha}$ regularity does not depend in any way on the method used to construct the flow, so long as the underlying mod~3 homology structure is present.

\medskip

In combination with White's stratification theorem \cite{White_stratification}, Corollary \ref{cor-applications} allows one to conclude the following structural result on the singular set of a Brakke flow. Given an open interval $I \subset \R$, an open set $U \subset \R^n$, and a $k$-dimensional Brakke flow $\mathscr{V}=\{V_t\}_{t\in I}$ in $U$, the (interior) \emph{singular set} $\Sing\,\mathscr{V}$ is defined as the set of points $(X,t) \in U \times I$ for which no parabolic neighborhood $U_r(X) \times (t-r^2,t+r^2)$ can be found where the support $\spt\|\mathscr V\|$ (where $\|\mathscr V\|$ is the measure $\|V_t\|\otimes \mathscr L^1$ in $U \times I$) is a smooth mean curvature flow.
\begin{theorem}\label{thm:structure-intro}
    Let $I \subset \R$ be an open interval, and let $\mathscr V =\{V_t\}_{t\in I}$ be a $k$-dimensional Brakke flow in an open set $U \subset \R^n$ with multi-phase cluster structure or that is a flow of currents mod~3. Assume the following:
    \begin{equation}\tag{$\mathrm{H}$}\label{ipotesi-H}
    \begin{split}
        &\mbox{$\mathscr V$ has no static tangent flows having, after rotations, the form $\bC^{(0)}\times\R^{k-1}$} \\
        &\mbox{for a one-dimensional (stationary) cone $\bC^{(0)}$ with $\Theta(\bC^{(0)},0)\geq 2$.} 
    \end{split}
    \end{equation}
    Then, the singular set $\Sing\,\mathscr{V}$ admits the decomposition 
    \begin{equation} \label{eq:sing-decomp}
        \Sing\,\mathscr{V}=\mathcal R \cup \mathcal S\,,
    \end{equation}
    where
    \begin{itemize}
        \item[(i)] $\mathcal R$ is a $k$-dimensional submanifold of $U \times I$ of class $C^{1,\alpha}$, and
        \item[(ii)] $\mathcal S$ has parabolic Hausdorff dimension $\dim_{\mathcal P}(\mathcal S) \leq k$.  
    \end{itemize}
    In particular, for a.e. $t \in I$ the singular set at time $t$, that is the set $(\Sing\,\mathscr{V})_t := \{X\,\colon\,(X,t)\in\Sing\,\mathscr{V}\}$ decomposes as $(\Sing\,\mathscr{V})_t = \mathcal R_t \cup \mathcal S_t$, where $\mathcal R_t$ is a $(k-1)$-dimensional submanifold of $U$ of class $C^{1,\alpha}$ and the Euclidean Hausdorff dimension of $\mathcal S_t$ is $\dim_{\mathcal H}(\mathcal S_t) \leq k-2$.
\end{theorem}
We remark that, for any Brakke flow $\mathscr V= \{V_t\}$, if the Gaussian density at a point $(X_{\mathrm o},t_{\mathrm{o}})$ satisfies $\Theta(\mathscr{V},(X_{\mathrm o},t_{\mathrm{o}})) < 2$ then, by upper semicontinuity of the Gaussian density, the assumption \eqref{ipotesi-H} is satisfied in a neighborhood $U'\times I'\ni(X_{\mathrm o},t_{\mathrm{o}})$. Hence, if $\mathscr{V}$ has multi-phase cluster structure or is a flow of currents mod~3 then the structural decomposition \eqref{eq:sing-decomp} holds for $\Sing\,\mathscr V \cap (U'\times I')$ in this case. 

\medskip

In the following subsection, we will describe the structure of the paper and the plan of the proof of our main results.

\subsection{Plan of the paper and strategy of proof}

In Section \ref{sec:results} we introduce the relevant notation in place throughout the paper, including those related to the geometry of triple junctions. We also recall the fundamental facts from the theory of varifolds, and we then recall the notion of Brakke flows with forcing. In Definition \ref{def:eps-nbd}, we list the conditions for a flow to be $\varepsilon$-close to a triple junction $\bC$, defining the $\varepsilon$-neighborhood $\mathcal N_\varepsilon(\bC)$, Finally, we discuss extensively the structural assumption (A6), we state our Main Theorem \ref{thm:main}, and we show how Theorem \ref{thm:structure-intro} follows via stratification technique.

Towards the proof of Theorem \ref{thm:main}, the most important result is the Decay Theorem \ref{thm:decay}. The latter states, roughly speaking, that if the flow is sufficiently close, at a given scale $r$, to a triple junction $\bC$ then there exists another triple junction $\bC'$ of the form $\bC'=a'+O'(\bC)$ for a small translation vector $a'$ and a rotation $O'$ close to the identity so that the (scale invariant) $L^2$-excess (in space-time) of the flow from $\bC'$ at a smaller scale $\theta_\star r$ (for some $\theta_\star \in (0,1)$) has decayed by a fixed factor $\theta_\star^\alpha$. This information is essentially sufficient to establish uniqueness of tangent flows that are static triple junctions, as well as the corresponding rate of convergence of blow-up sequences. To prove the structure theorem (statement (iii) in the informal statement presented as Theorem \ref{thm:main-informal}), the other important ingredient is Proposition \ref{p:NH_property}, which establishes the validity of the following ``no-hole property'': provided the flow is sufficiently $L^2$-close to $\bC$, at every time $t$ and for every point $y$ on the spine of $\bC$ the slice of $V_t$ perpendicular to the spine of $\bC$ and passing through $y$ contains at least one point where the Gaussian density of the flow is (bigger than or) equal to $\tfrac32$. Once these results have been established, the proof of the structure theorem is obtained upon comparing the oscillation of the (unique) tangent flows at different ``no-hole'' points, and is by now considered standard.

With the no-hole property established in Section \ref{sec:NH}, essentially all the effort through Sections \ref{sec:graph}, \ref{sec:HS}, \ref{sec:BU}, and \ref{sec:proof-main} is directed to the proof of the Decay Theorem \ref{thm:decay}. As in many similar regularity proofs (starting from the pioneering work of De Giorgi \cite{DG_frontiere}) the main argument is a ``blow-up'' procedure: after scaling, we focus on a sequence of Brakke flows $\{V_t^{(m)}\}_t$ with forcing fields $u^{(m)}$ which are close at scale $1$ to a reference triple junction $\bC$. The distance between the flow $\{V_t^{(m)}\}_t$ and $\bC$, measured in a space-time $L^2$-sense, is a relevant parameter, it will be called {\em excess}, cf. Definition \ref{def:eps-nbd}, and denoted by $\mu^{(m)}$. The other relevant parameter is the (scale invariant) norm of $u^{(m)}$ in $L^q_tL^p_X$, denoted $\|u^{(m)}\|$. Essentially all relevant analytic estimates are obtained with respect to the control quantity $\max\{\mu^{(m)},\|u^{(m)}\|\}$. The first step is performed in Section \ref{sec:graph}: there, we show that, for any choice of $\sigma > 0$, upon assuming that the flow is sufficiently close to $\bC$, and thus that $\max\{\mu^{(m)},\|u^{(m)}\|\}$ is sufficiently small depending on $\sigma$, the regions in the flow at distance at least $\sigma$ from the spine of $\bC$ can be parameterized as normal graphs of functions defined on the three half-planes in $\bC$, with estimates on a suitable parabolic $C^{1,\alpha}$ norm in terms of the control quantity. This Graphicality Theorem \ref{thm:graphical} is obtained by taking advantage of the end-time regularity theorem for Brakke flows with forcing that are close to a multiplicity-one static plane, proved by the authors in \cite{ST_endtime}. This is an extension to the end-time of Brakke's regularity theorem (see also \cite{Brakke_mcf,DPGS,Kasai-Tone,Ton-2,Ton1}). The availability of an end-time regularity theorem is crucial in this step: inspired by \cite{Simon_cylindrical}, we cover the space by toroidal regions of width comparable to the distance from the spine of $\bC$, we then extend them in the time direction appropriately, and finally we apply the end-time regularity theorem at all scales such that the planar excess is small. At the end of Section \ref{sec:graph}, we will have complete geometric control on the part of the flow that is far away from the spine, but the geometry near the spine will still be unresolved.

Section \ref{sec:HS} is the technical core of the paper. Here, we obtain the parabolic counterpart to the fundamental estimates of Simon aimed at proving that the $L^2$ excess does not concentrate near the singular spine. The main result of this section is Theorem \ref{thm:HS}, which establishes the two fundamental estimates \eqref{e:HS1} and \eqref{e:HS2}. The first states that any no-hole point, that is space-time points $(\Xi,\tau)$ with Gaussian density $\Theta(\Xi,\tau) \geq \tfrac32$, must lie in a tubular neighborhood of the spine of the triple junction $\bC$ whose width is bounded above by the control quantity $\max\{\mu,\|u\|\}$: this implies that the graphical representation established in Section \ref{sec:graph} can be pushed at least until this scale. The second states that if $(\Xi,\tau)$ is a no-hole point then the $L^2$ distance in space between the flow and the translated triple junction $\Xi +\bC$, when weighted by the $k$-dimensional backward heat kernel $\rho_{(\Xi,\tau)}$ with pole $(\Xi,\tau)$, must tend to zero as $t \to \tau^-$ at a rate of at least a positive power $(\tau-t)^\kappa$ with respect to the control quantity, uniformly in $t\leq \tau$. This is the parabolic counterpart to the estimate in Simon's \cite[Theorem 3.1(i)]{Simon_cylindrical}, and its formulation was inspired by \cite{tone-wic}. While in Simon's work one obtains an integral estimate for the $L^2$ distance weighted by the singular kernel $|X-\Xi|^{-k+\kappa}$, where $k$ is the dimension and $\kappa \in (0,1)$, the parabolic formulation weights the $L^2$ distance with the time-singular kernel $(\tau-t)^{\kappa} \rho_{(\Xi,\tau)}$, which is of order $\mathrm{O}((\tau-t)^{-\frac{k}{2}+\kappa})$ indeed (this is the correct scaling, as time is effectively two-dimensional in parabolic regularity), and the estimate is uniform in time.

The proof of Theorem \ref{thm:HS} hinges upon gaining control of two key geometric quantities: the deviation from stationarity, measured by $\iint |h|^2$, and the deviation from self-similarity, measured by the Huisken integral $\iint |h - (\nabla\rho_{(\Xi,\tau)})^\perp/\rho_{(\Xi,\tau)}|^2 \rho_{(\Xi,\tau)}$. This is done in Proposition \ref{p:error_t_est}. As Simon's estimates were obtained by suitably testing the stationarity \emph{identity} $\delta V=0$ along cleverly chosen vector fields, here we must test Brakke's \emph{inequality} with appropriate choices of (non-negative, compactly supported) test functions. It is here, to control the error terms coming from the necessary use of cut-off functions for localization inside Brakke's inequality, that we need the structure assumption (A6). 

Once Theorem \ref{thm:HS} is proved, and the graphical representation of the flow has been pushed to distance comparable to excess from the spine of the cone, we pass to the blow-up limit in Section \ref{sec:BU}: after normalization by $\mu^{(m)}$, the graphing functions $f^{(m)}$ are shown to converge to a solution $\tilde f$ of the heat equation on each branch of $\bC$, satisfying compatibility conditions at the spine. Such conditions eventually lead to a Taylor-type expansion for the graphing functions (see Corollary \ref{newcone-decay}) which is then the key towards the proof of the Decay Theorem \ref{thm:decay} and of the Main Theorem \ref{thm:main} in Section \ref{sec:proof-main}.

In Section \ref{lastapp}, we show that the structure assumption (A6) is automatically satisfied in the important cases of flows with an underlying cluster structure and flows of currents mod~3, thus reaching the proof of Corollary \ref{cor-applications}. Finally, the last Section \ref{sec:finalrmk} contains some concluding remarks on future directions of research.

\medskip

\noindent\textbf{Acknowledgements.} 
S.S. acknowledges support from the project PRIN 2022PJ9EFL “Geometric Measure Theory: Structure of Singular Measures, Regularity Theory and Applications in the Calculus of Variations,” funded by the European Union under NextGenerationEU and by the Italian Ministry of University and Research, as well as partial support from the \emph{Gruppo Nazionale per l'Analisi Matematica, la Probabilità e le loro Applicazioni} (INdAM). 
Y.T. was partially supported by JSPS grant 23H00085.

\section{Notation and main results} \label{sec:results}

\subsection{General notation}
The integers $1\leq k<n$ are fixed, and the space-time
coordinate $(X,t)\in \mathbb R^n\times \mathbb R$
is often used, with the variable $t$ referred to as
``time''. The symbol $0_k$ denotes the origin in $\R^k$.  The standard orthonormal basis of $\R^n$ is denoted $e_1,\ldots,e_n$. For any Borel set $A\subset\mathbb R^n$, the symbols $\mathcal L^n(A)$ and $\mathcal H^k(A)$ denote, respectively, the Lebesgue measure and
the $k$-dimensional Hausdorff measure of $A$.  
When $X\in\mathbb R^n$ and $r>0$, $U_r(X)$ and $B_r(X)$ denote the open and closed ball centered at $X$ with radius $r$, respectively, and $U_r$ and $B_r$ are used for $U_r(0)$ and $B_r(0)$, respectively. 
More generally, $U_r^k(X)$ and $B_r^k(X)$ denote the open and closed ball in $\mathbb R^k$ and $\omega_k:=\mathcal L^k(B^k_1)$. 
In $\R^n\times\R$, $P_r(X,t)$ denotes
the open parabolic cylinder $U_r(X)\times(t-r^2,t)$ and $P_r$ is used for $P_r(0,0)$. 

For two subsets $A,B\subset \R^n$, ${\rm dist}_H(A,B)$ is the
Hausdorff distance between $A$ and $B$.

For $0<\alpha<1$ and a function $f:B_R^k\times[-R^2,0]\rightarrow \mathbb R^\ell$ for some $\ell\geq 1$, we define 
\begin{equation*}\begin{split}
    \|f\|_{C^{1,\alpha}}:=&\sup_{(X,t)}\big(R^{-1}|f(X,t)|
    +|\nabla f(X,t)|\big)+\sup_{(X,t)\neq (X',t') }R^{\alpha}
    \frac{|\nabla f(X,t)-\nabla f(X',t')|}{|X-X'|^\alpha+|t-t'|^{\alpha/2}}\\
    &+\sup_{(X,t)\neq (X,t')} R^\alpha\frac{|f(X,t)-f(X,t')|}{|t-t'|^{(1+\alpha)/2}}
    \end{split}
\end{equation*}
throughout the paper, and whenever $\|f\|_{C^{1,\alpha}}$
is used for a different domain, it is understood that it is defined similarly so that it is in this specific invariant form.  

\subsection{$k$-dimensional triple junctions}

We will let $\bC$ denote a $k$-dimensional triple junction in $\R^n$. This is the product $\bC = \hat\bC \times \bS$, where $\bS$ is a $(k-1)$-dimensional linear subspace of $\R^n$, and $\hat\bC$ is the subset of a two-dimensional linear subspace $Z \subset \bS^\perp$ defined, in an orthonormal system of coordinates $(x_1,x_2)$ in $Z$, by
\[
\hat\bC := \{(s,0) \, \colon \, s \geq 0\} \cup \{(-s,\sqrt{3}s) \, \colon \, s \geq 0\} \cup \{(-s,-\sqrt{3}s) \, \colon \, s \geq 0\}\,.
\]

For a triple junction $\bC$ as above, the linear subspace $\bS = \bS(\bC)$ is called the \emph{spine} of $\bC$. Upon a suitable choice of the coordinates in $\R^n$, we may assume that 
\[
\bS =  \{0_{n-k+1}\}  \times \R^{k-1},\qquad {\bf S}^\perp=\R^{n-k+1}\times\{0_{k-1}\} \qquad \mbox{and} \qquad Z = \R^2\times \{0_{n-2}\}  \,.
\]
The three unit vectors $w_1,w_2,w_3$ in $Z$ are defined in these
coordinates as 
\begin{equation}\label{defwi}
w_1=(1,0,0_{n-2})=e_1,\,w_2=(-1/2,\sqrt{3}/2,0_{n-2}),\, w_3=(-1/2,-\sqrt{3}/2,0_{n-2}).
\end{equation}
Define for each $i=1,2,3$
\begin{equation}\label{halfpdef}
{\bf H}_i:=\{sw_i+y: s>0,y\in {\bf S}\}
\quad\mbox{and}\quad {\bf P}_i:=\{sw_i+y: s\in\R, y\in{\bf S}\}.
\end{equation}
Accordingly, we have ${\bf C}={\bf S}\cup \bigcup_{i=1}^3{\bf H}_i$.
To simplify the notation, we often use the same notation
$\bC$, $\bS$ and $Z$ to represent $\bC\times \R$, $\bS\times \R$ and $Z\times \R$, respectively, which are ``static'' 
in space-time $\R^n\times \R$, and similarly for 
${\bf H}_i$ and ${\bf P}_i$. These identifications
should not cause confusion. 
The coordinates of a point $X \in \R^n = \R^{n-k+1} \times \R^{k-1} $ are $X = (x,y)$, and we will often write $x$ and $y$ in place of the more cumbersome $(x,0)$ and $(0,y)$, respectively, so to have $X = x + y$. In particular, $x(X) = \mathbf{p}_{\bS^\perp}(X)$, where $\mathbf{p}_W$ denotes the orthogonal projection onto a subspace $W$, and $\abs{x(X)} = \dist(X,\bS)$ is the distance of the point $X \in \R^n$ from the spine ${\bf S}$ of $\bC$. 

\subsection{Varifolds}
The symbol ${\bf G}(n,k)$ is the Grassmannian of the unoriented $k$-dimensional linear subspace of $\R^n$. We often identify $S\in {\bf G}(n,k)$ with the orthogonal projection map $\R^n\rightarrow S$ and also the matrix that represents the map $S$ in the standard coordinates. The orthogonal projection to the orthogonal complement of $S$ is denoted by $S^\perp$. 
A $k$-dimensional varifold in an open set $U\subset\R^n$ is defined as a positive Radon measure $V$ in
the space $G_k(U):=U\times{\bf G}(n,k)$.
For a comprehensive exposition of varifold theory, see \cite{Allard, Simon}. The set of all $k$-varifolds in $U$ is denoted by ${\bf V}_k(U)$. We let $\|V\|$ and 
$\delta V$ denote the weight
measure and first variation of $V$, respectively. When $\delta V$ is locally bounded as vector measure and
absolutely continuous with respect to $\|V\|$, we let 
$h(\cdot,V)\in L^1_{\rm loc}(\|V\|;\R^n)$ denote the generalized mean curvature vector of $V$, so that $\delta V=-h(\cdot, V)\|V\|$.
A subset $M\subset \R^n$ is countably $k$-rectifiable if it is $\mathcal H^k$-measurable and satisfies
\[
\mathcal H^k(M\setminus
\cup_{i=1}^\infty f_i(\R^k))=0
\]
for some countably many Lipschitz maps $f_i:\R^k\rightarrow\R^n$. Additionally, if $M$ has
locally finite $\mathcal H^k$-measure, $M$ is said to be (locally) $\mathcal H^k$-rectifiable. For such $M$, for $\mathcal H^k$-a.e.
$X\in M$ there exists a unique approximate tangent space
denoted by $T_X M$ or ${\rm Tan}(M,X)$. 
If $M$ is $\mathcal H^k$-rectifiable and $\theta\in L^1_{\mathrm{loc}}(\mathcal H^k\mres_{M})$ is positive and
integer-valued, we let ${\bf var}(M,\theta)$ denote the
varifold  ${\bf var}(M,\theta):=\theta\mathcal H^k\mres_{M}\otimes\delta_{T_{\cdot}M}$, where 
$\delta_{T_X M}$ is the Dirac delta on ${\bf G}(n,k)$ 
with $\delta_{T_X M}(\{T_X M\})=1$. This varifold is called an integral $k$-varifold and we write $V\in {\bf IV}_k(U)$. The function $\theta$ is called multiplicity. In addition, if $\theta=1$ $\mathcal H^k$-a.e.~on $M$, $V$ is said to be a unit-density varifold.
We write ${\rm spt}\|V\|$ for the support of $\|V\|$.
If $X\in {\rm spt}\|V\|$ and there exists $r>0$ such that
$U_r(X)\cap {\rm spt}\|V\|$ is an embedded $k$-dimensional surface of class $C^1$, we write $X\in {\rm reg}\,V$, and
${\rm spt}\|V\|\setminus {\rm reg}\,V$ is denoted by
${\rm sing}\,V$. 

\subsection{Brakke flows with forcing}

For $R>0$ we write $I=(-R^2,0] \subset \R$. 
For every $t \in I$ let $V_t$ be a $k$-varifold in $U_R$ and $u (\cdot,t) \colon U_R \to \R^n$ a $\|V_t\|$-measurable vector field such that the following conditions hold. See
\cite{Ton1} for a comprehensive treatment of Brakke flows in general. 
\begin{itemize}
    \item[(A1)] For a.e. $t \in I$, $V_t \in {\bf IV}_k (U_R)$, the first variation $\delta V_t$ is bounded and absolutely continuous with respect to $\|V_t\|$, so that the generalized mean curvature vector $h(\cdot,V_t)$ exists and $\int_I\int_{U_R}|h(X,V_t)|^2\,d\|V_t\|dt<\infty$; 
    \item[(A2)] there exists $E_1 \in [1,\infty)$ such that for every $t \in I$ and $B_r(X)\subset U_R$, we have
    \begin{equation} \label{e:uniform-area-bound}
        \|V_t\|(B_r (X)) \leq \omega_k\,r^k\,E_1;
\end{equation}
\item[(A3)] let $p \in [2,\infty)$ and $q \in (2,\infty)$ be such that
    \begin{equation} \label{e:the-Holder-exp}
        \alpha := 1 - \frac{k}{p} - \frac{2}{q} > 0\,
    \end{equation}
    and with these $p$ and $q$, $u$ satisfies
\begin{equation} \label{e-bound forcing}
   \|u\|:= R^{\alpha }\left( \int_I \left( \int_{U_R} |u (X,t)|^p \,d\|V_t\|(X) \right)^\frac{q}{p} \,dt \right)^\frac{1}{q} <\infty\,;
\end{equation}
\item[(A4)] for each $\phi \in C^1 (U_R \times I ; \R^+)$ with $\phi(\cdot, t) \in C^1_c (U_R)$ for every $t \in I$, and for any $t_1,t_2 \in I$ with $t_1 < t_2$
\begin{equation} \label{e:Brakke-ineq}
\begin{split}
    \|V_{t_2}\|(\phi(\cdot, t_2)) & - \|V_{t_1}\|(\phi(\cdot, t_1)) \\
    &\leq \int_{t_1}^{t_2} \mathcal{B}(V_t,u(\cdot, t),\phi(\cdot,t))\,dt + \int_{t_1}^{t_2} \int_{U_R} \frac{\partial \phi}{\partial t}(X,t) \,d\|V_t\|(X)\,dt\,,
\end{split}
\end{equation}
where for $V\in {\bf IV}_k(U_R)$,
\begin{equation} \label{Brakke derivative}
    \mathcal{B}(V,u,\phi) := \int_{U_R} (-\phi(X) h(X,V) + \nabla \phi(X)) \cdot (h(X,V)+(u(X))^\perp)  \,d\|V\|(X)\,.
\end{equation}
\end{itemize}
Here, $V_t$ is in ${\bf IV}_k(U_R)$ for a.e.~$t\in I$ by (A1), thus $V_t={\var}(M_t,\theta_t)$ for some locally $\mathcal H^k$-rectifiable set $M_t$ and $\theta_t\in L^1_{\mathrm{loc}}(\mathcal H^k\mres_{M_t})$. 
The symbol $u(X,t)^\perp d\|V_t\|(X)$
is a simplified notation for $(T_X M_t)^\perp(u(X,t))\,d\|V_t\|(X)$. 
The formulation (A4) is an integral formulation of
``normal velocity $= h+u^\perp$'' in the measure-theoretic
manner, originally due to Brakke \cite{Brakke_mcf}
in the case of $u\equiv 0$. From this point onwards, we introduce the notation $\left.\psi(t)\right|_{t=t_1}^{t_2}:=\psi(t_2)-\psi(t_1)$ for any function $\psi$ of time, so that, for instance, the left-hand side of \eqref{e:Brakke-ineq} can be shortened to $\left. \|V_t\|(\phi(\cdot,t))\right|_{t=t_1}^{t_2}$.

The well-known
monotonicity formula due to Huisken \cite{Huisken_mono} 
for MCF can be extended similarly for more general flows with forcing as above.
For $X,\hat X\in\mathbb R^n$ and $s>t$ with $t\in \R$, define
\begin{equation}\label{defgaussian}
    \rho_{(\hat X,s)}(X,t):=\frac{1}{(4\pi(s-t))^{k/2}}\exp\Big(-\frac{|X-\hat X|^2}{4(s-t)}\Big).
\end{equation}
To localize the formula, fix
$\gamma\in (0,1/10)$ and
introduce a cut-off function 
$\eta\in C^\infty_c(U_R)$ such that
$\eta=1$ on $B_{(1-\gamma) R}$
and $\eta=0$ outside of $U_{(1-\gamma/2)R}$. We set $\hat\rho_{(\hat X,s)}(X,t):=\eta(X)\rho_{(\hat X,s)}(X,t)$. 
\begin{proposition} (\cite[Proposition 6.2]{Kasai-Tone}) \label{p:Huisken}
  For $\hat X\in B_{(1-2\gamma)R}$ and 
  $s>t_2>t_1$ with $t_1,t_2\in I$, we have
  \begin{equation}
      \int \hat\rho_{(\hat X,s)}(X,t)\,d\|V_{t}\|(X)\Big|_{t=t_1}^{t_2}
      \leq c\|u\|^2 E_1^{1-\frac{2}{p}}R^{-2\alpha}(t_2-t_1)^\alpha +cE_1 R^{-2}(t_2-t_1)\,,
      \label{Huis}
  \end{equation} 
  where the constant $c$ depends only on $k,p,q$ and $\gamma$. 
\end{proposition}
With \eqref{Huis}, for all $\hat X\in B_{(1-2\gamma)R}$, one can prove
the existence of the Gaussian density
\begin{equation}
    \Theta(\hat X,s):=\lim_{t\rightarrow s-}
    \int \hat\rho_{(\hat X,s)}(X,t)\,d\|V_t\|(X)
\end{equation}
and the standard argument for monotone quantities (see \cite[17.8]{Simon}) shows the
upper semicontinuity of $\Theta$
in $B_{(1-2\gamma)R}\times I$. Ultimately one can prove that $\Theta$ does not depend on the choice of $\gamma$ or
the cut-off function $\eta$ and $\Theta$ is defined in 
$U_R\times I$ as well. Another important 
property of the flow is:
\begin{proposition} (\cite[(3.5)]{Kasai-Tone})
For non-negative $\phi\in C_c^2(U_R)$, there 
exists a constant $c=c(\|\phi\|_{C^2},E_1,\|u\|)$ such that
$\|V_t\|(\phi)-ct$ is a non-increasing function of $t$
on $I$.
\end{proposition}
This shows that $\|V_t\|(\phi)$ is continuous on a 
co-countable set, and we may redefine $\|V_t\|(\phi)$ 
for discontinuous times
so that it is left-continuous on
$I$ while keeping the inequality \eqref{e:Brakke-ineq}.
By density argument, we may re-define $\|V_t\|$ on
countable times so that it is left-continuous as Radon
measures (but not necessarily as varifolds), still having \eqref{e:Brakke-ineq}. 
The replacement is only for a countable set of times, so
all properties (A1)-(A4) are kept, and we additionally have
the left-continuity of $\|V_t\|$. Note that this
eliminates a certain arbitrariness of $\|V_t\|$: for example, we may have
$\|V_t\|=\mathcal H^k\mres_{\bC}$ for $t\in (-1,0)$ 
and $\|V_0\|=0$ which satisfies (A1)-(A4) with $h=u=0$,
but this left-continuous replacement results in the
extension of $\|V_t\|=\mathcal H^k\mres_{\bC}$ to $t=0$ as well, which is more natural. In the following,
we assume that this replacement is always performed.
\begin{itemize}
    \item[(A5)] $\|V_t\|$ is left-continuous with respect to $t$ as Radon measures on $U_R$.
\end{itemize}
Although we phrase (A5) as part of the assumption, in reality it is simply a convention that we use in the present paper.

\subsection{Flow close to triple junction}
We 
are interested in the situation where the flow is close to
$\bC$ in the weak sense of measure in $U_R\times I$.  
For this purpose, we define the following.
\begin{definition}\label{def:eps-nbd}
For $\{V_t\}_{t\in I}$ and $\{u(\cdot,t)\}_{t\in I}$ satisfying (A1)-(A5),
$\nu\in (0,1]$ and $\varepsilon\in(0,1)$,
we write $$(\{V_t\}_{t\in I},\{u(\cdot,t)\}_{t\in I})\in \mathscr N_{\varepsilon,\nu}(U_R\times I)$$ if the following conditions are all satisfied:
\begin{equation}
    \mu := \left(R^{-k-4}  \int_I \int_{U_R} \dist(X,\bC)^2 \, d\|V_t\| (X) \, dt \right)^{\frac12}<\varepsilon,  \label{def:L2-excess}
\end{equation}
\begin{equation}\label{smuterm}
    \|u\|:=R^\alpha\left( \int_I \left( \int_{U_R} |u (X,t)|^p \,d\|V_t\|(X) \right)^\frac{q}{p} \,dt \right)^\frac{1}{q}<\varepsilon,
\end{equation}
\begin{equation} \label{def:mass deficit}
     \|V_{-9R^2/10}\|(B_{R/2})\leq  (3-\nu)\omega_k (R/2)^k\,,
\end{equation}
\begin{equation}\label{non0}
    {\rm spt}\,\|V_0\|\cap B_{\frac{R}{24}}\big(\frac{R}{2}w_i\big)\neq \emptyset\mbox{ for }i=1,2,3\quad\mbox{ (see \eqref{defwi} for $w_i$)}.
\end{equation}
\end{definition}
The first \eqref{def:L2-excess} requires the closeness
of $V_t$ to ${\bf C}$ in the ``$L^2$-excess'', and
this notion is suitable with respect to the topology of
weak convergence of measures in the framework of Brakke flow in general. The second \eqref{smuterm} is automatically
fulfilled for sufficiently small $R$. 
Inequality \eqref{def:mass deficit} requires that
the measure within $B_{R/2}$ at time $t=-9R^2/10$
be strictly less than that of
triple junction with multiplicity $=2$. 
The particular value of $-9R^2/10$
is not important,
and it can be replaced by any number in $(-R^2,0)$ with
a suitable modification on the side of conclusion.
The last \eqref{non0} requires that the measure at $t=0$ is not zero near $Rw_i/2$ ($i=1,2,3$), excluding
the possibility that $V_t$ is trivial. Otherwise, note that $V_t=0$ for all $t\in I$ satisfies \eqref{def:L2-excess}-\eqref{def:mass deficit} trivially. Moreover, we need some non-zero condition for each neighborhood of $Rw_i/2$ at $t=0$,
$i=1,2,3$, otherwise we could have $V_t={\bf var}(\bC,1)$ for $t\in (-R^2,-R^2\delta)$
and $V_t={\bf var}({\bf H}_1\cup {\bf H}_2,1)$ for $t=-R^2\delta$ which subsequently flows and moves little during $t\in(-R^2\delta,0]$ for small $\delta>0$. This flow can have non-zero $\|V_0\|$ around $Rw_1/2$ and $Rw_2/2$, but not around $Rw_3/2$, while the flow is a
Brakke flow satisfing \eqref{def:L2-excess} and \eqref{def:mass deficit}. 
Obviously, if $\varepsilon<\varepsilon'$,
then $\mathscr N_{\varepsilon,\nu}(U_R\times I)\subset\mathscr N_{\varepsilon',\nu}(U_R\times I)$. We record the following simple observation, which follows immediately from the above considerations and from the general theory of weak convergence of Brakke flows.
\begin{remark}\label{rmk:tangents}
    For $\mathscr{V}=\{V_t\}_{t\in I}$ and $\{u(\cdot,t)\}_{t\in I}$ satisfying (A1)-(A5), $\nu\in (0,1]$ and $\varepsilon\in(0,1)$, if a \emph{tangent flow} (see \cite{Ilmanen_mono,Ton1,White_stratification})  to $\mathscr{V}$ at $(0,0)$ is $V'_t \equiv \var(\bC,1)$ for all $t \leq 0$ then there exists $r > 0$ such that $(\{V_t\},\{u(\cdot,t)\})\in \mathscr N_{\varepsilon,\nu}(U_{rR}\times (-(rR)^2,0])$.
\end{remark}

For all sufficiently small $\varepsilon>0$, we have
the following. 
\begin{proposition}\label{prouniden}
    Given any $r\in (0,4/5]$, there exists $\Cl[eps]{onlyunit}=\Cr{onlyunit}(n,k,p,q,E_1,r,\nu)\in(0,1)$
    such that, if $(\{V_t\}_{t\in I},\{u(\cdot,t)\}_{t\in I})\in \mathscr N_{\Cr{onlyunit},\nu}(U_R\times I)$ and 
    (A1)-(A5) are all satisfied, then $V_t$ is a 
    unit-density varifold in $U_{r R}$ for a.e. $t\in (-r R^2,0)$.
\end{proposition}
\begin{proof} Assume without loss of generality that $R=1$.
For a contradiction, assume that there exists a sequence
$(\{V_t^{(m)}\}_{t\in I},\{u^{(m)}(\cdot,t)\}_{t\in I})\in \mathscr N_{\varepsilon^{(m)},\nu}(U_1\times I)$ with $\lim_{m\rightarrow\infty}\varepsilon^{(m)}=0$
such that $V_t^{(m)}$ is not a unit-density varifold in $U_r$ for
$t$ with positive measure on $(-r,0)$. By (A1) this implies that 
there exists $t^{(m)}\in (-r,0)$ such that 
$V_{t^{(m)}}^{(m)}\in {\bf IV}_k(U_1)$ and it has a non-zero portion of multiplicity $\geq 2$. In particular, there is a point 
$X^{(m)}\in U_{r }$ so that the blow-up of $V^{(m)}_{t^{(m)}}$ at $X^{(m)}$ is a plane with multiplicity $\geq 2$. 
Then one can choose $\delta^{(m)}>0$ with $\lim_{m\rightarrow \infty}\delta^{(m)}=0$ such that 
\begin{equation}
    1.9\leq \int\hat\rho_{(X^{(m)},t^{(m)}+\delta^{(m)})}(X,t^{(m)})\,d\|V_{t^{(m)}}^{(m)}\|(X).
\end{equation}
We may assume by choosing a subsequence (denoted by
the same index) that
$X^{(m)}$ and $t^{(m)}$ converge to some $\hat X\in B_{r}$ and $\hat t\in [-r,0]$. 
Then, by \eqref{Huis} and for all sufficiently 
large $m$, we may choose $s>0$ depending only on 
$n,k,p,q,E_1,r$ such that
\begin{equation}\label{limbigm}
    1.8\leq \int \hat\rho_{(X^{(m)},t^{(m)}+\delta^{(m)})}(X,\hat t-s)\,d\|V_{\hat t-s}^{(m)}\|(X).
\end{equation}
Note that $\hat\rho_{(X^{(m)},t^{(m)}+\delta^{(m)})}(X,
\hat t-s)$ as a function of $X$ converges uniformly to
$\hat\rho_{(\hat X,\hat t)}(X,\hat t-s)=\hat\rho_{(\hat X,0)}(X,-s)$ as $m\rightarrow \infty$. By the compactness
theorem of Brakke flow (which also holds with the forcing $\|u^{(m)}\|\rightarrow 0$, see \cite{Ton1}),
there exists a further subsequence (denoted by the same index) and the limit Brakke flow $\{\hat V_t\}_{t\in I}$
with forcing $=0$
such that $\lim_{m\rightarrow\infty}\|V_t^{(m)}\|=\|\hat V_t\|$ as Radon measures on $U_1$ for all $t\in I$. By \eqref{def:L2-excess}, we have
${\rm spt}\,\|\hat V_t\|\subset\bC$, and since $\hat V_t\in {\bf IV}_k(U_1)$ and $h(\cdot,\hat V_t)\in L^2(\|\hat V_t\|)$ for a.e.~$t$, the multiplicity
of $\|\hat V_t\|$ on ${\bf H}_i\cap U_1$ is a constant function
with integer-value, and again by $h(\cdot, \hat V_t)\in L^2(\|\hat V_t\|)$, it is constant on ${\bf C}\cap U_1$. By the property of Brakke flow, one can
check that this multiplicity has to be non-increasing in $t$. The inequality
\eqref{def:mass deficit} shows that $\|\hat V_{-9/10}\|(U_{1/2})\leq (3-\nu)\omega_k/2^k$, thus $\hat V_t$
is either ${\bf var}({\bf C},1)$ or $0$
for $t\in (-9/10,0]$ (note that $\|{\bf var}({\bf C},2)\|(U_{1/2})=3\omega_k/2^k$). If $\hat V_{t'}=0$ for some $t'\in (-9/10,0)$, then again
the property of being a Brakke flow shows that $\hat V_t$ remains $0$ for $t>t'$. However, using \eqref{non0} and
\eqref{Huis}, one can prove a positive lower bound of $\|\hat V_t\|(U_{1})$ for $t$ close to $0$ which depends only on $t$ and $k$, and this leads to a contradiction. This shows that $\hat V_t={\bf var}(\bC,1)$
in $U_1$ for $t\in(-9/10,0]$. Since $\int\hat\rho_{(\hat X,0)}(X,-s)\,d\|\bC\|(X)\leq 1.5$, this would be a contradiction to \eqref{limbigm} for large $m$. 
\end{proof}
\begin{remark}
    In the following, since the value of $\nu\in(0,1]$ is not particularly important, 
    we fix $\nu=1$ and write $\mathscr N_{\varepsilon,1}(U_R\times I)$ as $\mathscr N_{\varepsilon}(U_R\times I)$ unless otherwise stated. 
\end{remark}
\subsection{A further technical assumption on the flow}
\label{futh}
Let $\Cr{onlyunit}\in(0,1)$ be the constant corresponding to
$r=3/4$ in Proposition \ref{prouniden},
and assume that
$(\{V_t\}_{t\in I},\{u(\cdot,t)\}_{t\in I})$ satisfies
(A1)-(A5) in $U_R\times I$. 
\begin{itemize}
    \item[(A6)]
    There exists a constant $\Cl[con]{conslice}>0$ with the following property. Consider an arbitrary parabolic cylinder $P_r(X,t)\subset U_R\times I$, and change space-time coordinates so that $X$ and $t$ are moved to the respective origins and $P_r(X,t)$ is expressed as $P_r$ in the new coordinates. Assume that, after an orthogonal transformation of $\R^n$, we have
    $(\{V_t\}_{t\in I'},\{u(\cdot,t)\}_{t\in I'})\in \mathscr{N}_{\Cr{onlyunit}}(P_r)$
    (where $I'=(-r^2,0]$). Then, by Proposition \ref{prouniden}, 
there exists a $\mathcal H^k$-rectifiable set
$M_t\subset B_{3r/4}$ for 
a.e.~$t\in(-3r^2/4,0)$ such that 
$$V_t={\bf var}(M_t,1)$$ in $B_{3r/4}$.
For this $M_t$, we write the slice of $M_t$ by $\R^{n-k+1}\times\{y\}$ as
\begin{equation}
    M_t^y:=\{x\in\R^{n-k+1}: (x,y)\in M_t\}\mbox{ for }y\in \R^{k-1}
\end{equation}
which is a $\mathcal H^1$-rectifiable for $\mathcal H^{k-1}$-a.e.~$y$. 

Define 
\begin{equation}\label{anasup}
    K:=\sup \frac{{\rm dist}_H(M_t^y\cap 
    \{r/8\leq |x|\leq r/2\}, \hat\bC\cap \{r/8\leq |x|\leq r/2\})}{r},
\end{equation}
where $\sup$ is taken over $y\in B_{r/2}^{k-1}$ and $t\in (-3r^2/4,0)$\footnote{A definite distance away from the spine, note that ${\spt}\,\|V_t\|$ can be expressed as a $C^{1,\alpha}$ graph over $\bC$ with small $C^1$-norm by Theorem \ref{thm:graphical} so that $M_t^y$ is a $C^{1,\alpha}$ curve in $\{r/8\leq |x|\leq r/2\}$}. With this, we assume the validity of the following two inequalities
for $\mathcal H^{k-1}$-a.e.~$y\in B_{r/2}^{k-1}$ and a.e.~$t\in (-3r^2/4,0)$:
    \begin{equation}\label{cotriple}
        \frac{\mathcal H^1(M_t^y\cap B_{r/2}^{n-k+1})}{r}\geq \frac32-\Cr{conslice} K^2
    \end{equation}
    and
    \begin{equation}\label{cotriple2}
       \frac{1}{r^3}\int_{M_t^y\cap B_{r/2}^{n-k+1}}|x|^2\,d\mathcal H^1(x)\geq \frac{1}{r^3}\int_{\hat\bC\cap B_{r/2}^{n-k+1}} |x|^2\,d\mathcal H^1(x)-\Cr{conslice}K^2.
    \end{equation}
\end{itemize}

As discussed in Section \ref{sec:intro}, assumption (A6) is a key structural hypothesis of our main theorem. It provides a quantitative lower bound on the mass and second moment (with respect to the axis $\bS(\bC)$) of 1-dimensional slices. The assumption leverages the geometric behavior of each slice in an annulus \emph{away from the spine} --- a region where the flow is well-controlled by the graphical representation, see Section \ref{sec:graph} --- to enforce a crucial measure-theoretic bound on the entire slice within the disc. This condition expressly prohibits topological degeneracies such as that illustrated in Figure \ref{fig:A6}.
\begin{figure}[ht]
    \centering
\includegraphics{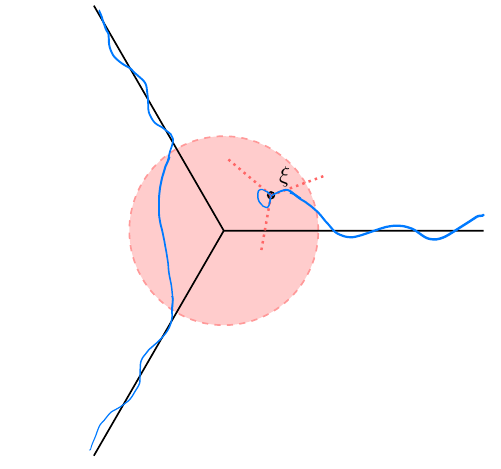}
    \caption{An illustration of the condition (A6). The depicted slice configuration, despite containing a triple junction point, has a shorter length than that of the triple junction in the red disc and is excluded by the assumption.}
    \label{fig:A6}
\end{figure}

Although we shall prove that each slice $M_t^y$ contains a singular point with density at least $3/2$ (see Proposition \ref{p:NH_property}), the parabolic monotonicity formula of Huisken is not sufficient to leverage this local information into the global mass bound (for the slice) we need, in contrast to the stationary theory of Simon. Nonetheless, assumption (A6) is naturally satisfied by flows with additional structure. In Section \ref{lastapp}, we demonstrate that (A6) holds, in codimension $1$, if the flow arises as the boundary of a partition of space, as in the case of the multiphase flows constructed in \cite{ST_canonical} (see also \cite{kim-tone}). Furthermore, it holds in higher codimension for flows representing a mod~3 integral current, where homological constraints guarantee the required connectivity of the slices.

\subsection{The Main Theorem}
With the assumptions stated precisely above, the
main theorem of the present paper is as follows.
\begin{theorem} \label{thm:main}
    For every $n,\,k\in\mathbb N$ with $k<n$, $p \in [2,\infty)$, $q \in (2,\infty)$ satisfying \eqref{e:the-Holder-exp}, $E_1 \in [1,\infty)$ and $\Cr{conslice}\in [0,\infty)$, there exist $\Cl[eps]{e_main} \in (0,1)$ and $\Cl[con]{c_main}\in (1,\infty)$ with the following property. Let $I=(-R^2,0]$ and suppose that $(\{V_t\}_{t\in I},\{u(\cdot,t)\}_{t\in I})\in {\mathscr N}_{\Cr{e_main},1}(U_R\times I)$ and (A1)-(A6) are all satisfied. Then there exists 
    $\xi\in C^{1,\alpha}(U_{R/2}^{k-1}\times[-R^2/2,0];{\bf S}^\perp)$ such that 
    \begin{equation}
        {\rm sing}\, V_t\cap U_{R/2}={\rm graph}\,\xi(\cdot,t)\cap U_{R/2}
        \,\,\mbox{ for all }t\in[-R^2/2,0]\mbox{ and }
    \end{equation}
    \begin{equation}
        \|\xi\|_{C^{1,\alpha}(U_{R/2}^{k-1}\times[-R^2/2,0])}\leq \Cr{c_main}\max\{\mu,\|u\|\}.
    \end{equation}
    Furthermore, for each $i=1,2,3$ and $t\in[-R^2/2,0]$, define
    \begin{equation}
        \Omega_i:=\{(x,y,t)\in ({\bf P}_i\cap U_{R/2})\times[-R^2/2,0]: x\cdot w_i> \xi(y,t)\cdot w_i\}.
    \end{equation}
    Then, for each $i=1,2,3$,
    there exists $f_i\in C^{1,\alpha}(\Omega_i;{\bf P}_i^\perp)$ such that 
    \begin{equation}
        {\rm spt}\,\|V_t\|\cap U_{R/2}=\big({\rm graph}\,\xi(\cdot,t)\cup\,\cup_{i=1}^3{\rm graph}\,f_i(\cdot,t)\big)\cap U_{R/2}
    \end{equation}
    for $t\in[-R^2/2,0]$ and 
    \begin{equation}\label{e:main-thm-est}
        \|f_i\|_{C^{1,\alpha}(\Omega_i)}\leq \Cr{c_main}\max\{\mu,\|u\|\}.
    \end{equation}
\end{theorem}

The theorem states that if a flow satisfying (A1)-(A6) is sufficiently close, is a parabolic cylinder $P_R$ and in the topology defined by the neighborhoods $\mathcal N_\varepsilon$, to the static $\bC$ then in $P_{R/2}$ it is a $C^{1,\alpha}$ deformation of $\bC$. A direct consequence of the proof is the uniqueness of the tangent flow at each singular point, namely at each point $(X,t)$ on the graph of the map $\xi$. The (static) tangent cone is of the form $O_{X,t}(\mathbf{C})$, where the rotation $O_{X,t} \in \mathrm{O}(n)$ varies in a $C^\alpha$ H\"older continuous manner in space-time. Away from this singular set, on each domain $\Omega_i$, the graphing function $f_i$ possesses higher regularity. An application of Motegi's results \cite{Motegi} shows that $f_i$ has weak derivatives $\partial_t f_i, \nabla^2 f_i \in L^2_{\mathrm{loc}}(\Omega_i)$ and satisfies the mean curvature flow equation with forcing as a strong solution.

Higher regularity depends on the forcing term. If $u$ is $C^{\ell,\alpha}$, standard parabolic theory implies $f_i \in C^{\ell+2,\alpha}_{\mathrm{loc}}(\Omega_i)$ \cite{Ton-2}; in particular, if $u \equiv 0$, the sheets $f_i$ are $C^\infty$. 

Achieving such an optimal regularity \emph{up to the singular set} for the flow is a major open problem. In the stationary case, the work of Kinderlehrer, Nirenberg, and Spruck \cite{KNS78} on free-boundary problems implies that the sheets $f_i$ are real-analytic up to their common boundary $\mathrm{graph}\,\xi$ (the free boundary in this problem), and that the map $\xi$ is real-analytic as well. On the other hand, the techniques of \cite{KNS78} rely on the divergence form structure of the minimal surfaces equation, a structure that is absent in the mean curvature flow system. Therefore, it is not known if the sheets $f_i$ are even $C^2$ up to the boundary $\overline{\Omega_i} \cap \mathrm{graph}\,\xi$, even for smooth $u$. While a recent result by Krummel \cite{Krummel} establishes that $C^{2,\alpha}$ regularity of the sheets up to the boundary would imply smoothness (for $u \equiv 0$), bridging the gap between the $C^{1,\alpha}$ regularity established here and the required $C^{2,\alpha}$ condition remains a significant challenge.

We conclude this section by showing how Theorem \ref{thm:main} (and Corollary \ref{cor-applications}) imply Theorem \ref{thm:structure-intro}.

\begin{proof}[Proof of Theorem \ref{thm:structure-intro}]
    By assumption, $\mathscr{V}=\{V_t\}_t$ satisfies (A1)-(A4), and we can assume that the convention (A5) is enforced. Let $(X,t)$ be a point in $\Sing\,\mathscr{V}$ where a tangent flow is a static stationary cone splitting a Euclidean factor $\R^{k-1}$. By the assumption \eqref{ipotesi-H}, the density of such cone at the origin may only be $1$ or $\tfrac32$: in the first case, the cone is a multiplicity-one plane, and this cannot be the case because then $(X,t)$ would be regular by the regularity theorem in \cite{ST_endtime}; in the second case, the cone is a multiplicity-one triple junction. Hence, by Remark \ref{rmk:tangents} and modulo a suitable translation and rotation, the flow belongs to $\mathcal N_\varepsilon$ for $\bC$ at some scale. Since $\mathscr{V}$ has multi-phase cluster structure or is a flow of currents mod~3, by Corollary \ref{cor-applications} (A6) is satisfied. Hence, the flow is, locally at $(X,t)$, a $C^{1,\alpha}$ deformation of $\bC$, and the singular set is a $C^{1,\alpha}$ graph, by Theorem \ref{thm:main}. We define $\mathcal R$ to be the set of these points. All other points in $\Sing\,\mathscr{V}$, which include points where the flow has a static tangent flow which is a stationary cone with strictly less than $k-1$ spatial symmetries as well as quasi-static or shrinking tangent flows, are in $\mathcal S$. By White's stratification theorem \cite{White_stratification}, $\mathcal S$ has parabolic Hausdorff dimension $\dim_{\mathcal P}(\mathcal S) \leq k$. 
\end{proof}

\section{Graphical parametrization} \label{sec:graph}

The main theorem of this section is Theorem \ref{thm:graphical}: it establishes the crucial geometric fact that, when the flow is $L^2$-close to $\bC$, namely when it belongs to $\mathscr N_{\varepsilon}(U_R\times I)$, then it is $C^{1,\alpha}$-close to $\bC$, and in fact a graph over $\bC$ with small $C^{1,\alpha}$ norm, \emph{outside of a small tubular neighborhood of the spine $\bS(\bC)$}. This is a parabolic analogue to \cite[Lemma 2.6]{Simon_cylindrical}, and the idea of the proof is similar. The main technical tool is the end-time $\varepsilon$-regularity theorem for unit density $k$-dimensional Brakke-type flows close to a static $k$-dimensional plane proved by the authors in \cite[Theorem 2.2]{ST_endtime}, and recorded here as Proposition \ref{e-reg}. 

\begin{proposition}\label{e-reg}
    There exist $\Cl[eps]{e_den}\in (0,1)$ and $\Cl[con]{linestimate}\in (1,\infty)$
    depending only on $n,k,p,q,E_1$ with the following property. Assume that $\{V_t\}_{t \in I}$ and $\{u(\cdot,t)\}_{t \in I}$ defined in $U_R$ satisfy (A1)-(A5). Suppose furthermore that 
    \[
    \|V_0\|(B_{R/2})>0,\hspace{.5cm}\,\,\,(R/2)^{-k}\|V_{-9R^2/10}\|(B_{R/2})\leq  \omega_{k}+\Cr{e_den}
    \]
    and that, for some $T\in{\bf G}(n,k)$,
    \[
    {\mathscr E}:=\Big(R^{-k-4}\int_{I} \int_{U_R} {\rm dist}(X,T)^2 \, d\|V_t\|(X) \,dt\Big)^{\frac12} < \Cr{e_den}\,\,\, \mbox{ and }\|u\|<\Cr{e_den}.
    \]
    Then, setting $\tilde D:=\left(B_{R/2} \cap T \right) \times[-R^2/2,0]$ there is
    a function $f\in C^{1,\alpha}(\tilde D;T^\perp)$ 
such that
\begin{equation} \label{udef5-12}
    {\rm spt}\|V_t\|\cap T^{-1}(B_{R/2}\cap T)\cap B_{3R/4}={\rm graph}\,f(\cdot,t)
    \mbox{ for all }t\in[-R^2/2,0],
\end{equation}
\begin{equation}\label{udef5}
\|f\|_{C^{1,\alpha}(\tilde D)}\leq \Cr{linestimate} \, \max\{\mathscr E,\|u\|\}.
\end{equation}
\end{proposition}

With a slight abuse of notation, we
use the following notation for the
space-time measure defined as
$$
d\|\mathscr V\|:=d\|V_t\| dt.
$$
\begin{theorem} \label{thm:graphical}

For every $\beta \in (0,1)$ and $\sigma \in (0,\sfrac14)$ there exists $\Cl[eps]{eps-tau} = \Cr{eps-tau}(k,n,p,q,E_1,\beta,\sigma) \in(0,1)$ with the following property. Assume that $(\{V_t\},\{u(\cdot,t)\})\in\mathscr N_{\varepsilon}(U_R\times I)$ satisfies
(A1)-(A5) with $\varepsilon \leq \Cr{eps-tau}$. Then, there are a
relatively open set 
\[
U\subset \bC \cap (U_R\times(-R^2,0))
\]
such that
\begin{equation}\label{three}
    (x,y,t) \in U \implies (\tilde x,y,t) \in U \qquad \mbox{whenever $(\tilde x, y) \in \bC$ with $|\tilde x| = |x|$}\,, 
\end{equation}
\begin{equation}\label{three2}
     U \supset \{(x,y,t) \in \bC \cap (B_{R/2}\times(-R^2/2,0))  \, \colon \, |x| > \sigma R\}\,,
\end{equation}
and a function $f \in C^{1,\alpha} (U ; \bC^\perp)$ with the property that
    \begin{align} \label{e:bound}
        &\sup_{(X,t)\in U}\big( |x|^{-1} \, |f(X,t)| + |\nabla f(X,t)|\big) \leq \beta\,, 
        \\ \label{parametrization}
        & \spt\, \|\mathscr V\| \cap (B_{R/2}\times(-R^2/2,0))\cap \{(x,y,t) \, \colon \, |x| > \sigma R\} \subset {\rm graph} \, f \subset \spt \|\mathscr V\| \,, \\ 
        \label{para1}
        & \int_{(B_{R/2}\times(-R^2/2,0)) \setminus {\rm graph } f} |x|^2 \, d\|\mathscr V\|(X,t)\leq \Cl[con]{nonsig} \mu^2 R^{k+4},
        \\
        \label{parC}
        & \int_{(B_{R/2}\times(-R^2/2,0))\cap \bC\setminus U}|x|^2\,d\mathcal H^k(X)dt\leq \Cr{nonsig} \mu^2 R^{k+4},\\
        \label{para1-ex}
        & \int_{U} |x|^2 \, |\nabla f|^2 \, d\Ha^k(X)dt \leq \Cr{nonsig} \, \max\{\mu,\|u\|\}^2 \,R^{k+4}
    \end{align}
for $\Cr{nonsig}=\Cr{nonsig}(k,n,p,q,E_1,\beta)$ which does not depend on $\sigma$. Moreover, for a constant $\Cl[con]{nonsing1}=\Cr{nonsing1}(k,n,p,q,E_1,\beta,\sigma)$, we have
\begin{equation}\label{expara1}
    \|f\|_{C^{1,\alpha}(U\cap \{|x|>\sigma R\})}\leq \Cr{nonsing1}\max\{\mu,\|u\|\}.
\end{equation}
\end{theorem}

\begin{proof}
The proof is a multi-step argument based on a covering argument and a dichotomy. We first define a set $U$ where the flow is known to be graphical as a consequence of Proposition \ref{e-reg} (Steps 1-2). The core of the argument is a dichotomy established in Step 3: if the graphical representation fails at a certain scale, then the $L^2$-excess at that scale must be large. Finally, we use a covering argument (Steps 4-6) to show that the total volume of the 'non-graphical' region must be small, as the total $L^2$-excess over the entire domain is controlled.

\smallskip

    Since the statement is scale invariant, we may
    assume $R=2$. For every $\zeta \in \R^{k-1}$ with $|\zeta| < 9/8$, $s \in (-9/4,0]$, $\rho\in (0,9/8]$ and $\kappa\in(0,1]$, consider the region
    \begin{equation}\begin{split}
    T_{\rho, \kappa} (\zeta,s) := \Big\lbrace (x,y,t) &\in \R^{n-k+1} \times \R^{k-1} \times \R \, \colon \\ & (|x|-\rho)^2 + |y-\zeta|^2 < \frac{\kappa^2\rho^2}{4} \,, \quad s-\frac{\kappa^2\rho^2}{4} < t < s   \Big\rbrace\,,
    \end{split}
    \end{equation}
    see Figure \ref{fig:torus_geom}.
    \begin{figure}[ht]
\centering
\begin{tikzpicture}[
    scale=1.3,
    axis/.style={dashed, gray, -{Stealth[length=2mm]}},
    cone/.style={blue!60!black, thick},
    annulus/.style={fill=red!30}
]
    \pgfmathsetmacro{\rhovalue}{2.5}
    \pgfmathsetmacro{\kappavalue}{0.4}
    \pgfmathsetmacro{\innerR}{\rhovalue * (1 - \kappavalue/2)}
    \pgfmathsetmacro{\outerR}{\rhovalue * (1 + \kappavalue/2)}

    \draw[axis] (-3.5,0) -- (3.5,0);
    \draw[axis] (0,-3.5) -- (0,3.5);

    \fill[annulus, even odd rule] (0,0) circle (\outerR) circle (\innerR);
    \draw[dashed, red] (0,0) circle (\innerR);
    \draw[dashed, red] (0,0) circle (\outerR);

    \draw[cone] (0,0) -- (0:3.2) node[anchor=west]{$\mathbf{H}_1$};
    \draw[cone] (0,0) -- (120:3.2) node[anchor=east]{$\mathbf{H}_2$};
    \draw[cone] (0,0) -- (240:3.2) node[anchor=east]{$\mathbf{H}_3$};

    \fill (0,0) circle (1.5pt);
    \node[anchor=north west] at (0,0) {$\zeta$};

    \draw[<->, shorten <=1pt, shorten >=1pt] (135:\innerR) -- (135:\outerR) node[midway, fill=white, inner sep=1pt] {$\kappa\rho$};
    \draw[->, shorten >=1pt] (0,0) -- (30:\innerR) node[midway, above, sloped] {$\rho(1-\frac{\kappa}{2})$};
    
    \node[red!80!black] at (20:{\rhovalue}) {$T_{\rho, \kappa}(\zeta,s)$};

\end{tikzpicture}
\caption{The figure illustrates the 2D annulus obtained by slicing the space-time toroidal cylinder $T_{\rho, \kappa}(\zeta,s)$ at a fixed spine location $y=\zeta$ and a fixed time $t$. Notice that the width of each toroidal region and its distance from the spine are comparable quantities, as in classical Whitney-type domain decompositions.}
\label{fig:torus_geom}
\end{figure}
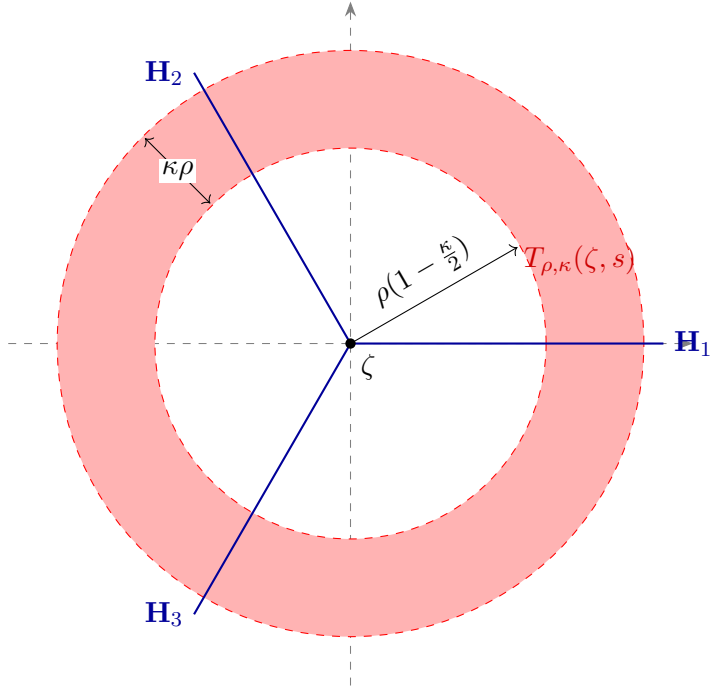

     Next, for each $j \in \{1,2,3\}$ let $$T^j_{\rho, \kappa} (\zeta,s) := T_{\rho, \kappa} (\zeta,s) \cap \mathbf{H}_j,$$ and notice that $T^j_{\rho, \kappa} (\zeta,s) = P_{\kappa \rho/2}(\zeta + \rho w_j,s) \cap\mathbf{H}_j$.
    
    We define the subset $\tilde U\subset\bC$ to be the union of all $T_{|x|,1/4}(y,s)\cap \bC$ 
    over all $(x,y,s)\in \bC \cap (U_{\sfrac98} \times (-9/4,0))$ such that there exists $f^j\in C^{1,\alpha}(T_{|x|,3/8}^j(y,s);{\mathbf H}_j^\perp)$ ($j=1,2,3$) with
    \begin{equation}\label{gr1}
        {\rm spt}\,\|\mathscr V\|\cap T_{|x|,5/16}(y,s)\subset\cup_{j=1}^3 {\rm graph}\, f^j\subset {\rm spt}\,\|\mathscr V\|
    \end{equation}
    and 
\begin{equation}\label{e:gra2}
    |x|^{-1} \sup_{ T^j_{|x|,3/8} (y,s) }|f^j|+\sup_{T^j_{|x|,3/8} (y,s) }
    |\nabla f^j|\leq \beta
\end{equation}
for each $j=1,2,3$. This $f^j$ depends on the choice
of $(x,y,s)$ initially, but since the graph of $f^j$ represents 
${\rm spt}\|\mathscr V\|$ as in \eqref{gr1}, it is uniquely
defined on each $T_{|x|,1/4}^j(y,s)\cap \bC$ and hence
on $\tilde U$ as well. We then define 
\begin{equation}\label{tildeU}
    U=\tilde U\cap (U_1\times(-2,0)).
\end{equation}
From the way $U$ is defined, 
note that \eqref{three} and \eqref{e:bound}
are satisfied already. 
\newline
{\bf Step 1}. There exists $\Cl[eps]{eps-lip}
=\Cr{eps-lip}(k,n,p,q,E_1,\beta)\in(0,1)$ such that
the following holds. Suppose
\begin{itemize}
\item[(1)] 
\begin{equation} (\{V_t\},\{u(\cdot,t)\})\in \mathscr N_{\Cr{eps-lip}}(U_2\times(-4,0]),
\end{equation}
\item[(2)] for $(x_0,y_0,s_0)\in \bC$ with $(x_0,y_0)\in U_{\sfrac98}$ and $s_0\in (-9/4,0]$,
\begin{equation}\label{e:small_excess_tori}
    |x_0|^{-(k+4)}\int_{T_{|x_0|,1}(y_0,s_0)} {\rm dist}(X,\bC)^2\,d\|\mathscr V\|(X,t)<\Cr{eps-lip}, 
\end{equation}
\item[(3)] 
for all $(\tilde x,y_0)\in \bC$ with $|\tilde x|=|x_0|$,
\begin{equation}\label{spt-int}
    {\rm spt}\|\mathscr V\|\cap (B_{|x_0|/10}(\tilde x,y_0)
\times\{s_0\})\neq \emptyset.
\end{equation}
\end{itemize}
Then $T_{|x_0|,1/4}(y_0,s_0)\cap \bC\subset \tilde U$. 
\newline
{\bf Proof of Step 1}. 
Assume that $(x_0,y_0,s_0)$ satisfies
(1) and (2) with $\Cr{eps-lip}$ 
to be determined and write $\rho:=|x_0|$. First we claim that 
\begin{equation}\label{e:small_ht}
    {\rm spt}\|\mathscr V\|\cap T_{\rho,7/8}(y_0,s_0)\setminus\cup_{j=1}^3\{(X,t):|\mathbf H_j^\perp(X)|\geq \rho/20\}=\emptyset
\end{equation}
if $\Cr{eps-lip}$ is sufficiently small depending only on $k,n,p,q,E_1$. 
Indeed, suppose by contradiction that it 
contains a point $(X_1,t_1)$, so in particular 
${\rm dist}(X_1,\bC)\geq \rho/20$. Then, by \cite[Corollary 6.3]{Kasai-Tone}, there exist small constants $\Cl[con]{loden1}, \Cl[con]{loden2}> 0$ depending only on $k,n,p,q,E_1$ 
such that $\|V_t\|(B_{\sfrac{\rho}{40}}(X_1)) \geq \Cr{loden1} \,\rho^k$ for every $t \in (t_1 - 2\Cr{loden2}\rho^2, t_1 - \Cr{loden2}\rho^2)$. Since ${\rm dist}(B_{\sfrac{\rho}{40}}(X_1),\bC)\geq \rho/40$, this contradicts \eqref{e:small_excess_tori} for a suitable choice of $\Cr{eps-lip}$ depending only on the stated constants.

Next, we wish to apply the $\varepsilon$-regularity theorem, Proposition \ref{e-reg} in $P_{3\rho/8}(\tilde x,y_0,s_0)$ with $|\tilde x|=\rho$ where we now assume (3) in addition. Suppose that $(\tilde x,y_0)\in \mathbf H_j$. Note that $P_{3\rho/8}(\tilde x,y_0,s_0)\subset T_{\rho,1}(y_0,s_0)$,
so in particular we have by \eqref{e:small_excess_tori}
\begin{equation}
    \rho^{-(k+4)}\int_{P_{3\rho/8}(\tilde x,y_0,s_0)}
    {\rm dist}(X,\bC)^2\,d\|\mathscr V\|(X,t)<\Cr{eps-lip},
\end{equation}
and by \eqref{e:small_ht},
\begin{equation}
    {\rm spt}\|\mathscr V\|\cap P_{3\rho/8}(\tilde x,y_0,s_0)\subset\{(X,t):|{\mathbf H}_j^{\perp} (X)|<\rho/20\},
\end{equation}
and \eqref{spt-int} implies
\begin{equation}\label{non-empty}
{\rm spt}\|\mathscr V\|\cap (B_{\rho/10}(\tilde x,y_0)\times\{s_0\})\neq \emptyset.
\end{equation}
We now only need to have 
\begin{equation}\label{int-flat}
    (3\rho/16)^{-k}\|V_{s_0-81\rho^2/640}\|(B_{3\rho/16}(\tilde x, y_0))\leq  \omega_k + \Cr{e_den}
\end{equation}
to apply the $\varepsilon$-regularity 
theorem. This can be achieved by 
compactness argument as in the proof of 
Proposition \ref{prouniden}. 
Thus we may apply the
$\varepsilon$-regularity theorem and obtain 
$f_j$ defined on $(B_{3\rho/16}(\tilde x,y_0)\cap\mathbf H_j)\times [s_0-(3\rho/8)^2/2,s_0]$, and thus in particular on  $T^j_{\rho,3/8}(y_0,s_0)$, satisfying \eqref{gr1}, and
for sufficiently small $\Cr{eps-lip}$ depending
also on $\beta$, \eqref{e:gra2}. This shows
that $T_{\rho,1/4}^j(y_0,s_0)\subset \tilde U$ for each
$j=1,2,3$ and ends the
proof of Step 1.

{\bf Step 2}.
For any $\sigma\in(0,1/4)$, there exists
$\Cr{eps-tau}=\Cr{eps-tau}(k,n,p,q,E_1,\beta,\sigma)\in(0,\Cr{eps-lip})$ such that $\mu<\Cr{eps-tau}$ implies 
\begin{equation}\label{Ucov}
   \{(x,y,s) \in \bC : (x,y) \in U_{\sfrac98},\,s\in(-9/4,0),\,|x| \geq \sigma\}\subset\tilde U,
\end{equation}
and which also implies
\eqref{three2}. We also have \eqref{expara1} for $\Cr{nonsing1}=\Cr{nonsing1}(k,n,p,q,E_1,\beta,\sigma)$.  
\newline
{\bf Proof of Step 2}. 
By choosing $\Cr{eps-tau}>0$ small so
that $\sigma^{-(k+4)}\Cr{eps-tau} <
\Cr{eps-lip}$, for $\mu<\Cr{eps-tau}$, we can make sure that
\eqref{e:small_excess_tori} is 
satisfied for all point on 
$A$.
Then,
the only condition to be checked for the application of 
Step 1 is \eqref{spt-int}.
Since we have ${\rm spt}\|
\mathscr V\|\cap (B_{1/12}(w_j)\times\{0\})
\neq \emptyset$ for each $j=1,2,3$ by \eqref{non0}, 
(1)-(3) of Step 1 are satisfied in a neighborhood of $w_j$. 
For
$(x,y,s)$ close to $\bC\cap T_{1,1/2}(0,0)$, note 
that \eqref{spt-int} is satisfied.
We may repeat this argument until all points with $|x|\geq \sigma$ are covered
by $\tilde U$.
This proves \eqref{Ucov}, and 
combined with \eqref{gr1} and \eqref{tildeU}, we also showed
\eqref{parametrization}. Due to the
estimate \eqref{udef5}, we immediately
obtain \eqref{expara1} on $U\cap \{|x|>\sigma\}$. 

{\bf Step 3}. 
For each $y\in\R^{k-1}$ with $|y|<9/8$ and $s\in (-9/4,0)$, let
\begin{equation}\label{Ucov3}
g(y,s):=\inf\{r>0\,:\,(x,y,s)\in \tilde U
\mbox{ for all }(x,y)\in \bC \cap U_{\sfrac98} \mbox{ with }r<|x|\}.
\end{equation}
Then, whenever $g(y,s)>0$, we have
\begin{equation}\label{Ucov2}
    (g(y,s))^{-(k+4)}\int_{T_{g(y,s),1}(y,s)}
    {\rm dist}(X,\bC)^2\,d\|\mathscr V\|(X,t)\geq \Cr{eps-lip}. 
\end{equation}
{\bf Proof of Step 3}. Because of \eqref{Ucov},
note that $g(y,s)\leq \sigma$. Suppose that $g(y,s)>0$ and \eqref{Ucov2} does not hold. 
Let $\tilde x$ be such that $(\tilde x,y,s)\in \bC\cap (U_{\sfrac98} \times (-9/4,0))$
with $|\tilde x|=g(y,s)$. By the definition \eqref{Ucov3}, for all sufficiently small $\delta>0$,
$((1+\delta)\tilde x,y,s)\in \tilde U$. Since $\tilde U$ is open, 
one can argue that
${\rm spt}\|\mathscr V\|\cap (B_{|\tilde x|/10}(\tilde x,y)\times\{s+\delta\})\neq \emptyset$
for all sufficiently small $\delta>0$.
By the continuity of integral, for all sufficiently
small $\delta>0$, the negation of \eqref{Ucov2} gives
\begin{equation*}
    |\tilde x|^{-(k+4)} \int_{T_{|\tilde x|,1}(y,s+\delta)}
    {\rm dist}(X,\bC)^2\,d\|\mathscr V\|(X,t)<\Cr{eps-lip}.
\end{equation*}
Then the conclusion of Step 1 shows that $T_{|\tilde x|,1/4}(y,s+\delta)\cap\bC\subset \tilde U$, and implies $(\tilde x, y,s)\in \tilde U$. This is a 
contradiction to \eqref{Ucov3}. This ends the proof 
of Step 3. 

{\bf Step 4}. 
There exists a (at most countable) set of points $\{(x_i,y_i,s_i)\}_{i\in \Lambda}
\subset \bC\cap (B_1\times(-2,0])$ such that 
the number of intersection of $T_{|x_i|,1}(y_i,s_i)$ is bounded by a constant $\Cl[con]{c-int}=\Cr{c-int}(n,k)$ and with 
\begin{equation}\label{pcover}
    (B_1\times[-2,0))\setminus \bigcup_{i\in \Lambda} T_{|x_i|,1/4}
    (y_i,s_i)\subset \bigcup_{|y|<9/8,\,s\in (-9/4,0)}\tilde P_{3g(y,s)}(0,y,s)
\end{equation}
and 
\begin{equation}\label{pcos}
\bC\cap T_{|x_i|,1}(y_i,s_i)\subset \tilde U
\end{equation}
for each $i\in \Lambda$. Here we use the symbol \[\tilde P_r(x_0,y_0,s_0):=\{(x,y,s):|x-x_0|^2+|y-y_0|^2<r^2, |s-s_0|<r^2\}.\]
\newline
{\bf Proof of Step 4}. 
We may choose a set of points $\{(x_i,y_i,s_i)\}_{i\in
\mathbb N}\subset\bC\cap (B_{1}\times(-2,0])$ such that $(B_1\times[-2,0))\subset \cup_{i\in\mathbb N} T_{|x_i|,1/4}
    (y_i,s_i)$
    and the number of intersection of 
    $T_{|x_i|,1}(y_i,s_i)$ is bounded by a 
    constant $\Cr{c-int}$ depending only on $n$ and $k$. In essence, this can be done by
    considering first the covering of 
    ${\bf H}_1\cap B_{1}$ by balls of 
    type ${\bf H}_1\cap B_{|x|/8}(x,y)$ with $(x,y)\in{\bf H}_1\cap B_{1}$ so that the 
    intersection number of ${\bf H}_1\cap B_{|x|/2}(x,y)$ is bounded by a constant,
    and then extend the covering in the time
    direction appropriately. Without loss of 
    generality, we may assume that $(x_i,y_i)\in {\bf H}_1$ and $T_{|x_i|,1}(y_i,s_i)\subset U_{\sfrac98}\times(-9/4,0)$. With these points fixed, we 
    define $\Lambda\subset\mathbb N$ as follows:
    $i\in \Lambda$ if we have $g(y,s)<|x|$ for all $(x,y,s)\in \bC\cap  T_{|x_i|,1}(y_i,s_i)$. 
    The definition implies that $\bC\cap 
T_{|x_i|,1}(y_i,s_i)\subset \tilde U$ if $i\in \Lambda$. Note that $(B_1\times[-2,0))\subset \cup_{i\in\mathbb N}
T_{|x_i|,1/4}(y_i,s_i)$, thus, if $(x,y,s)\in
(B_1\times[-2,0))\setminus \bigcup_{i\in \Lambda} T_{|x_i|,1/4}
    (y_i,s_i)$, then there exists $i\in \mathbb N\setminus \Lambda$ such that $(x,y,s)\in 
    T_{|x_i|,1/4}(y_i,s_i)$,
    and $i\notin \Lambda$ implies
    that there exists
    $(\tilde x,\tilde y,\tilde s)
    \in T_{|x_i|,1}(y_i,s_i)$
    such that $|\tilde x|\leq g(\tilde y,\tilde s)$.
    These inclusions imply
    \begin{equation*}
    \begin{split}
      &  |x|<\frac{9 |x_i|}{8},\,\,
        |y-y_i|<\frac{|x_i|}{8},\,\,s\in (s_i-\frac{|x_i|^2}{64},s_i),\\
        &\frac{|x_i|}{2}<|\tilde x|<\frac{3 |x_i|}{2},\,\,|\tilde y-y_i|<\frac{|x_i|}{2},\,\,
        \tilde s\in (s_i-\frac{|x_i|^2}{4},s_i),
       \end{split} 
    \end{equation*}
    and consequently, $|s-\tilde s|<|x_i|^2/4<|\tilde x|^2\leq (g(\tilde y,\tilde s))^2$ and
    \begin{equation*}
        |x|^2+|y-\tilde y|^2
        <\frac{81|x_i|^2}{64}
        +\frac{25|x_i|^2}{64}<
        \frac{53|\tilde x|^2}{8}\leq 9 (g(\tilde y,\tilde s))^2.
    \end{equation*}
    In particular, we have
    $(x,y,s)\in \tilde P_{3g(\tilde y,\tilde s)}(0,\tilde y,\tilde s)$, and
    this proves \eqref{pcover}.
    This ends the proof of Step 4. 

    {\bf Step 5}. We have \eqref{para1}
    and \eqref{parC}.
    \newline
    {\bf Proof of Step 5}.
    If $(x,y,t)\in{\rm spt}\,\|\mathscr V\|\cap  T_{|x_i|,1/4}(y_i,s_i)$ for some $i\in \Lambda$, 
    since $\bC\cap T_{|x_i|,1}(y_i,s_i)\subset \tilde U$ by 
    Step 4, $(x,y,t)\in{\rm graph}\, f$. By \eqref{pcover}, thus we only need to estimate the 
    integral over the region 
    on the right-hand side of \eqref{pcover}.
    By Vitali's covering lemma, there exists a
    set of points $\{(y_j,s_j)\}$ with $|y_j|<9/8$ and
    $s_j\in (-9/4,0]$ such that $\{\tilde P_{3g(y_j,s_j)}(0,y_j,s_j)\}$ is pairwise disjoint and 
    \begin{equation} \label{pcover1}
      \bigcup_{|y|<9/8,\,s\in (-9/4,0]} \overline{\tilde P_{3g(y,s)}(0,y,s)}\subset \bigcup_{j} \overline{\tilde P_{15g(y_j,s_j)}(0,y_j,s_j) }
      \subset \bigcup_{j} \tilde P_{16g(y_j,s_j)}(0,y_j,s_j) .
    \end{equation}
    Since $T_{g(y,s),1}(y,s)\subset P_{3g(y,s)}(0,y,s)$, we have as a consequence of Step 3 that for each $j$
    \begin{equation}\label{pcover2}
        \Cr{eps-lip}(g(y_j,s_j))^{k+4}\leq 
        \int_{P_{3g(y_j,s_j)}(0,y_j,s_j)}{\rm dist}(X,\bC)^2\,d\|\mathscr V\|(X,t)\,. 
    \end{equation}
    By \eqref{e:uniform-area-bound}, we also have
    \begin{equation}\label{pcover3}
        \int_{\tilde P_{16g(y_j,s_j)}(0,y_j,s_j)} |x|^2
        \,d\|\mathscr V\|(X,t)\leq 2\omega_k E_1 (16g(y_j,s_j))^{k+4}.
    \end{equation}
    Thus, by combining \eqref{pcover1}-\eqref{pcover3}, we obtain
    \begin{equation}
    \begin{split}
        \int_{\cup \tilde P_{3g(y,s)}(0,y,s)} |x|^2\,d\|\mathscr V\| &\leq 2\omega_k E_1 \sum_j 
        (16g(y_j,s_j))^{k+4} \\ &\leq 2\omega_k(16)^{k+4}E_1\Cr{eps-lip}^{-1} 
        \sum_{j}\int_{P_{3g(y_j,s_j)}(0,y_j,s_j)}
        {\rm dist}(X,\bC)^2\,d\|\mathscr V\|\\
        &\leq 2\omega_k(16)^{k+4}E_1\Cr{eps-lip}^{-1} \int_{P_2}
        {\rm dist}(X,\bC)^2\,d\|\mathscr V\|,
        \end{split}\label{pcosub}
    \end{equation}
    where we used that the $\{P_{3g(y_j,s_j)}(0,y_j,s_j)\}_j$ are disjoint. This proves \eqref{para1} with 
$\Cr{nonsig}=2\omega_k(16)^{k+4}E_1\Cr{eps-lip}^{-1}$.
    The proof for \eqref{parC} proceeds similarly. If $(x,y,t)\in (B_1\times[-2,0))\cap \bC\setminus U=(B_1\times[-2,0))\cap \bC\setminus \tilde U$, then for all $i\in \Lambda$,
    $(x,y,t)\notin T_{|x_i|,1}(y_i,s_i)$
    due to \eqref{pcos}. Thus by \eqref{pcover},
    \begin{equation}\label{pcosub2}\begin{split}
        (B_1\times[-2,0))\cap \bC\setminus U&\subset
        (B_1\times[-2,0))\cap \bC\setminus\bigcup_{i\in \Lambda} T_{|x_i|,1}(y_i,s_i)\\ &\subset\bigcup_{|y|<9/8,\,s\in(-9/4,0)}\bC\cap \tilde P_{3g(y,s)}(0,y,s).\end{split}
    \end{equation}
 We have equally \eqref{pcover2}, and 
 \eqref{pcover3} holds true for integration over $\bC$. Then arguing similarly in \eqref{pcosub} using \eqref{pcosub2}, we obtain \eqref{parC}.
 
    {\bf Step 6}. We have \eqref{para1-ex}.
    \newline
    {\bf Proof of Step 6}. The estimate of integral over 
    $U\setminus \cup_{i\in \Lambda} T_{|x_i|,1/4}(y_i,s_i)$
    can be carried out similarly as in Step 5 using
    \eqref{pcover} and $|\nabla f|\leq \beta$ pointwise
    on $U$ thus we only need to estimate
    the summation of integral on each $T_{|x_i|,1/4}(y_i,s_i)\cap\bC$ for $i\in \Lambda$. Since $T_{|x_i|,1}(y_i,s_i)\cap \bC\subset\tilde U$, ${\rm spt}\|\mathscr V\|$ is 
    represented as ${\rm graph}\,f$ with the
    size of gradient less than $\beta$.
   By restricting $\beta$ suitably small, one can apply Proposition \ref{e-reg} to conclude
\begin{equation}\label{pcover5}
    \int_{\bC\cap T_{|x_i|,1/4}(y_i,s_i)}
    |x|^2|\nabla f|^2
    \leq \Cr{linestimate}\Big(\int_{T_{|x_i|,1}(y_i,s_i)}
    {\rm dist}(X,\bC)^2\,d\|\mathscr V\|
    +|x_i|^{k+4+2\alpha}\|u\|^2\Big),
\end{equation}
where $\|u\|=\|u\|_{L^{p,q}(P_2)}$. To sum the second summand in this estimate over $i \in \Lambda$, we argue by dyadic decomposition. We partition the index set $\Lambda$ into shells $\Lambda_j = \{ i \in \Lambda \,:\, 2^{-j-1} < |x_i| \le 2^{-j} \}$ for $j \in \mathbb{N}$. Since $\{T_{|x_i|,1}\cap\bC\}_{i \in \Lambda_j}$ have a volume of order $(2^{-j})^{k+2}$ and bounded overlap, the cardinality $\#(\Lambda_j)$ is bounded by $C(n,k)(2^{-j})^{-(k+2)}$. Summing the forcing term's contribution over all shells then yields:
\begin{equation*}
\begin{split}
\sum_{i\in \Lambda} |x_i|^{k+4+2\alpha}\|u\|^2 &= \sum_{j=0}^{\infty} \sum_{i \in \Lambda_j} |x_i|^{k+4+2\alpha}\|u\|^2 \le \|u\|^2 \sum_{j=0}^{\infty} \#(\Lambda_j) (2^{-j})^{k+4+2\alpha} \\
&\le C \|u\|^2 \sum_{j=0}^{\infty} (2^{-j})^{2+2\alpha}.
\end{split}
\end{equation*}
As $\alpha>0$, this geometric series converges. Summing the full estimate \eqref{pcover5} over all $i \in \Lambda$ and using the bounded overlap of the domains on the right-hand side, we obtain
\begin{equation}
    \int_{\cup_{i\in \Lambda}
    T_{|x_i|,1/4}(y_i,s_i)\cap U}|x|^2|\nabla f|^2
    \leq \Cl[con]{listing}\Big(\int_{P_2}{\rm dist}(X,\bC)^2\,d\|\mathscr V\|
    +\|u\|^2\Big),
\end{equation}
where $\Cr{listing}$ depends on $n,k$, and $\Cr{linestimate}$. This completes the proof of Step 6 and of Theorem \ref{thm:graphical}.
\end{proof}

\begin{remark} \label{rmk:C1a-deg}
We note two further consequences of the proof of Theorem \ref{thm:graphical}, which we record here for later use.

First, the argument provides a quantitative relationship between the global excess $\mu$ and the size of the non-graphical region. The proof requires the condition $\mu < \Cr{eps-tau}(\sigma)$, where the threshold can be chosen such that $\Cr{eps-tau}(\sigma) < \sigma^{k+4}\, \Cr{eps-lip}$. By inverting this, we see that for a given $\mu$, the conclusions of the theorem hold for any $\sigma > C' \mu^{1/(k+4)}$. In particular, the support of the flow is guaranteed to be a graph in the region
\[
    \{ (x,y,t) \in B_{R/2} \times (-R^2/2, 0) : |x| > C' R \mu^{1/(k+4)} \}\,,
\]
where $C'=C'(k,n,p,q,E_1)$. In other words, the information currently at our disposal implies that, in a ``blow-up'' regime, with $\mu \to 0$, the graphicality region approaches the spine at the rate $\mathrm{O}(\mu^{1/(k+4)})$. The non-concentration estimates of Section \ref{sec:HS}, among other things, have the goal of greatly improving this picture, as they show that triple junction points (roughly speaking the points where graphicality fails), albeit always present, are situated at a distance $\mathrm{O}(\mu)$ from the spine. 

Second, the local nature of Proposition \ref{e-reg} implies a weighted $C^{1,\alpha}$ estimate for the graphing function $f$ on its entire domain of definition $U$. Indeed, if $(x,y,t) \in U$ then by definition the flow is graphical in the parabolic cylinder $P_{|x|/14}(x,y,t)$, and the estimates \eqref{udef5} can be applied with space-time center $(x,y,t)$ at the scale $R=|x|/14$. Since $(x,y,t) \in U$ is arbitrary, one obtains, in particular, that
\[
    \sup_{(X,t) \in U} |x|^{k/2 + 2} \left( |x|^{-1}|f(X,t)| + |\nabla f (X,t)| \right) \leq C R^{k/2 + 2} \max\{\mu,\|u\|\}\,.
\]
The H\"older seminorms of $\nabla f$ are controlled in a similar fashion. This estimate, which provides a more precise version of \eqref{expara1}, quantifies the natural degeneration of $C^1$ regularity in space when approaching the spine. 
\end{remark}

\section{No-hole property} \label{sec:NH}

This section is dedicated to establishing a crucial structural property of the flow, which we call, as it is customary in the literature, the ``no-hole'' property (Proposition \ref{p:NH_property}). This result guarantees that for any time $t$ and any slice location $y$ on the spine, there must exist a singular point $\Xi$ (with Gaussian density $\Theta(\Xi,t) \ge 3/2$) nearby. This proposition is an essential prerequisite for the analysis in Section 5, which requires a singular point to serve as a center for the non-concentration estimates. The proof is by contradiction, leveraging the graphical control from Section \ref{sec:graph} against the constraints imposed by White's stratification theorem on the dimension of the singular set.

\begin{proposition} \label{p:NH_property}
    For every $\delta \in \left(0, \sfrac18\right)$ there exists $\Cl[eps]{e_NH} = \Cr{e_NH}(k,n,p,q,E_1,\delta) \in(0,1)$ with the following property. Assume that $(\{V_t\},\{u(\cdot,t)\})\in \mathscr N_{\Cr{e_NH}}(U_R\times(-R^2,0])$
    and satisfies (A1)-(A5). Then
    we have the following:
    \begin{equation} \label{e:NH}
    \begin{split}
        &\mbox{for every $y \in B^{k-1}_{R/2} \subset \R^{k-1} = \mathbf{S}(\mathbf{C})$ and for every $t \in [-R^2/2,0]$} \\
        &\mbox{there exists $\Xi \in 
        B_{\delta R}^{n-k+1}\times\{y\}$ such that $\Theta(\Xi,t) \ge \sfrac32$}\,. 
    \end{split}\tag{\rm{NH}}
    \end{equation}
\end{proposition}

\begin{proof}
Without loss of generality, assume $R=2$.
First notice that it is sufficient to prove the validity of \eqref{e:NH} for $\mathcal H^{k-1}$-a.e. $y\in B_{1}^{k-1}$ and a.e. $t \in [-2,0]$ because the function $\Theta(X,t)$ is upper semi-continuous. Next observe that, under the present assumption 
with $\varepsilon$ sufficiently small, the Gaussian density satisfies $\Theta(X,t) < 2$ for every $(X,t) \in B_1\times[-2,0]$
by arguing as in the proof of Proposition \ref{prouniden}. Now, at each point $(X,t)\in {\rm spt}\,\|\mathscr V\| \cap (B_1 \times [-2,0])$, we have a set of tangent flows, each of which may be classified as static, quasi-static, or shrinking (see \cite{White_stratification}). Any static tangent flow with spine of dimension $=k$ is a multiplicity one static $k$-dimensional plane due to the fact that $\Theta (X,t) < 2$. By \cite{ST_endtime}, any point where the flow admits such a tangent flow has a regular neighborhood. Analogously, if the flow admits at $(X,t)$ a static tangent flow with spine of dimension $k-1$ then $\Theta(X,t) = \sfrac{3}{2}$ and the tangent flow is a static multiplicity one $k$-dimensional triple junction. Consider then the set of point of ${\rm spt}\|\mathscr V\|\cap(B_1\times[-2,0])$ where the flow does not admit any of the two types of static tangent flows discussed above, and 
call it $G$. In other words, $G$ is a set
of point where any static tangent flow has 
spine dimension $\leq k-2$, or where there may be quasi-static or shrinking tangent flows.
By the stratification theorem of White \cite[Theorem 9]{White_stratification}, $G$ has parabolic Hausdorff dimension $\leq k$. 
For $t\in[-2,0]$, define
\begin{equation}
\begin{split}
    &G_t:=\{X=(x,y)\in B_1^{n-k+1}\times B_1^{k-1} : 
    (X,t)\in G\}, \\
    &{\bf S}(G_t):=\{y\in B_1^{k-1} : 
    (B_1^{n-k+1}\times\{y\})\cap G_t\neq \emptyset\},
    \end{split}
\end{equation}
and define for each $m\in\mathbb N$
\begin{equation}
    T_{G,m}:=\{t\in[-2,0] : \mathcal H^{k-1}(
    {\bf S}(G_t))>1/m\}\,. 
\end{equation}
We claim that $\mathcal H^1(T_{G,m})=0$, which in turn shows
$\mathcal H^1(\{t\in[-2,0] : \mathcal H^{k-1}({\bf S}(G_t))>0\})=0$. Since the parabolic Hausdorff dimension of $G$ is $\leq k$, for every $\varepsilon>0$ we may choose a set of 
parabolic cylinders $\{P(i)\}_{i\in\mathbb N}$
with the form $P(i)=B_{r_i}(X_i)\times(t_i-r_i^2,t_i)$ such that 
\begin{equation}
    G\cap ((B_1^{n-k+1}\times B_1^{k-1})\times[-2,0])\subset \cup_{i=1}^\infty
    P(i)
\end{equation}
and 
\begin{equation}
    \sum_{i=1}^\infty r_i^{k+1}<\varepsilon.
\end{equation}
If we write $\tilde P(i)$ as the projection to 
the spine, namely, $$\tilde P(i):=\{(y,t)\in 
\mathbb R^{k-1}\times \mathbb R : (\mathbb R^{n-k+1}\times\{y\}\times\{t\})\cap P(i)\neq \emptyset\},
$$
then 
\begin{equation}
   \cup_{t\in T_{G,m}} {\bf S}(G_t)\times \{t\}
   \subset \cup_{i=1}^\infty \tilde P(i).
\end{equation}
Note that $\mathcal H^{k}(\cup_{i=1}^\infty \tilde P(i))\leq c(k)\sum_{i=1}^\infty r_i^{k+1}$ while
$\mathcal H^k(\cup_{t\in T_{G,m}} {\bf S}(G_t)\times\{t\})\geq \mathcal H^1(T_{G,m})
\times \sfrac{1}{m}$. Since $\varepsilon>0$ is arbitrary, this in particular implies $\mathcal H^1(T_{G,m})=0$, proving the claim. We have then proved that for a.e.~$t\in[-2,0]$
and a.e.~$y\in B_1^{k-1}$, the set $B_1^{n-k+1}\times\{y\}\times\{t\}$ does not intersect $G$: hence, at every point of 
${\rm spt}\,\|\mathscr V\|\cap (B_1^{n-k+1}\times\{y\}\times\{t\})$, the flow has a static tangent flow consisting of either a multiplicity one
$k$-dimensional plane or a multiplicity one $k$-dimensional triple junction. If there is no point with triple junction tangent flow, then, we must have a space-time neighborhood 
of $B_1^{n-k+1}\times\{y\}\times\{t\}$ such that ${\rm spt}\|\mathscr V\|$ is a flow of $C^{1,\alpha}$
$k$-dimensional surfaces. On 
the other hand, away from a tubular neighborhood of the spine, ${\rm spt}\,\|\mathscr V\|$ is a graph over the three 
$k$-dimensional planes in $\bC$ by Theorem \ref{thm:graphical}. In particular, for such times $s$ in a neighborhood of $t$ and for such points $z$ in a neighborhood of $y$ the slice of $\spt\|V_s\|$ with the plane through $z$ and orthogonal to the spine is a $C^1$ curve in a disc centered at $\{0_{n-k+1}\}$ with trace given by three points on the sphere. This is a contradiction. Thus, for $t$ and $y$ as above there must be at least one point $\Xi \in B_1^{n-k+1}\times\{y\}$ where a tangent flow is a triple junction, and the Gaussian density at such point is equal to $3/2$. By choosing 
$\varepsilon>0$ sufficiently small, we may guarantee that this point is within the $\delta$-neighborhood of the spine, again by Theorem \ref{thm:graphical}. This concludes the proof.
\end{proof}

\section{Non-concentration estimates}\label{sec:HS}

The main result of this section is the following
set of estimates \eqref{e:HS1}-\eqref{e:HS3}. They are the technical core of our paper: as the parabolic analogue of the estimates at the heart of Simon's regularity theory for minimal surfaces \cite{Simon_cylindrical}, they are the essential ingredient to carry out the blow-up method in Section \ref{sec:BU}. First, \eqref{e:HS1} shows that any point of high Gaussian density must lie close to the spine of the reference cone, with a distance linearly controlled by the global excess $\mu$. Second, \eqref{e:HS2} is the central analytical result: it implies, in particular, that the $L^2$-excess does not concentrate at small scales, as specified in the corollary recorded as Proposition \ref{p:nonconc-excess}. Finally, \eqref{e:HS3} translates \eqref{e:HS2} into an estimate for the graphing function in the graphicality region. 

\begin{theorem}\label{thm:HS}
There exist $\Cl[eps]{e_HS}=\Cr{e_HS}(k,n,p,q,E_1) \in (0,1)$ 
and $\Cl[con]{c_HS} =\Cr{c_HS}(k,n,p,q,E_1,\Cr{conslice})\in(1,\infty)$
so that the following holds. Assume that $(\{V_t\},\{u(\cdot,t)\})\in \mathscr N_{\Cr{e_HS}}(U_5\times[-25,0])$ satisfies (A1)-(A6). Then, for any point $(\Xi,\tau) \in P_1$ with $\Theta (\Xi,\tau) \geq \sfrac32$ we have, setting $\Xi = (\xi,\zeta)$ with $|\xi| = \dist (\Xi, \mathbf{S}(\mathbf{C}))$:
    \begin{equation} \label{e:HS1}
        |\xi| \leq \Cr{c_HS}\,\max\{\mu, \|u\|\}\,.
    \end{equation}
    Furthermore, for every $\kappa \in [0,1)$ there exists $\Cl[con]{c_HS2}
    =\Cr{c_HS2}(k,n,p,q,E_1,\Cr{conslice},\kappa)\in (1,\infty)$ such that
    \begin{equation} \label{e:HS2}
        \sup_{t \in [-1+\tau,\tau)} (\tau-t)^{-\kappa} \int_{B_1} \dist^2(X-\Xi,\mathbf{C}) \, \rho_{(\Xi,\tau)}(X,t)\, d\|V_t\|(X) \leq \Cr{c_HS2}\,\max\{\mu,\|u\|\}^2\,,
    \end{equation}
as well as
\begin{equation}\label{e:HS3}
      \sup_{t\in[-1+\tau,\tau)}  
      (\tau-t)^{-\kappa-\frac{k}{2}}\int e^{-\frac{|X+f(X,t)-\Xi|^2}{4(\tau-t)}}|f(X,t)-\xi^{\perp_X}|^2\,d\mathcal H^k(X)\leq \Cr{c_HS2}\,\max\{\mu,\|u\|\}^2,
    \end{equation}
    where $f$ is as in Theorem \ref{thm:graphical}, the integration is over
    $\{X\in B_1\cap \bC : (X,t)\in U\}$, and $\xi^{\perp_X}$ is the
    projection of $\xi$ to $\bC^\perp$
    at $X\in \bC$.
\end{theorem}

The core calculation is in the estimates of Proposition \ref{p:HS_technical}. The right-hand sides of \eqref{e:HS_technical} and \eqref{e:HS_technical2} are further estimated in terms of $\max\{\mu,\|u\|\}^2$ thanks to Proposition \eqref{p:error_t_est}. All preliminary estimates of this section culminate there. The analysis begins with Proposition \ref{lateral}, where we obtain \eqref{nocon1} from a localized version of Huisken's monotonicity formula and a comparison between the flow $\{V_t\}_t$ and $\bC$. We must then aim at estimating the right-hand side of \eqref{nocon1}. An essential step towards this goal is made in deriving formula \eqref{e:compar} in Proposition \ref{prononwith}, similar to the formula in \cite[p.614 (3)]{Simon_cylindrical}. The technical work to estimate the terms appearing in \eqref{e:compar} is carried in Propositions \ref{lateral-pro} to \ref{p:error0}. This is where Assumption (A6) is crucially needed.

\subsection{Preliminary estimates}\label{preest}
Throughout Subsection \ref{preest},
we assume that $(\{V_t\},\{u(\cdot,t)\})\in\mathscr N_{\varepsilon}(U_4\times(-16,0])$, it
satisfies (A1)-(A6), and assume
$\varepsilon\leq \min\{\Cr{onlyunit},
\Cr{eps-tau}\}$, where
$\Cr{onlyunit}$ corresponds to $r=15/16$ in Proposition \ref{prouniden} and $\Cr{eps-tau}$ to $\beta=\sigma=1/40$ in
Theorem \ref{thm:graphical}. 

\begin{proposition}\label{lateral}
    Suppose that $\tilde\eta:[0,\infty)\rightarrow[0,1]$ is a $C^\infty$ function such that $\tilde\eta=1$ on $[0, 1/2]$, $\tilde\eta=0$ on $[1,\infty)$ and $\tilde\eta'\leq 0$. 
    Set $\eta(x,y):=\tilde\eta(|x|)\tilde\eta(|y|)$
    and $\rho:=\rho_{(0,0)}$. Assume that $\Theta(0,0)\geq\sfrac32$. Then for $0>s>-16$, 
    \begin{equation}
    \begin{split}
    \frac12\int_s^0\int \eta\rho&\left|h-\frac{(\nabla\rho)^\perp}{\rho}\right|^2\,d\|V_t\|dt\leq \int \eta\rho(\cdot,s)\,d\|V_s\|-\int
    \eta\rho(\cdot,s)\,d\|\bC\| \\ &+\int_s^0\int\nabla\rho\cdot\nabla\eta+\frac32\rho\eta|u|^2+\rho\frac{3|(\nabla\eta)^\perp|^2}{2\eta}\,dV_tdt-\int_s^0\int\nabla\rho\cdot\nabla\eta\,d\|\bC\|dt.
    \end{split}\label{nocon1}
    \end{equation}
\end{proposition}
\begin{proof}
Use \eqref{e:Brakke-ineq} with $\phi=\eta\rho$ and
$s<t_2<0$
to obtain
\begin{equation}
\begin{split}
    &\|V_t\|(\eta\rho)\Big|_{t=s}^{t_2}\\&\leq\int_{s}^{t_2}\int(\nabla(\eta\rho)-\eta \rho h)\cdot(h+u^\perp)+\eta\partial_t \rho\,dV_t dt 
    =\int_s^{t_2}\int -\left|h-\frac{(\nabla\rho)^\perp}{\rho}\right|^2\rho\eta\,dV_t dt \\& \quad+\int_{s}^{t_2}\int-h\cdot(\eta\nabla \rho-\rho\nabla\eta)+\frac{|(\nabla\rho)^\perp|^2}{\rho}\eta
    +u^\perp\cdot(\nabla(\eta\rho)-\eta\rho h)+\eta\partial_t \rho\,dV_tdt.
\end{split}\label{add1}
\end{equation}
Since 
\begin{equation}
\begin{split}
&    \int -h\cdot(\eta\nabla\rho-\rho\nabla\eta)\,dV_t \\
&\qquad \qquad=\int \eta\nabla^2\rho\cdot S+(\nabla\eta\otimes\nabla\rho)\cdot S+\nabla \eta\cdot\left(h-\frac{(\nabla\rho)^\perp}{\rho}\right)\rho+\nabla\eta\cdot(\nabla\rho)^\perp\,dV_t 
\end{split}
\end{equation}
and $(\nabla\eta\otimes\nabla\rho)\cdot S+\nabla\eta\cdot(\nabla\rho)^\perp=\nabla\eta\cdot\nabla\rho$, we may easily estimate
\begin{equation}\label{add2}
\begin{split}
&\int -h\cdot(\eta\nabla\rho-\rho\nabla\eta)\,dV_t \\
& \qquad \qquad\leq \int\eta\nabla^2\rho\cdot S+\nabla\eta\cdot\nabla\rho
    +\frac14\left|h-\frac{(\nabla\rho)^\perp}{\rho}\right|^2 \rho\eta+\rho\frac{|(\nabla\eta)^\perp|^2}{\eta}\,dV_t.
\end{split}
\end{equation}
For the term involving $u$, 
\begin{equation}\label{add3}
\begin{split}
    \int u^\perp\cdot(\nabla(\eta\rho)-\eta\rho h)\,dV_t&=\int \rho\eta\left(\frac{(\nabla\rho)^\perp}{\rho}-h\right)\cdot u^\perp+\rho\nabla\eta\cdot u^\perp \\
    &\leq \int \frac14\rho\eta\left|h-\frac{(\nabla\rho)^\perp}{\rho}\right|^2
    +\frac32\rho\eta|u|^2+
    \rho\frac{|(\nabla\eta)^\perp|^2}{2\eta}\,dV_t.
    \end{split}
\end{equation}
For $\bC$, we have
\begin{equation}\label{add4}
    \|\bC\|(\eta\rho)\big|_{t=s}^{t_2}=\int_{s}^{t_2}\int\eta\partial_t \rho\,d\|\bC\|dt
=-\int_{s}^{t_2}\int\eta\nabla^2\rho\cdot S\,d\bC dt=\int_{s}^{t_2}\int\nabla\eta\cdot\nabla\rho\,d\|\bC\|dt.
\end{equation}
Since $\lim_{t_2\rightarrow 0-}\|V_{t_2}\|(\eta\rho)\geq\sfrac32$ and $\lim_{t_2\rightarrow 0-}\|\bC\|(\eta\rho)=\sfrac32$, using \eqref{add1}-\eqref{add3} and subtracting \eqref{add4}, the identity
\begin{equation}\label{e:Huisken_id}
    \partial_t \rho+S\cdot\nabla^2 \rho+\frac{|S^\perp(\nabla\rho)|^2}{\rho}=0, 
\end{equation} 
yields \eqref{nocon1}.
\end{proof}

The following proposition establishes the parabolic analogue of \cite[p. 614, (3)]{Simon_cylindrical}. In the stationary case, the formula is obtained essentially by testing the first variation formula with the gradient of the distance squared from $\bS(\bC)$ (multiplied by a suitable cut-off); here, we will test Brakke's inequality precisely with $\dist^2(\cdot,\bS(\bC))$ multiplied by a suitable cut-off. We recall that, in our system of coordinates, $\mathrm{span}(e_j)_{n-k+2\leq j\leq n}=\bS(\bC)$.

\begin{proposition}\label{prononwith}
    For any non-negative function $\psi\in C^\infty_c(U_2\times[-4,0])$
    with $\psi(X,-4)=0$,
    we have 
\begin{equation}\label{e:compar}
    \begin{split}
       & \iint \psi \left( \frac{|x|^2}{2} \, \left| h \right|^2 + 1 + \sum_{j=n-k+2}^n |S^\perp e_j|^2 \right) \, d V_t \,dt - \iint \psi \, d\|\bC\|\,dt \\
       &\qquad \quad
       \leq \int \frac{|x|^2}{2} \psi(\cdot,0) \, d\|\bC\| - \int \frac{|x|^2}{2} \psi(\cdot,0) \, d\|V_0\| \\
        &\qquad\quad +\iint \frac{|x|^2}{2}\psi_t\,d\|V_t\|dt-\iint\frac{|x|^2}{2}\,\psi_t\,d\|\bC\|dt\\
       &\qquad \quad +2  \iint S^\perp x \cdot \nabla_{\mathbf{S}}\psi \,  dV_t\,dt \\
       &\qquad \quad -  2\iint S x \cdot \nabla_{\mathbf S^\perp}\psi \, dV_tdt+ 2 \iint x \cdot \nabla_{\mathbf{S}^\perp}\psi \, d\|\bC\|\,dt \\
       &\qquad\quad 
       -\iint \frac{|x|^2}{2}\nabla^2\psi\cdot S\,dV_t dt
       +\iint \frac{|x|^2}{2}\nabla^2\psi\cdot S\,d\bC dt \\
    &\qquad\quad+  \iint  \left| \nabla^\perp \left( \frac{|x|^2}{2} \psi  \right) \right| \,  |u| + \frac{|x|^2}{4} \psi  |u|^2 \,dV_t\,dt=: I_0+\cdots+I_5,
    \end{split}
\end{equation}   
where all the integrals take place on $P_2=U_2\times(-4,0)$ or $G_k(U_2)\times(-4,0)$, and where $\psi_t=\partial_t\psi$. Also, $\nabla_{\bS}$ and $\nabla_{\bS^\perp}$ are the projections of the gradient operator on $\bS(\bC)$ and its orthogonal complement, respectively. 
\end{proposition}
\begin{proof}

We test Brakke's inequality \eqref{e:Brakke-ineq} with
\[
\phi(X,t) = |x|^2\,\psi(X,t) = \dist^2(X, \mathbf{S}(\bC))\,\psi(X,t)\,,
\]
on the time interval $[-4,0]$. By the hypothesis on $\psi$, we then get
\begin{equation}\label{d2B}
\int |x|^2 \psi(X,0) \, d\|V_0\| \leq \iint (\nabla(|x|^2\psi)-|x|^2\psi \,h) \cdot (h + u^\perp) + |x|^2\psi_t\,dV_t\,dt\,,
\end{equation}
and we work on each summand on the right-hand side. First, for a.e. $t \in (-4,t_2)$ 
\[
\int \nabla(|x|^2\psi)\cdot h \, d\|V_t\|\,dt=-\int \nabla^2(|x|^2\psi) \cdot S\, dV_t\,dt\,.
\]
Notice now that
\begin{equation} \label{hessian term}
\nabla^2(|x|^2\psi)  \cdot S = 2\,\psi\,  \mathbf{S}^\perp \cdot S + 4 x \otimes \nabla \psi \cdot S + |x|^2 \nabla^2 \psi \cdot S\,.
\end{equation}
Let us denote $(e_1, \ldots, e_{n-k+1})$ and $(e_{n-k+2},\ldots e_n)$ orthonormal bases of $\mathbf{S}^\perp$ and $\mathbf{S}$, respectively.  We can then write
\[
\mathbf{S}^\perp \cdot S=\sum_{i=1}^{n-k+1} e_i \cdot S e_i = (n-k+1) - \sum_{i=1}^{n-k+1} e_i \cdot S^\perp e_i\,,
\]
and since
\[
\sum_{i=1}^{n-k+1} e_i \cdot S^\perp e_i + \sum_{i=n-k+2}^{n} e_i \cdot S^\perp e_i = {\rm tr}(S^\perp) = n-k,
\]
we have
\begin{equation} \label{hessian term 1}
\mathbf{S}^\perp \cdot S= 1 + \sum_{j=n-k+2}^{n} e_j \cdot S^\perp e_j = 1 + \sum_{j=n-k+2}^n |S^\perp e_j|^2\,.
\end{equation}

Then, we estimate the term involving $u$ by
\[
(\nabla(|x|^2\psi)-|x|^2\psi\,h)\cdot u^\perp \leq \left|\nabla^\perp (|x|^2\psi)\right||u| + \frac12 \psi |x|^2|h|^2 + \frac12\psi|x|^2|u|^2\,.
\]

Finally, we notice that for every vector $w \in \R^n$,
\[
Sx \cdot w = Sx \cdot (\mathbf{S}w + \mathbf{S}^\perp w) = Sx \cdot \mathbf{S}^\perp w - S^\perp x \cdot \mathbf{S} w\,,
\]
since $x \cdot \mathbf{S}w = \mathbf{S}^\perp X \cdot \mathbf{S}w=0$. In particular,
\begin{equation} \label{hessian term 2}
Sx \cdot \nabla \psi = Sx \cdot \nabla_{\mathbf{S}^\perp}\psi - S^\perp x \cdot \nabla_{\mathbf{S}}\psi\,.
\end{equation}

We can then conclude from \eqref{d2B}-\eqref{hessian term 2} that
\begin{equation}\label{d2B2}
\begin{split}
  & \iint  \psi \left( \frac14 |x|^2|h|^2 + 1 + \sum_{j=n-k+2}^{n}|S^\perp e_j|^2 \right) \,dV_t\,dt \\ &\quad\leq - \frac12\int |x|^2 \psi(\cdot,0) \, d\|V_0\| 
  +\iint \frac12 |x|^2\psi_t + 2 S^\perp x \cdot \nabla_{\mathbf{S}}\psi - 2 Sx \cdot \nabla_{\mathbf{S}^\perp}\psi \\&\qquad -\frac12|x|^2\nabla^2\psi \cdot S+ 
  \frac12\left|\nabla^\perp (|x|^2\psi)\right||u| + \frac14\psi|x|^2|u|^2\,dV_t\,dt\,.
  \end{split}
\end{equation}

Applying now Brakke's inequality (as an equality) to the constant Brakke flow identically equal to $\bC$, again with test function $\phi=|x|^2\psi$ on $[-4,0]$, we get
\begin{equation} \label{d2cone}
    \iint \psi \, d\|\bC\|\,dt =-\frac12 \int |x|^2 \psi(\cdot,0) \, d\|\bC\| +\iint \frac12 |x|^2\psi_t  - 2 x \cdot \nabla_{\mathbf{S}^\perp}\psi -\frac12|x|^2\nabla^2\psi \cdot S \,d\bC\,dt\,.
\end{equation}

The inequality \eqref{e:compar} is obtained by subtracting \eqref{d2cone} from \eqref{d2B2}. 
\end{proof}

In the following, we will work on the term
\[
\iint \psi\,d\|V_t\|\,dt-\iint\psi \,d\|\bC\|\,dt
\]
appearing on the left-hand side of \eqref{e:compar}. Writing it as
\[
\int_{-4}^0\left(\int \psi \,d\|V_t\| -\int \psi \,d\|\bC\|\right)_+\,dt -  \int_{-4}^0\left(\int \psi \,d\|\bC\| -\int \psi \,d\|V_t\|\right)_+\,dt\,,
\]
we bring the second summand on the right-hand side of \eqref{e:compar}, and we proceed to estimate it in terms of the square of the excess. This is where assumption (A6) is crucial.

\begin{proposition}\label{lateral-pro}
    Suppose that $\psi=\psi(x,y,t)\in C^\infty(P_2)$ is non-negative and
    radially symmetric with respect to $x$. Then 
    there exists $\Cl[con]{la1}=\Cr{la1}(n,k,p,q,E_1,\Cr{conslice},\|\psi\|_{C^3})\in(1,\infty)$ such that
    \begin{equation}
        \int_{-4}^0\Big(\int_{B_1^{n-k+1}\times B_1^{k-1}
    }\psi\,d\|\bC\|-\int_{B_1^{n-k+1}\times B_1^{k-1}
    }\psi\,d\|V_t\|\Big)_+\,dt\leq \Cr{la1}\max\{\mu,\|u\|\}^2.
    \end{equation}
\end{proposition}
\begin{proof}
We use Theorem \ref{thm:graphical} with 
   $\beta=\sigma=1/40$ and $R=4$ to obtain 
   a graphical representation $f:U\rightarrow \bC^\perp$ on
$U\subset \bC\cap (U_2\times(-8,0))$  with the error estimates \eqref{para1}-\eqref{para1-ex}. In the following, we redefine $U$ to be $U\cap (B_1^{n-k+1}\times B_1^{k-1}\times(-4,0))$ and note that all the estimates related to
$f$ hold just as well even for this new $U$.
Let $\tilde\eta$ be as in Proposition \ref{lateral} and define
\begin{equation}\label{lateral2}
    \phi(x,y,t):=|x|^{-2}\big\{
    \psi(x,y,t)-\psi(0,y,t)\tilde\eta(|x|)\big\}. 
\end{equation}
Since $\psi$ is radially symmetric with respect to $x$ and smooth, $\phi$ is a $C^1$ function
with $\|\phi\|_{C^1}\leq c(n,k,\|\psi\|_{C^3},\|\tilde\eta\|_{C^1})$. 
By \eqref{lateral2},
\begin{equation}\label{lateral3}
\psi(x,y,t)=|x|^2\phi(x,y,t)+\psi(0,y,t)\tilde\eta(|x|),
\end{equation}
and we estimate the integral of each term. For the first term, we claim
\begin{equation}\label{lateral4}
    \int_{-4}^{0}\Big|\int_{B_1^{n-k+1}\times B_1^{k-1}} |x|^2\phi\,d\|V_t\|
    -\int_{B_1^{n-k+1}\times B_1^{k-1}}|x|^2\phi\,d\|\bC\|\Big|dt\leq \Cl[con]{dete}\min\{\mu,\|u\|\}^2
\end{equation}
where $\Cr{dete}$ depends only on $n,k,\|\phi\|_{C^1}$
and $\Cr{nonsig}$. 
The non-graphical part is estimated by $\Cr{nonsig}\mu^2$ due to
\eqref{para1} and \eqref{parC}. 
For the graphical part, we may estimate the difference for
the corresponding term over $\bC$ as
\begin{equation}
    \int_{U}|(|x|^2+|f|^2)\phi J_{\nabla f}
    - |x|^2\phi|\,d\mathcal H^kdt\leq c(n,k,\|\phi\|_{C^1})
    \int_U |f|^2+|x|^2|\nabla f|^2d\mathcal H^kdt.
\end{equation}
The last expression is bounded by a constant multiple of 
$\max\{\mu,\|u\|\}^2$ due to \eqref{para1-ex} and the definition 
of $\mu$. This gives \eqref{lateral4}.

For the second term of \eqref{lateral3}, we use (A6), in particular \eqref{cotriple}. 
Note that $\psi(0,y,t)\tilde\eta(|x|)$ is a constant 
function on $|x|\leq 1/2$ for each fixed $y$ and $t$. 
By the graphical
representation 
as well as \eqref{e:bound} and Proposition \ref{e-reg}, the 
$C^1$-norm of $f$ on $1/16\leq |x|\leq 1$ is 
bounded by a constant multiple of
$\max\{\mu,\|u\|\}$. Thus, $K$ in \eqref{anasup} is bounded by 
a constant multiple of $\max\{\mu,\|u\|\}$, and by \eqref{cotriple} and $\mathcal H^{k-1}$-a.e.~$y$ with $|y|<1$, we have
\begin{equation}
    \int_{B_1^{n-k+1}\cap\hat{\bC}}\tilde\eta(|x|)\,
    d\mathcal H^1(x)-\int_{B_1^{n-k+1}\cap M_t^y}\tilde\eta(|x|)\,d\mathcal H^1(x)\leq \Cl[con]{ancon1}\max\{\mu,\|u\|\}^2,
\end{equation}
where $\Cr{ancon1}$ depends in addition on $\Cr{conslice}$. 
By the coarea
formula, we have for a.e.~$t$
\begin{equation}\begin{split}
    \int_{y\in B_1^{k-1}}\psi(0,y,t)&\int_{B_1^{n-k+1}\cap M_t^y} \tilde\eta(|x|)\,d\mathcal H^1(x)d\mathcal H^{k-1}(y)\\&
    \leq \int_{M_t\cap (B_1^{n-k+1}\times B_1^{k-1})}\psi(0,y,t)\tilde\eta(|x|)\,d\|V_t\|,
    \end{split}
\end{equation}
and thus 
\begin{equation}
\begin{split}
    \int_{B_1^{n-k+1}\times B_1^{k-1}}\psi(0,y,t)\tilde\eta(|x|)\,d\|\bC\|&-\int_{B_1^{n-k+1}\times B_1^{k-1}} \psi(0,y,t)\tilde\eta(|x|)\,d\|V_t\|\\
    &\leq \Cr{ancon1}(\sup\psi)\max\{\mu,\|u\|\}^2.
    \end{split}
\end{equation}
Combined with \eqref{lateral4}, we proved the desired estimate.
\end{proof}
The same proof with $\phi=0$ together with (A5) shows the following,
which we record for the later use.
\begin{proposition}\label{san-eq}
    Suppose that $\psi=\tilde\psi(y,t)\tilde\eta(|x|)\in C^\infty(P_2)$ is non-negative, where $\tilde\eta$ is as in Proposition \ref{lateral}. Then there exists $\Cl[con]{lmin}=\Cr{lmin}(n,k,p,q,E_1,\Cr{conslice},\|\psi\|_{C^3})\in(1,\infty)$ such that, for each $t\in [-4,0]$, we have
    \begin{equation}
        \int_{B_1^{n-k+1}\times B_1^{k-1}}\psi\,d\|\bC\|-\int_{B_1^{n-k+1}\times B_1^{k-1}}\psi\,d\|V_t\|\leq \Cr{lmin}\max\{\mu,\|u\|\}^2.
    \end{equation}
\end{proposition}

We next turn our attention to the other terms on the right-hand side of \eqref{e:compar}, and finally prove the following estimate.
    
\begin{proposition} \label{lat4}
Suppose that $\psi(x,y,t)\in C_c^\infty(U_1^{n-k+1}\times U_1^{k-1}\times[-4,0])$ is 
non-negative and radially symmetric in $x$ and $y$, that is, there exists $\tilde\psi(s_1,s_2,t)\in C^\infty(\R^3)$ such that $\psi(x,y,t)=\tilde\psi(|x|,|y|,t)$.
Moreover assume $\psi(X,0)=\psi(X,-4)=0$. Then
there exists $\Cl[con]{er1}=\Cr{er1}(n,k,p,q,E_1,\Cr{conslice},\|\psi\|_{C^3})>0$ such that
   \begin{equation}\label{kore}
   \iint_{P_2}
       \sum_{j=n-k+2}^n \psi |S^\perp e_j|^2\,dV_t dt+
    \int_{-4}^0 \Big(\int\psi\,d\|V_t\|-\int\psi\,d\|\bC\|\Big)_+ dt\leq \Cr{er1}\max\{\mu,\|u\|\}^2.
\end{equation}
\end{proposition}
\begin{proof}
   As in Proposition \ref{lateral-pro}, 
   we use Theorem \ref{thm:graphical} with 
   $\beta=\sigma=1/40$ and $R=4$ to obtain 
   a graphical representation $f:U\rightarrow \bC^\perp$ on 
$U\subset \bC\cap P_2$ with the error estimates \eqref{para1}-\eqref{para1-ex}.
In \eqref{e:compar}, $I_0=0$ due to assumption $\psi(X,0)=0$. Thanks to Proposition \ref{lateral-pro}, we only need to estimate $I_1$ to $I_5$. 
\newline
{\bf Estimate of $I_1$}.
\newline
The integration on
$P_2\setminus \,{\rm graph}\,f$ and $P_2\cap \bC\setminus U$ may be estimated by $\Cr{nonsig}\mu^2$ due to \eqref{para1} and
\eqref{parC}, so we only need
to estimate the integration over $U$. We have
\begin{equation}
    \iint_{{\rm graph}\,f}\frac{|x|^2}{2}\psi_t\,d\|V_t\|dt=\iint_{U} \frac{|f(X,t)|^2+|x(X)|^2}{2}\psi_t(X+f(X,t),t)J_{\nabla f(\cdot,t)}\,d\mathcal H^k(X)dt\,,
\end{equation}
and $|\psi_t(X+f(X,t),t)-\psi_t(X,t)| \leq c|f(X,t)|^2$ because, due to radial symmetry, $\nabla\psi_t(X,t) \cdot f(X,t)=0$. Also we have $|f(X,t)|\leq |x|/10$ and 
$|J_{\nabla f(X,t)}-1|\leq c|\nabla f(X,t)|^2$ for $X\in \bf C$ with some $c$ 
depending only on $n,k$.  
Thus we may conclude that
\begin{equation}
    \Big|\iint_{{\rm graph}\,f}\frac{|x|^2}{2}\psi_t\,d\|V_t\|dt-\iint_{U} \frac{|x|^2}{2}\psi_t\,d\mathcal H^k dt\Big|\leq c\max\{\mu,\|u\|\}^2.
\end{equation}
Thus $|I_1|$ is estimated by a constant multiple of $\max\{\mu,\|u\|\}^2$.
\newline
{\bf Estimate of $I_2$}.
\newline
Note that $\psi(x,y,t)=\tilde\psi(|x|,|y|,t)$ so that 
$\nabla_{\bf S}\psi(x,y,t)=\tilde\psi_{s_2}(|x|,|y|,t)y/|y|$ and
\begin{equation}
\begin{split}
    |S^\perp x&\cdot\nabla_{\bf S}\psi|=|y|^{-1}\left|\tilde\psi_{s_2}
    S^{\perp}x\cdot\sum_{j=n-k+2}^n y_j S^\perp e_j\right|\leq 
    |\tilde\psi_{s_2}||S^\perp x|
    \left(\sum_{j=n-k+2}^n|S^\perp e_j|^2\right)^{\frac12} \\
    &\leq \frac{|\nabla\psi|^2}{\psi}|S^\perp x|^2+\frac{\psi}{4}\sum_{j=n-k+2}^n|S^\perp e_j|^2
    \leq 2\|\psi\|_{C^2}|S^\perp x|^2+\frac{\psi}{4}\sum_{j=n-k+2}^n|S^\perp e_j|^2.
    \end{split}
\end{equation}
The second term
will be absorbed to the left-hand side of \eqref{e:compar}.
For the integral of the first term, note that
on $U$ where the support of $\|V_t\|$ is
expressed as ${\rm graph}\,f$, 
\begin{equation}\label{Ses}
|S^\perp x|^2\leq c(n,k)(|f(X,t)|^2 + |x|^2\nabla f|^2)
 \end{equation}   
for some $c(n,k)$. Since $|f|\leq \beta|x|$,
separating integration on $U$ and $B_2\setminus {\rm graph}\,f$ and using 
\eqref{para1}-\eqref{para1-ex}, we may obtain
\begin{equation}
    |I_2|\leq \Cl[con]{I2}\max\{\mu,\|u\|\}^2+\frac{1}{2}\iint \psi\sum_{j=n-k+2}^n|S^\perp e_j|^2\,dV_t dt
\end{equation}
with $\Cr{I2}=\Cr{I2}(n,k,p,q,E_1,\|\psi\|_{C^2})$.
\newline
{\bf Estimate of $I_3$}.
\newline
Since $\nabla_{{\bf S}^\perp}\psi(x,y,t)=\tilde\psi_{s_1}(|x|,|y|,t)x/|x|$, in particular $|Sx\cdot \nabla_{\bS^\perp}\psi|\leq \|\psi\|_{C^2}|x|^2$.
Thus the integral outside of ${\rm graph}\,f$ 
is estimated by the constant multiple of $\max\{\mu,\|u\|\}^2$.
On the integral over $U$, to write the computations
explicitly, we may write the graph representation $f(x_1,y,t):U\subset 
\mathbb R\times
\mathbb R^{k-1}\times\mathbb R\rightarrow\mathbb R^{n-k}$ on one of
the half-space ${\bf H}_1$. Writing $f:=f(x_1,y,t)$,
\begin{equation}\label{Ses2}
\begin{split}
&|(x_1,f)\cdot\nabla_{\bS^\perp}\psi(x_1,f,y,t)
-(x_1,0)\cdot\nabla_{\bS^\perp}\psi(x_1,0,y,t)| 
\\ &
= \Big|\tilde\psi_{s_1}(\sqrt{x_1^2+|f|^2},|y|,t)\sqrt{x_1^2+|f|^2}
-\tilde\psi_{s_1}(x_1,|y|,t)x_1\Big| \leq 
\|\psi\|_{C^2}|f|^2.
\end{split}
\end{equation}
Also, we have
\begin{equation}\label{Ses3}
    |S^\perp x\cdot\nabla_{\bS^\perp}\psi|=|\tilde\psi_{s_1}(|x|,|y|,t)||x|^{-1} |S^\perp x|^2\leq \|\psi\|_{C^2}|S^\perp x|^2.
\end{equation}
Thus using \eqref{Ses2}, \eqref{Ses3} and \eqref{Ses}, we 
may estimate the integral of $I_3$ on the graphical part as 
\begin{equation}\begin{split}
     &\Big|\iint_{{\rm graph}\, f} Sx\cdot\nabla_{\bS^\perp}\psi\,dV_tdt
    -\iint_U x\cdot\nabla_{\bS^\perp}\psi\,d\|\bC\|dt\Big|
    \\ &\leq \Big|\iint_{{\rm graph}\,f} x\cdot\nabla_{\bS^\perp} \psi\,dV_tdt
    -\iint_U x\cdot\nabla_{\bS^\perp}\psi\,d\|\bC\|\Big|
    +\Big|\iint_{{\rm graph}\,f} S^\perp x\cdot\nabla_{\bS^\perp}\psi\,dV_tdt\Big| \\
    &\leq c(n,k,\|\psi\|_{C^2})\iint_U(f^2+|x|^2|\nabla f|^2)\,d\|\bC\|dt. 
\end{split}    
\end{equation}
Here we used that $|x|^2 |J_{\nabla f}-1| \leq |x|^2|\nabla f|^2$. Then the integral is bounded by $c\max\{\mu,\|u\|\}^2$. 
\newline
{\bf Estimate of $I_4$}.
\newline
The integral of the non-graphical part near the
spine may be estimated by $\Cr{nonsig}\mu^2$ due to
\eqref{para1} and \eqref{parC}, so we need to
estimate the graphical part. 
Using the same notation as above, we have
\begin{equation}
\begin{split}
\nabla^2\psi=&(\tilde\psi_{s_1 s_1}-|x|^{-1}\tilde\psi_{s_1})\frac{x\otimes x}{|x|^2}+|x|^{-1}\tilde\psi_{s_1}I_{n-k+1}
+\tilde\psi_{s_1 s_2}\frac{x\otimes y+y\otimes x}{|x||y|} \\
&+(\tilde\psi_{s_2 s_2}-|y|^{-1}\tilde\psi_{s_2})\frac{y\otimes y}{|y|^2}
+|y|^{-1}\tilde\psi_{s_2}I_{k-1}.
    \end{split}
\end{equation}
We need to check the difference between the evaluations of $\nabla^2\psi$ at $(x_1,f(x_1,y,t),y,t)$
and $(x_1,0,y,t)$ using the Taylor theorem. 
For example, one can check 
that ($S$ evaluated as the tangent plane at $(x_1,f,y,t)$)
\begin{equation}\label{fr1}
\begin{split}
    \Big|&(\tilde\psi_{s_1 s_1}-\frac{\tilde\psi_{s_1}}{|x|})\big|_{x=(x_1,f)}
    \frac{(1,f/x_1)\otimes (1,f/x_1)}{1+|f/x_1|^2} 
   \cdot S - (\tilde\psi_{s_1 s_1}-\frac{\tilde\psi_{s_1}}{|x|})\big|_{x=(x_1,0)}
    \Big|\\
    &
    \leq c(n,k)\|\psi\|_{C^3}(|f|^2+|\nabla f|^2).
    \end{split}
\end{equation}
Here we used the fact that $(1,0)\otimes(0,f/x_1)\cdot S=O(|f||\nabla f|)$ and
$\tilde \psi$ has vanishing odd derivatives at $0$.
Using \eqref{fr1} (with also $J_{\nabla f}$ being considered), the integral stemming from the first term (whose integrand also contains $|x|^2$) can be estimated by $c(n,k)\|\psi\|_{C^3}\max\{\mu,\|u\|\}^2$. 
For the second term, Note that $I_{n-k+1}\cdot S-1=
O(|\nabla f|^2)$ and the difference of $|x|^{-1}\tilde\psi_{s_1}$ evaluated at $x=(x_1,f)$
and $x=(x_1,0)$ may be estimated by $c(n,k)\|\psi\|_{C^3}|f|^2$. Using these fact, we may estimate the integral related to the second term similarly. For the third term, first note
that on $\bC$, 
\begin{equation}
    ((1,0)\otimes y+y\otimes (1,0))\cdot S=0. 
\end{equation}
On ${\rm graph}\,f$, since $(x\otimes y)\cdot I_n=0$, we have $(x\otimes y)\cdot S=(y\otimes x)\cdot S=-(x\otimes y)\cdot S^\perp$, and 
\begin{equation}
(x\otimes y)\cdot S^\perp=((x_1,f)\otimes y)\cdot S^\perp=S^\perp((x_1,0))\cdot S^\perp(y)
    +S^\perp((0,f))\cdot S^\perp(y).
\end{equation}
Considering that $S$ is a tangent plane of ${\rm graph}\,f$, we have
$|S^\perp(y)|\leq c(n,k)|y||\nabla f|$ and $|S^\perp((x_1,0))|\leq c(n,k)|x_1||\nabla f|$, and thus
\begin{equation}
    \Big|\tilde\psi_{s_1 s_2}\frac{x\otimes y}{|x||y|}\cdot S\Big|\leq c(n,k)\|\psi\|_{C^2}(|f|^2+|\nabla f|^2),
\end{equation}
and the integral of the third term is bounded by 
$c(n,k)\|\psi\|_{C^2}\max\{\mu,\|u\|\}^2$. 
For the fourth term, the difference
of $\tilde\psi_{s_2 s_2}-|y|^{-1}\tilde\psi_{s_2}$ at ${\rm graph}\,f$ and $\bC$ can be estimated by $c(n,k)\|\psi\|_{C^3}|f|^2$ due to the
radial symmetry. On ${\rm graph}\,f$, $S\cdot(y\otimes y)=|y|^2-|S^\perp(y)|^2$, and $|S^\perp(y)|\leq c(n,k)|y||\nabla f|$, and on $\bC$, this quantity is equal to $|y|^2$. Thus we can handle this 
term similarly as others. For the last term,
$I_{k-1}\cdot S=(k-1)-I_{k-1}\cdot S^\perp=k-1-\sum_{j=n-k+2}^n|S^\perp e_j|^2$, and 
on ${\rm graph}\,f$, $|S^\perp(e_j)|\leq c(n,k)|\nabla f|$. Using this, the integral
of the last term can be bounded by $c(n,k)\|\psi\|_{C^3}\max\{\mu,\|u\|\}^2$.
\newline
{\bf Estimate of $I_5$}.
\newline
The first term is bounded by $(\iint|S^\perp x|^2\psi\,d\|V_t\|dt)^{\frac12}(\iint|u|^2\psi\,d\|V_t\|dt)^{\frac12}$ and can be handled as in the estimate of $I_2$ for the integral of $|S^\perp x|^2$ and by the H\"{o}lder inequality for $|u|^2$. The second term is
bounded by a constant multiple of $\iint|u|^2\psi\,d\|V_t\|dt$, so by the H\"{o}lder inequality, is estimated by $\|u\|^2$. 
\newline
{\bf Summary}.
\newline
Combined with all the estimates above as well
as Proposition \ref{lateral-pro}, we 
obtain \eqref{kore}.
\end{proof}

We shall need a slight modification of Proposition \ref{lat4}, allowing for test functions which may not vanish at time $t=0$. Under our assumptions on the flow, such refinement is possible, as long as the test function is constant in the variable $s_1=|x|$ for small $|x|$.

\begin{proposition} \label{p:error0}
Suppose that $\psi(y,t) \in C^\infty_c(U_1^{k-1}\times[-4,0])$ is non-negative and radially symmetric in $y$, that is there exists $\tilde\psi(s_2,t)\in C^\infty(\R^2)$ such that $\psi(y,t)=\tilde\psi(|y|,t)$. Moreover, assume that $\psi(y,-4)=0$, and let $\tilde\eta=\tilde\eta(|x|)$ be as in Proposition \ref{lateral}. Then, there exists $\Cl[con]{er0}=\Cr{er0}(n,k,p,q,E_1,\Cr{conslice},\|\psi\|_{C^3},\|\tilde\eta\|_{C^3})>0$ such that
\begin{equation} \label{kore:0}
\begin{split}
 &   \iint_{P_2} \sum_{j=n-k+2}^{n} \tilde\eta\psi|S^\perp e_j|^2\, dV_t\,dt + \int_{-4}^0 \left( \int \tilde\eta \psi d\|V_t\| - \int \tilde\eta \psi d\|\bC\|   \right)_+\, dt \\
 & \qquad \qquad \qquad \leq \Cr{er0}\max\{\mu,\|u\|\}^2\,.
\end{split}
\end{equation}
\end{proposition}

\begin{proof}
    The only difference with Proposition \ref{lat4} is that the summand $I_0$ in \eqref{e:compar} corresponding to the choice of $\tilde\eta(|x|)\psi(y,t)$ for $\psi(x,y,t)$ there does not vanish necessarily. We then proceed to estimate it. Recall that
    \begin{equation} \label{who-is-I0}
    2I_0 = \int |x|^2 \tilde\eta(|x|)\psi(y,0) d\|\bC\| - \int |x|^2 \tilde\eta(|x|)\psi(y,0) d\|V_0\|\,.
    \end{equation}
    Let us look at the quantity on the right-hand side of \eqref{who-is-I0} at an arbitrary $t \in [-4,0]$, namely the integral
    \begin{equation} \label{I at time t}
        2 I(t) := \int |x|^2 \tilde\eta(|x|)\psi(y,t) d\|\bC\| - \int |x|^2 \tilde\eta(|x|)\psi(y,t) d\|V_t\|\,.
    \end{equation}
     As usual, we split into integrals over the graphical and non-graphical regions. Let us focus first on the graphical region. By symmetry, we may assume to be working on one of the $k$-dimensional planes in $\bC$: we will let $(x_1,y)$ be the variables on such plane, and the other $(n-k)$ variables will be $(x_2,\ldots,x_{n-k+1})=f (x_1,y,t)$. With this notation set in place,
    \[
    \begin{split}
    &\int_{\mathrm{graph} f(\cdot,t)} |x|^2\tilde\eta(|x|)\psi(y,t) d\|V_t\| \\
    &\quad = \int_{U} \left(|x_1|^2+|f(x_1,y,t)|^2\right) \tilde\eta\left(\left(|x_1|^2 + |f(x_1,y,t)|^2\right)^{\sfrac12}\right)\psi(y,t) J_{\nabla f(\cdot,t)}(x_1,y) d\mathcal H^{k}(x_1,y)\,,
    \end{split}
    \]
    whereas 
    \[
    \int_U |x|^2 \tilde\eta(|x|)\psi(y,t) d\|\bC\| = \int_{\bC\cap U} |x_1|^2\tilde\eta(|x_1|)\psi(y,t)\,d\mathcal H^k(x_1,y)\,.
    \]
    We can then estimate the difference of these quantities in terms of 
    integral of $|f|^2$ and $|\nabla f|^2$. On $|x|\geq 1/2$, by 
    Proposition \ref{e-reg}, we have
    \begin{equation} \label{I0a}
    \begin{split}
    &\left| \int_{\{|x|\geq 1/2\}} |x|^2 \tilde\eta(|x|)\psi(y,t) d\|\bC\| - \int_{\{|x|\geq 1/2\}} |x|^2 \tilde\eta(|x|)\psi(y,t) d\|V_t\| \right| \\
    & \qquad \qquad \qquad \leq \Cr{linestimate}^2\max\{\mu,\|u\|\}^2\,.
    \end{split}
    \end{equation}
   
   Coming to the part in $\{|x|<1/2\}$, we see that this lies within the region where $\tilde\eta=1$. Thus (as $K(y,t,0)$ in \eqref{anasup} with $\xi=0$ is bounded by $\Cr{linestimate}\max\{\mu,\|u\|\}$), by \eqref{cotriple2}, for $\mathcal H^{k-1}$-a.e. $y \in B^{n-k+1}_1$ it holds
    \[
    \int_{\{|x|<1/2\}\cap\hat \bC} \tilde \eta(|x|) |x|^2 \, d\mathcal H^{1} \leq \int_{M_t^y\cap \{|x|<1/2\}} \tilde\eta(|x|) |x|^2 \, d\mathcal H^{1} + \Cr{conslice}\Cr{linestimate}^2\max\{\mu,\|u\|\}^2\,,
    \]
    and thus, by the coarea formula, 
    \begin{equation} \label{I0b}
    \begin{split}
  &  \int_{\{|x|<1/2\}} \tilde\eta(|x|) \psi(y,t) |x|^2 \, d\|\bC\| - \int_{\{|x|<1/2\}} \tilde\eta(|x|) \psi(y,t) |x|^2 \, d\|V_t\| \\
  & \qquad \qquad \qquad \leq \Cr{conslice}\Cr{linestimate}\max\{\mu,\|u\|\}^2\,.
  \end{split}
    \end{equation}
    By combining \eqref{I0a} and \eqref{I0b}, we obtain the estimate
    \begin{equation}
        2 I(t) \leq \Cr{linestimate}^2(\Cr{conslice}+1) \max\{\mu,\|u\|\}^2 \qquad \mbox{for a.e. $t \in [-4,0]$}\,.
    \end{equation}
    Since $\|V_t\|$ is left-continuous with respect to $t$ as Radon measures, the estimate holds for every $t \in [-4,0]$, and in particular the same estimate hold also for $2I_0$. 
\end{proof}

The technical work done in Propositions \ref{lateral-pro} to \ref{p:error0} is used here to show that the square of the excess controls both the space-time $L^2$ norm squared of the mean curvature and Huisken's integral appearing on the left-hand side of \eqref{nocon1}, in a parabolic neighborhood of a high-density point.

\begin{proposition} \label{p:error_t_est}
    Let $\eta$ be as in Proposition \ref{lateral}
    and assume $\Theta(0,0)\geq \sfrac32$. Then there exists $\Cl[con]{cdis}=\Cr{cdis}(n,k,p,q,\Cr{conslice},E_1)>0$ such that
    \begin{equation} \label{error estimates}
       \int_{-2}^0 \int \eta \left|h\right|^2\,d\|V_t\|\,dt+ \int_{-2}^0\int\eta\rho\Big|h-\frac{(\nabla\rho)^\perp}{\rho}\Big|^2\,d\|V_t\|dt\leq 
        \Cr{cdis}\max\{\mu,\|u\|\}^2.
    \end{equation}
\end{proposition}
\begin{proof}
We first show how to bound the second integral on the left-hand side; the bound on the first integral is simpler, and will be addressed at the end.

We use Proposition \ref{lateral} and estimate
the right-hand side of \eqref{nocon1}.
Let $\eta_1\in C^\infty(\R)$ be a function
    such that $\eta_1(t)=0$ for $t\notin [-5/2,-1/2]$, $\eta_1(t)=1$ for $t\in[-2,-1]$, 
    $0\leq \eta_1(t)\leq 1$ and $|\eta_1'(t)|\leq 3$ for all $t$. We set
    $\psi(x,y,t)=\rho(x,y,t)\eta(x,y)\eta_1(t)$ 
    and note that $\psi$ satisfies all the 
    assumption for Proposition \ref{lat4}.
    Thus we have
    \begin{equation}\label{lat5}
        \int^0_{-4}\Big(\int \rho\eta\eta_1\,d\|V_t\|-\int\rho\eta\eta_1\,d\|\bC\|\Big)_+ dt\leq \Cr{er1}\max\{\mu,\|u\|\}^2.
    \end{equation}
    Since $\eta_1(t)=1$ on $t\in[-2,-1]$, 
    there exists some $s\in[-2,-1]$ such that
    the integrand of \eqref{lat5} is bounded 
    by the same constant,
    that is, 
    \begin{equation}
        \int\rho\eta\,d\|V_s\|-\int\rho\eta\,d\|\bC\|\leq \Cr{er1}\max\{\mu,\|u\|\}^2.
    \end{equation}
    We fix this $s$. Next let $\eta_2\in C^{\infty}(\R)$ be a function such that
    $\eta_2(t)=1$ for $t\geq -2$, $\eta_2(t)=0$
    for $t\leq -3$, $0\leq \eta_2(t)\leq 1$ and
    $|\eta_2'(t)|\leq 2$ for all $t$. We set
    $\psi(x,y,t)=\nabla\rho(x,y,t)\cdot\nabla \eta(x,y)\eta_2(t)$ and note that this $\psi$
    also satisfies the assumption for Proposition 
    \ref{lat4}, particularly since $\nabla\rho\rightarrow 0$ on the support of 
    $\nabla\eta$ as $t\rightarrow 0-$. Thus 
    we obtain \eqref{kore} for $\psi$, and 
    by restricting the integral to $[s,0]$ 
    where $\eta_2=1$, we have
\begin{equation}
    \int_{s}^0\Big(\int \nabla\rho\cdot\nabla\eta\,d\|V_t\|-\int\nabla\rho\cdot\nabla\eta\,d\|\bC\|\Big)\,dt\leq \Cr{er1}\max\{\mu,\|u\|\}^2.
\end{equation}
Next, we choose and fix 
a non-negative smooth function
$\eta_3\in C^\infty(\R)$ 
such that $(\tilde\eta')^2/\tilde\eta\leq \eta_3$,
and which vanishes outside of $[1/4,5/4]$. Note that
$\tilde\eta'=0$ outside of $[1/2,1]$ so we may choose such a function. 
We then set 
\begin{equation}
    \psi(x,y,t)=\rho(x,y,t)\eta_2(t) \eta_3(|y|)\tilde\eta(|x|).
\end{equation}
With this choice of $\psi$ in Proposition \ref{lat4}, 
we obtain 
\begin{equation}\label{lat7}
    \iint_{P_2}\eta_2(t)
    \eta_3(|y|)\tilde\eta(|x|)\sum_{j=n-k+2}^n|S^\perp e_j|^2\rho\,dV_tdt\leq \Cr{er1}\max\{\mu,\|u\|\}^2.
\end{equation}
Now 
\begin{equation}\label{lat8}
    \frac{|(\nabla\eta)^\perp|^2}{\eta}\rho
    =\frac{(\tilde\eta'(|x|))^2}{\tilde\eta(|x|)}\tilde\eta(|y|)\frac{|S^\perp x|^2}{|x|^2}\rho+
    \frac{(\tilde\eta'(|y|))^2}{\tilde\eta(|y|)}\tilde\eta(|x|)\frac{|S^\perp y|^2}{|y|^2}\rho.
\end{equation}
The first term of \eqref{lat8} 
vanishes for $|x|\leq 1/2$, so that the integral 
is only over the graphical part of $V_t$,
and $\rho$ is bounded by an absolute
constant. 
By using \eqref{Ses}, it is bounded by a constant multiple of $\max\{\mu,\|u\|\}^2$.
For the second term of \eqref{lat8}, we have
\begin{equation}\label{lat9}
 \frac{(\tilde\eta'(|y|))^2}{\tilde\eta(|y|)}\tilde\eta(|x|)\frac{|S^\perp y|^2}{|y|}\rho   \leq \eta_2(t)\eta_3(|y|)\tilde\eta(|x|)\sum_{j=n-k+2}^{n}|S^\perp e_j|^2\rho
\end{equation}
for $-2\leq t<0$ due to the choice of $\eta_2$ 
and $\eta_3$. Combined with \eqref{lat7}, 
the integral of the second term is also 
bounded similarly. Finally, the term
involving $u$ is bounded in terms of 
$\|u\|^2$ and $E_1$ by the H\"{o}lder
inequality (see \cite[Proposition 6.2]{Kasai-Tone}). Thus all the terms on the right-hand side of \eqref{nocon1} are bounded by a constant multiple of
$\max\{\mu,\|u\|\}^2$ and this concludes the proof of the estimate for the second integral on the left-hand side of \eqref{error estimates}.

We finally come to the first integral. With the same choice of $\eta_1$, we set now $\psi(x,y,t)=\eta(x,y)\eta_1(t)$, and we note that this $\psi$ also satisfies all the assumptions for Proposition \ref{lat4}, so that
\begin{equation} \label{diss1}
    \int_{-4}^{0} \left( \int \eta \eta_1\,d\|V_t\| - \int \eta\eta_1\,d\|\bC\| \right)_+ \, dt \leq \Cr{er1}\max\{\mu,\|u\|\}^2\,,
\end{equation}
and since $\eta_1(t) =1$ for all $t \in [-2,-1]$ there exists $s\in [-2,-1]$ such that 
\begin{equation} \label{diss2}
    \int \eta \,d\|V_s\| - \int \eta\,d\|\bC\|\leq \Cr{er1}\max\{\mu,\|u\|\}^2\,.
\end{equation}
We now test Brakke's inequality \eqref{e:Brakke-ineq} with $\phi=\eta(x,y)$, with $t_1=s$ and $t_2=0$. This yields
\begin{equation} \label{diss3}
    \int_{s}^{0}\int \eta |h|^2\,d\|V_t\|\,dt \leq  \int \eta \,d\|V_s\| - \int \eta \,d\|V_{0}\| 
    +\int_{s}^{0}\int \nabla\eta\cdot(h+u^\perp) \,d\|V_t\|\,dt\,,
\end{equation}
which then immediately gives, using \eqref{diss2} 
\begin{equation} \label{diss4}
\begin{split}
&    \frac12 \int_{s}^0\int\eta |h|^2\,d\|V_t\|\,dt \\
& \qquad \leq \int \eta \,d\|\bC\| -\int \eta \,d\|V_0\| + \int_{s}^0 \int \frac{|\nabla^\perp\eta|^2}{\eta} + \frac12\eta|u|^2 + \Cr{er1}\max\{\mu,\|u\|\}^2\,.
\end{split}
\end{equation}
The difference 
\[
\int\eta\,d\|\bC\|-\int\eta\,d\|V_0\|
\]
is bounded by $\Cr{lmin}\max\{\mu,\|u\|\}^2$ by Proposition \ref{san-eq}, and the term involving $u$ is bounded by $c(E_1)\|u\|^2$ by H\"older's inequality. For the remaining term, we proceed as in \eqref{lat8} with $\rho\equiv 1$. Again the first summand vanishes for $|x|\leq 1/2$, so that the integral is only over the graphical part of $V_t$, and it is bounded by a constant multiple of $\max\{\mu,\|u\|\}^2$ thanks to \eqref{Ses} and \eqref{para1-ex}. The second summand is bounded as in \eqref{lat9}, namely
\[
\frac{(\tilde\eta'(|y|))^2}{\tilde\eta(|y|)}\tilde\eta(|x|)\frac{|S^\perp y|^2}{|y|}  \leq \eta_2(t)\eta_3(|y|)\tilde\eta(|x|)\sum_{j=n-k+2}^{n}|S^\perp e_j|^2
\]
whenever $-2 \leq t < 0$. Now, Proposition \ref{p:error0} guarantees that we also have the estimate
\[
\int_{-2}^0\int \eta_2(t)\eta_3(|y|)\tilde\eta(|x|)\sum_{j=n-k+2}^{n}|S^\perp e_j|^2 \,dV_t\,dt \leq \Cr{er0}\max\{\mu,\|u\|\}^2\,,
\]
and this completes the proof.
\end{proof}

In turn, as we show in the next proposition, at a high density point the space-time integral quantities on the left-hand side of \eqref{error estimates} control the $L^2$-distance of the flow from the triple junction \emph{uniformly in time}. This passage from an integral-in-time control to a pointwise-in-time control is a key feature of parabolic estimates. The proposition provides two distinct but complementary bounds.

The first, \eqref{e:HS_technical}, is parabolic or ``caloric'' in nature. When combined with Proposition \ref{p:error_t_est}, it yields the decay estimate
\[
(\max\{\mu,\|u\|\})^{-2}\,\int \dist^2(X,\bC) \,\rho(X,t) \eta(X)\,d\|V_t\|(X) \lesssim |t|^\kappa\,.
\]
The factor $|t|^{\kappa}$ forces the \emph{weighted} excess to vanish at a H\"older rate as $t \to 0^-$. This information, however, comes at the price of spatial localization due to the Gaussian weight $\rho$, which effectively confines the estimate to a parabolic neighborhood of the origin of scale $\sqrt{|t|}$.

The second estimate, \eqref{e:HS_technical2}, is more elliptic in character. It provides a uniform-in-time bound on the \emph{unweighted} $L^2$-distance over a fixed ball. It exchanges the time decay rate of the first estimate for more robust spatial control. Both types of estimate are essential for the subsequent proof of Theorem \ref{thm:HS}, which relies on both the fine decay structure near the singularity and the uniform control away from it.

\begin{proposition} \label{p:HS_technical}
Let $\eta$ be as in Proposition \ref{lateral}
    and assume $\Theta(0,0)\geq \sfrac32$. Then,
    for each $\kappa\in [0,1)$, there exists $\Cl[con]{c_nc1}=\Cr{c_nc1}(n,k,p,q,E_1,\kappa)>0$ such that
\begin{equation} \label{e:HS_technical}
\begin{split}
    &\sup_{t \in [-1,0)}|t|^{-\kappa} \int \dist^2(X,\mathbf{C}) \, \rho(X,t)\eta(X)\, d\|V_t\|(X) \\
    & \qquad \leq \Cr{c_nc1}\Big(\int_{-2}^{0}\int \Big| h - \frac{(\nabla\rho)^\perp}{\rho} \Big|^2\,\rho(X,t)\eta(X)\,d\|V_t\|\,dt + \max\{\mu,\|u\|\}^2\Big) \,,
    \end{split}
\end{equation}
and
\begin{equation} \label{e:HS_technical2}
\begin{split}
    &\sup_{t \in [-1,0)} \int \dist^2(X,\mathbf{C}) \, \eta(X)\, d\|V_t\|(X) \\
    & \qquad \leq \Cr{c_nc1}\Big(\int_{-2}^{0}\int |h|^2\,\eta(X)\,d\|V_t\|\,dt + \max\{\mu,\|u\|\}^2\Big) \,.
    \end{split}
\end{equation}
\end{proposition}

\begin{proof}
We begin with the proof of \eqref{e:HS_technical}: \eqref{e:HS_technical2} is simpler. We proceed mainly as in \cite[Proposition 5.2]{tone-wic} with some modifications. Let ${\rm d}\colon \R^n \to \R$ be a function satisfying the following properties: ${\rm d}$ is positively homogeneous of degree one, it is smooth away from $\bC$, and furthermore, writing as usual $X = (x,y)$,
\begin{align}
    & {\rm d}(X) = \dist(X,\bC) \qquad \qquad  \quad \qquad \qquad \mbox{$\forall\,X$ with $\dist(X,\bC) < \frac{|x|}{5}$}\,, \label{d1} \\
    & \frac12 \,\dist(X,\bC) \leq {\rm d}(X) \leq 2\,\dist(X,\bC) \quad \quad \forall\,X\in\R^n\,, \label{d2}\\ \label{d3}
    & |\nabla {\rm d}(X)| \leq 1 \qquad \qquad  \quad \qquad \qquad \qquad\quad \forall\,X \notin \bC\,.
\end{align}
By homogeneity, we have $X\cdot\nabla({\rm d}^2/|X|^2)=0$ and this leads to 
\begin{equation}\label{dhom}
    X\cdot\nabla{\rm d}^2=2{\rm d}^2.
\end{equation}

By the definition of $\mu$, we can choose $t_1 \in [-2,-1]$ such that
\begin{equation}\label{dhom2}
    \int \eta(X)\,\dist^2(X,\bC) \, d\|V_{t_1}\|(X) \leq \mu^2\,.
\end{equation}
We can then fix a smooth function $g=g(t)$ with 
\begin{equation}\label{ginq}
0 < g(t) \leq 1
\end{equation}
and test Brakke's inequality \eqref{e:Brakke-ineq} with
\begin{equation} \label{e:test1}
\phi(X,t) := |t|^{-\kappa}\, g(t) \,\rho_{(0,0)}(X,t)\, {\rm d}(X)^2\,\eta(X)
\end{equation}
for $t \in [t_1,t_2]$, and for arbitrary $t_2 \in [-1,0)$. For notational convenience, we denote $\rho(X,t) := \rho_{(0,0)}(X,t)$ and $\hat\rho (X,t):= |t|^{-\kappa}\, g(t) \,\rho(X,t)$, so that $\phi = \hat\rho\eta{\rm d^2}$ and $\eta{\rm d}^2$ is independent of the variable $t$. Thus, \eqref{e:Brakke-ineq} yields:
\begin{equation}\label{HS_calc1}
    \left.\int \hat\rho\eta{\rm d}^2 \, d\|V_t\|\right|_{t=t_1}^{t_2} \leq \int_{t_1}^{t_2} \int\left\lbrace\left( -\hat\rho \eta {\rm d}^2 h + \nabla (\hat\rho \eta {\rm d}^2) \right) \cdot \left( h + u^\perp \right) + \eta {\rm d}^2 \frac{\pa \hat\rho}{\pa t} \right\rbrace \,d\|V_t\|\,dt\,.
\end{equation}
By direct calculation, 
\begin{equation}
\begin{split}
    (-h\hat\rho&\eta{\rm d}^2+\nabla(\hat\rho\eta{\rm d}^2))\cdot(h+u^\perp) \\
    &=(-|h|^2\hat\rho+(\nabla\hat\rho\cdot h))\eta{\rm d}^2+\hat\rho\nabla(\eta{\rm d}^2)
    \cdot h+\eta{\rm d}^2(-h\hat\rho+\nabla\hat\rho)\cdot u^\perp+\hat\rho\nabla(\eta{\rm d}^2)\cdot u^\perp \\
    &\leq -\hat\rho\big|h-\frac{(\nabla\hat\rho)^\perp}{\hat\rho}\big|^2 \eta{\rm d}^2-(\nabla\hat\rho\cdot h)\eta{\rm d}^2+\frac{|(\nabla\hat\rho)^\perp|^2}{\hat\rho}\eta{\rm d}^2+\hat\rho\nabla(\eta{\rm d}^2)\cdot h\\
    &\quad\,+\frac12\hat\rho\big|h-\frac{(\nabla\hat\rho)^\perp}{\hat\rho}\big|^2\eta{\rm d}^2+\frac12 \hat\rho\eta{\rm d}^2|u|^2+\hat\rho\nabla(\eta{\rm d}^2)\cdot u^\perp,
    \end{split}
\end{equation}
and it follows from \eqref{HS_calc1} that
\begin{equation}\label{dhom3}
    \begin{split}
    \int{\rm d}^2\eta\hat\rho\,d\|V_t\|\Big|_{t=t_1}^{t_2}\leq \int_{t_1}^{t_2}\int& -(\nabla\hat\rho\cdot h)\eta{\rm d}^2+\frac{|(\nabla\hat\rho)^\perp|^2}{\hat\rho}\eta{\rm d}^2+\hat\rho\nabla(\eta{\rm d}^2)\cdot h 
    \\ &+\frac12 \hat\rho\eta{\rm d}^2|u|^2+\hat\rho\nabla(\eta{\rm d}^2)\cdot u^\perp+\eta{\rm d}^2\,\frac{\partial\hat\rho}{\partial t}\,d\|V_t\|dt.
    \end{split}
\end{equation}
For a.e.~$t$, we have
\begin{equation}\label{dhom4}
    \int -(\nabla\hat\rho\cdot h)\eta{\rm d}^2\,d\|V_t\|=\int S\cdot(\eta{\rm d}^2\nabla^2\hat\rho+\nabla\hat\rho\otimes
    \nabla(\eta{\rm d}^2))\,dV_t(\cdot, S).
\end{equation}
Using \eqref{dhom4} in \eqref{dhom3}
as well as \eqref{e:Huisken_id}, we obtain
\begin{equation}\label{dhom10}
\begin{split}
 \int\hat\rho\eta{\rm d}^2\,d\|V_t\|\Big|_{t=t_1}^{t_2}&\leq \int_{t_1}^{t_2}\int  S\cdot(\nabla\hat\rho\otimes\nabla (\eta{\rm d}^2))+\hat\rho\nabla(\eta{\rm d}^2)\cdot h 
 +\frac12\hat\rho\eta{\rm d}^2|u|^2 \\
 &+\hat\rho\nabla(\eta{\rm d}^2)\cdot u^\perp
 +\eta{\rm d}^2\rho\frac{d}{dt}(|t|^{-\kappa}g(t))\,dV_t(\cdot,S)dt\\
 &=I_1+I_2+I_3+I_4+I_5. 
 \end{split}
\end{equation}
{\bf Estimate of $I_1+I_2$}.
\newline
Since $S=I-S^\perp$, 
\begin{equation}\label{dhom5}
    \begin{split}
S\cdot&(\nabla\hat\rho\otimes\nabla(\eta{\rm d}^2))+\hat\rho\nabla(\eta{\rm d}^2)\cdot h=
\nabla\hat\rho\cdot\nabla(\eta{\rm d}^2)-(\nabla\hat\rho)^\perp\cdot\nabla(\eta{\rm d}^2)+\hat\rho\nabla(\eta{\rm d}^2)\cdot h\\
&=\nabla\hat\rho\cdot\nabla(\eta{\rm d}^2)+\hat\rho\nabla(\eta{\rm d}^2)\cdot\big(h-\frac{(\nabla\hat\rho)^\perp}{\hat\rho}\big)
\\
&\leq -\frac{\hat\rho}{2|t|}X\cdot\nabla(\eta{\rm d}^2)+\frac{\rho}{2}\big|h-\frac{(\nabla\rho)^\perp}{\rho}\big|^2\eta+\frac{g(t)|t|^{-\kappa}}{2\eta} |\nabla(\eta{\rm d}^2)|^2\,\hat\rho.
    \end{split}
\end{equation}
The terms involving $\nabla\eta$ is non-zero
only on $(B_1^{n-k+1}\times B_1^{k-1})
\setminus (B_{\sfrac12}^{n-k+1}\times B_{\sfrac12}^{k-1})$, and $|t|^{-1}\hat\rho$ is uniformly bounded 
by a constant on the domain. Using also 
\eqref{ginq}, \eqref{d1}-\eqref{dhom}
in \eqref{dhom5},
we obtain
\begin{equation}\label{dhom6}
    I_1+I_2\leq \int_{t_1}^{t_2}\int
    -\frac{\hat\rho\eta {\rm d}^2}{|t|}
    +\frac{\rho}{2}\big|h-\frac{(\nabla\rho)^\perp}{\rho}\big|^2
    \eta+4|t|^{-\kappa} {\rm d}^2\hat\rho\eta\,d\|V_t\|dt+c(n,k)\mu^2.
\end{equation}
{\bf Estimate of $I_3$}.
\newline
We separate the integration into two 
regions, $A_1(t)=\{X\in \mathbb R^n : |X|\leq |t|^{\kappa/2}\}$
and the complement $A_2(t)=\mathbb R^{n}\setminus A_1(t)$. On $A_1(t)$, 
${\rm d}(X)\leq 2\,{\rm dist}(X,\bC)\leq 2|t|^{\kappa/2}$ by \eqref{d2}, so
that $\hat\rho {\rm d}^2\leq 4\rho$.
Thus,
\begin{equation}
\int_{t_1}^{t_2}\int_{A_1(t)}\hat\rho\eta{\rm d}^2|u|^2\,d\|V_t\|dt
    \leq 4\int_{t_1}^{t_2}\int \rho\eta|u|^2\,d\|V_t\|dt\leq c(p,q,E_1)\|u\|^2.
\end{equation}
On $A_2(t)$, $\hat\rho$ is uniformly bounded by a constant that depends only on 
$k$ and $\kappa$. Thus the 
integral over $A_2(t)$ is similarly estimated and we have
\begin{equation} \label{dhom7}
    I_3\leq c(k,p,q,E_1,\kappa)\|u\|^2. 
\end{equation}
\newline
{\bf Estimate of $I_4$}. 
\newline
\begin{equation}
    I_4\leq\int_{t_1}^{t_2}\int \hat\rho(|\nabla \eta||u|{\rm d}^2+2\eta|u|{\rm d})\,d\|V_t\|dt,
\end{equation}
and since $\hat\rho$ is bounded on ${\rm spt}|\nabla\eta|$ and and $|u|{\rm d}^2\leq |u|^2+{\rm d}^4$, the first term can be
bounded by $c(\mu^2+\|u\|^2)$. Also since
$2\eta|u|{\rm d}\hat\rho\leq \eta|u|^2\rho+|t|^{-\kappa}\hat\rho{\rm d}^2\eta$, we have
\begin{equation}\label{dhom8}
    I_4\leq \int_{t_1}^{t_2}\int |t|^{-\kappa} \hat\rho{\rm d}^2\eta\,d\|V_t\|dt+c(n,k,p,q,E_1,\kappa)(\mu^2+\|u\|^2).
\end{equation}
{\bf Estimate of $I_5$}.
\newline
We make the explicit choice of $g$ given by 
\begin{equation}
    g(t)=\exp\big(-5\int_{t_1}^t |s|^{-\kappa}\,ds\big)
\end{equation}
for $t\in[t_1,0)$. Note that $g(t)\leq 1$ 
and 
\begin{equation}
    \inf_{t\in[t_1,0)} g(t)=\exp\big(-\frac{5|t_1|^{1-\kappa}}{1-\kappa}\big)
\end{equation}
and since $t_1\in[-2,-1]$, $g(t)$ is 
bounded from below by a positive constant 
depending only on $\kappa$. The function 
$g$ is chosen so that 
\begin{equation}
    \rho\frac{d}{dt}((|t|^{-\kappa}g(t))=
    \frac{\kappa\hat\rho}{|t|}-5|t|^{-\kappa}\hat\rho,
\end{equation}
and thus 
\begin{equation}\label{dhom9}
    I_5\leq \int_{t_1}^{t_2} \int\kappa\frac{\hat\rho\eta{\rm d}^2}{|t|}-5|t|^{-\kappa}\hat\rho{\rm d^2}\eta\,d\|V_t\|dt.
\end{equation}
Finally, add \eqref{dhom6}, \eqref{dhom7}, \eqref{dhom8} and \eqref{dhom9} to estimate
the left-hand side of \eqref{dhom10}.
Then use \eqref{dhom2} and the lower bound 
of $g(t)$ to obtain, with a suitable 
choice of $\Cr{c_nc1}$, the inequality \eqref{e:HS_technical}.
\newline
{\bf Proof of \eqref{e:HS_technical2}.}
\newline
We define the function ${\rm d}$ and the initial time $t_1$ as in the proof of \eqref{e:HS_technical}. Then, we test Brakke's inequality with
\begin{equation} \label{e:test2}
    \phi(X,t) :={\rm d}(X)^2\eta(X)\,.
\end{equation}
We then have $\hat\rho \equiv 1$ in the subsequent calculations, and we rapidly see that for arbitrary $t_2 \in [-1,0)$ it holds
\[
\left.\int {\rm d}^2 \eta \, d\|V_t\| \right|_{t=t_1}^{t_2} \leq \int_{t_1}^{t_2} \int\left\lbrace \nabla(\eta{\rm d}^2)\cdot h + \frac12 \eta{\rm d}^2|u|^2+\nabla(\eta{\rm d}^2)\cdot u^\perp \right\rbrace\, dV_t(\cdot,S)\,dt
\]
in place of \eqref{dhom10}. Using \eqref{d2}-\eqref{dhom2}, we then estimate
\[
\int \dist^2(X,\bC)\eta(X)\,d\|V_{t_2}\|\leq 4\mu^2 + \int_{t_1}^{t_2}\int \left\lbrace\eta |h|^2 + 2\frac{|\nabla(\eta{\rm d}^2)|^2}{\eta} + \eta{\rm d}^2 |u|^2 + \eta |u|^2\right\rbrace\,d\|V_t\|dt\,,
\]
and, since ${\rm d}$ is bounded by a constant on $\spt(\eta)$ and \eqref{d3} holds, the last three summands are bounded by $c(k,n,p,q,E_1) \max\{\mu,\|u\|\}^2$, thus completing the proof.
\end{proof}

\subsection{Proof of Theorem \ref{thm:HS}}
We can now finally come to the proof of Theorem \ref{thm:HS}. Before that, we isolate the following simple remark. From this point onwards, we introduce the following notation. Given $\Xi\in\R^n$ and $\lambda > 0$, we define $\iota_{\Xi,\lambda}(X):=\lambda^{-1}(X-\Xi)$, and we also set $\iota_{\Xi,1}=:\iota_\Xi$. Furthermore, given a flow $\left(\{V_t\}_t,\{u(\cdot,t)\}_t\right)$, a point $(\Xi,\tau)$ in space-time, and $\lambda > 0$, we also define the translated and rescaled flow $\left( \{V^{(\Xi,\tau),\lambda}_t\}_t, \{u^{(\Xi,\tau),\lambda}(\cdot,t)\}_t \right)$ by setting $V^{(\Xi,\tau),\lambda}_t :=(\iota_{\Xi,\lambda})_\sharp V_{\tau + \lambda^2t}$ and $u^{(\Xi,\tau),\lambda}(X,t):=u(\Xi+\lambda X,\tau+\lambda^2 t)$. As customary, we omit the index $\lambda$ when $\lambda=1$.

\begin{lemma}\label{l:translation}
    Under the assumptions of Theorem \ref{thm:HS}, upon choosing $\Cr{e_HS}$ sufficiently small
    depending only on $n,k,p,q,E_1$, the assumptions in Subsection \ref{preest} are satisfied
    for the flow $(\{V^{(\Xi,\tau)}_t\}_t, \{u^{(\Xi,\tau)}(\cdot,t)\}_t)$ on $U_4 \times (-16,0]$. In particular, Propositions \ref{p:error_t_est} and \ref{p:HS_technical} hold for this flow.
\end{lemma}
\begin{proof}
    The assumptions (A1)-(A6) are automatically
    satisfied for $(\{V^{(\Xi,\tau)}_t\}_t, \{u^{(\Xi,\tau)}(\cdot,t)\}_t)$. Thus we only need to prove that 
    $(\{V^{(\Xi,\tau)}_t\}_t, \{u^{(\Xi,\tau)}(\cdot,t)\}_t)\in \mathscr N_{\varepsilon}(U_4\times(-16,0])$ for
    $\varepsilon=\min\{\Cr{onlyunit},
\Cr{eps-tau}\}$ if $(\{V_t\}_t,\{u(\cdot,t)\}_t)\in \mathscr N_{\Cr{e_HS}}(U_5\times(-25,0])$. 
    We use Theorem \ref{thm:graphical},
    with $\beta=1/40$ and $\sigma=c\min\{\Cr{onlyunit},
\Cr{eps-tau}\}$ with small $c$ (to be chosen depending only on $k$ and $E_1$) and obtain a new $\overline{\Cr{eps-tau}}$ ($\overline\cdot$ to
indicate the new $\Cr{eps-tau}$, depending on this choice of $\sigma$). If $(\{V_t\},\{u(\cdot,t)\})\in\mathscr N_{\overline{\Cr{eps-tau}}}
(U_5\times(-25,0])$, then by the conclusion of Theorem \ref{thm:graphical}, we have (writing $\Xi=(\xi,\zeta)$ so that $|\xi|=\dist(\Xi,\bS(\bC))$) that
$|\xi|\leq 5\sigma=5c\min\{\Cr{onlyunit},
\Cr{eps-tau}\}$. We have
\begin{equation}\label{xiquant}
\begin{split}
    &\iint_{P_4}
    {\rm dist}(X,\bC)^2\,d\|V_t^{(\Xi,\tau)}\|(X)\,dt\Big)^{1/2}\\
    &=\Big(4^{-k-4}
    \iint_{P_4(\Xi,\tau)}
    {\rm dist}(X-\Xi,\bC)^2\,d\|V_t\|(X)\,dt\Big)^{1/2}\\ &\leq \Big( 4^{-k-4}\iint_{P_5}({\rm dist}(X,\bC)+|\xi|)^2\,d\|V_t\|(X)dt\Big)^{1/2} \\
    &\leq 2\, \big(\frac54\big)^{\frac{(k+4)}{2}}\Big(5^{-k-4}
    \iint_{P_5}{\rm dist}(X,\bC)^2\,d\|V_t\|dt\Big)^{1/2}+2^{-k-3}25(E_1 5^k\omega_k)^{1/2}c\min\{\Cr{onlyunit},
\Cr{eps-tau}\},
    \end{split}
\end{equation}
and if we set $c=2^{k+2}(25)^{-1}(E_1 5^k \omega_k)^{-1/2}$ (which fixes $\overline{\Cr{eps-tau}}$) and set
\[\Cr{e_HS}=\min\{\overline{\Cr{eps-tau}},
2^{-2}(4/5)^{(k+4)/2}\min\{\Cr{onlyunit},
\Cr{eps-tau}\}\},\] 
then we see that the left-hand side of 
\eqref{xiquant}
is $\leq \min\{\Cr{onlyunit},
\Cr{eps-tau}\}$. This gives \eqref{def:L2-excess} for $V^{(\Xi,\tau)}_t$. The inquality \eqref{smuterm} for $u^{(\Xi,\tau)}$ is achieved by
restricting $\Cr{e_HS}$ depending only on
$\alpha$, and \eqref{def:mass deficit}
and \eqref{non0} can be achieved by using 
the graphical representation and restricting
$\Cr{e_HS}$ if necessary depending on 
$\min\{\Cr{onlyunit},
\Cr{eps-tau}\}$. This ends the proof.
\end{proof}

We are ready to prove Theorem \ref{thm:HS}.

\begin{proof}[Proof of Theorem \ref{thm:HS}]
By Lemma \ref{l:translation}, we can apply Propositions \ref{p:error_t_est} and \ref{p:HS_technical} with $\kappa=\sfrac{1}{2}$ to the flow $\left( \{V_t^{(\Xi,\tau)}\}_t , \{u^{(\Xi,\tau)}(\cdot,t)\}_t\right)$  to conclude that 
\begin{equation} \label{fHS1}
\begin{split}
&    \sup_{t\in[-1+\tau,\tau)} (\tau-t)^{-\sfrac{1}{2}} \int \dist^2(X-\Xi,\mathbf C)\,\rho_{(\Xi,\tau)}(X,t)\,\eta(X-\Xi)\,d\|V_t\|(X) \\ 
& \qquad \qquad \qquad \leq  \Cl[con]{c_nc2}\,\max\{\mu^{(\Xi,\tau)},\|u^{(\Xi,\tau)}\|\}^2\,,
\end{split}
\end{equation}
where $\Cr{c_nc2}:=\Cr{c_nc1}(\Cr{cdis}+1)$, $\|u^{(\Xi,\tau)}\|$ is defined via integration over $B_4\times(-16,0)$, and
\begin{equation}\label{fHSadd}
4^{k+4}(\mu^{(\Xi,\tau)})^2 = \iint_{P_4(\Xi,\tau)}\dist^2(X-\Xi,\mathbf C) \, d\|V_t\|(X)\,dt\,.
\end{equation}
On the other hand, since $\bC$ is invariant with respect to translations along vectors in $\mathbf{S}(\mathbf{C})$, we have that $\dist(X-\Xi,\bC) \leq |\xi| + \dist(X,\bC)$. Combined with \eqref{fHSadd}, we have
\begin{equation} \label{fHS3}
(\mu^{(\Xi,\tau)})^2 \leq C\mu^2 + C|\xi|^2\,,
\end{equation}
where $C>0$ is a constant depending only on $k$ and $E_1$. 
Now, we have that
\begin{equation} \label{fHS2}
\|u^{(\Xi,\tau)}\| \leq \|u\|
\end{equation}
if we define $\|u\|$ via integration over $P_5$. 
Combining \eqref{fHS1} with \eqref{fHS3} and \eqref{fHS2}, and using the properties of the function $\eta$, we have that for every $t \in \left[-1+\tau,\tau\right)$
\begin{equation} \label{fHS4}
    \int_{B_{\frac12}(\Xi)} \dist^2(X-\Xi,{\bf C}) \, \rho_{(\Xi,\tau)}(X,t) \,d\|V_t\|(X) \leq \Cl[con]{c_nc3}(\tau-t)^{\sfrac12} (\mu^2 + \|u\|^2 + |\xi|^2)\,.
\end{equation}
Now, let us fix $0 < r_0< \sfrac14$ to be fixed later, and let $\Cr{e_HS}$ be smaller than $\Cr{eps-tau}(k,n,p,q,E_1,\beta,\sfrac{r_0}{8})$ with $\beta=1/40$, where $\Cr{eps-tau}$ is the threshold of Theorem \ref{thm:graphical}. This, in particular, guarantees that $|\xi| \leq \sfrac{r_0}{4}$. Also choose $t_1 \in \left[\tau-2r_0^2,\tau-r_0^2\right] \subset \left[-\sfrac{5}{4},\tau\right)$ such that
\begin{equation}\label{good-time}
    \int_{U_5}\dist^2(X,\bC) \,d\|V_{t_1}\|(X)\leq \frac{C}{r_0^2}\,\mu^2\,.
\end{equation}
Now, choose $i\in\{1,2,3\}$ so that the half-plane $\mathbf{H}_{i}$ (see \eqref{halfpdef}) maximizes the quantity $|\mathbf{p}_{\mathbf{P}_{j}^\perp}(\Xi)|=|\mathbf{p}_{\mathbf{P}_{j}^\perp}(\xi)|$. Here, $\mathbf{p}_{\mathbf{P}_{j}^\perp}$ denotes the orthogonal projection operator onto the linear subspace orthogonal to the plane $\mathbf P_j$ containing $\mathbf H_{j}$. Notice that $|\mathbf{p}_{\mathbf{P}_{i}^\perp}(\xi)| \geq \frac{\sqrt{3}}{2} |\xi|$. Also, if $Z \in \mathbf{H}_i$ is such that $\dist(Z,\mathbf{S}(\bC)) > |\xi|$ then 
\begin{equation}\label{geometric}
\dist(Z-\Xi,\bC)=|\mathbf{p}_{\mathbf P_i^\perp}(\xi)| \geq \frac{\sqrt{3}}{2} |\xi|\,.
\end{equation}
Without loss of generality, assume that $\mathbf H_{i}=\mathbf H_{1}=[0,\infty)\times \{0_{n-k}\} \times \R^{k-1}$, and choose $\Cr{e_HS}$ so small that the domain of the function $f(\cdot,t_1)$, whose existence is guaranteed by Theorem \ref{thm:graphical}, contains the region $\Omega := \left[\frac{r_0}{2},r_0\right] \times \{0_{n-k}\}\times B^{k-1}_{r_0}(\zeta)$. Notice that if $Z\in\Omega$ then $\dist(Z,\mathbf{S}(\bC)) \geq \frac{r_0}{2}> \frac{r_0}{4}\geq |\xi|$, and thus \eqref{geometric} holds. If $(X,t_1)$ belongs to the graph of $f(\cdot,t_1)$ over $\Omega$, and if $Z$ is its projection onto $\mathbf{H}_i$, then by triangle inequality we can estimate
\begin{equation} \label{concluding HS}
    |\xi|\leq\frac{2}{\sqrt{3}}\dist(Z-\Xi,\bC)\leq \frac{2}{\sqrt{3}}\dist(X-\Xi,\bC)+\frac{2}{\sqrt{3}}|f(Z,t_1)|\,.
\end{equation}
Furthermore, for such points $X=(x_1,f(Z,t_1),y), \, Z=(x_1,0_{n-k},y)$, and with $\,\Xi=(\xi,\zeta)$ we also have
\begin{equation}\label{loheat2}
    \begin{split}
|X-\Xi|^2&= |(x_1,f(Z,t_i))-\xi|^2+|y-\zeta|^2 \\
&\leq 
2(|x_1|^2+|f(Z,t_i)|^2+|\xi|^2)+|y-\zeta|^2\leq Cr_0^2 <\frac14
\end{split}
\end{equation}
for a geometric constant $C$ depending on $\beta=1/40$
and \eqref{e:bound}. We also restricted $r_0$ so that the last inequality holds. Since $r_0^2\leq \tau-t_1 \leq 2 r_0^2$ we immediately estimate
\begin{equation}\label{loheat}
\rho_{(\Xi,\tau)}(X,t_1) = \frac{1}{(4\pi(\tau-t_1))^{\sfrac{k}{2}}} \,\exp\left(-\frac{|X-\Xi|^2}{2(\tau-t_1)}\right)\geq r_0^{-k} \frac{1}{(8\pi)^{\sfrac{k}{2}}}\,e^{-C/2}\,.
\end{equation}
In particular, if we square \eqref{concluding HS} and we integrate over $\Omega$ 
and noting from \eqref{loheat2} that 
${\rm graph}\,f(\cdot,t_1)|_{\Omega}\subset B_{\frac12}(\Xi)$, 
we obtain
\[
\begin{split}
|\xi|^2 &\leq \frac{8}3 (\omega_{k-1}r_0^k)^{-1}\left(\int_{B_{\frac12}(\Xi)} \dist^2(X-\Xi,\bC)\,d\|V_{t_1}\|+\int_{\Omega} |f(Z,t_1)|^2\,d\mathcal{H}^k(Z)\right) \\
&\leq \frac{8 (8\pi)^{\sfrac{k}{2}}e^{C/2}}{3 \omega_{k-1}} \int_{B_{\frac12}(\Xi)}\dist^2(X-\Xi,\bC)\,\rho_{(\Xi,\tau)}(X,t_1)\,d\|V_{t_1}\| + C r_0^{-(k+2)}\mu^2 \\
&\leq \frac{8 (8\pi)^{\sfrac{k}{2}}e^{C/2}}{3 \omega_{k-1}} \Cr{c_nc3}r_0(\mu^2+\|u\|^2+|\xi|^2) + C r_0^{-(k+2)} \mu^2\,.
\end{split}
\]
Here we used \eqref{loheat}, \eqref{good-time} and \eqref{fHS4}.
By suitably choosing $r_0$ depending on $k$ and $\Cr{c_nc3}$ (thus ultimately only on $k,n,p,q,E_1,\Cr{conslice}$) so to absorb the $|\xi|^2$ summand on the right-hand side, we conclude \eqref{e:HS1}. Then, by combining \eqref{fHS1}-\eqref{fHS3}-\eqref{fHS2} with \eqref{e:HS1} we get \eqref{e:HS2}. 

Finally, to prove \eqref{e:HS3} we restrict \eqref{e:HS2} to the region where the support of the flow coincides with the graph of $f$: in such region, points $X \in \spt\|V_t\|$ are parametrized as $X=Z+f(Z,t)$ with $(Z,t) \in U \cap (B_1\times(-1,0))$, and thus $\dist^2(X-\Xi,\bC)= |f(Z,t)-\xi^{\perp_Z}|^2$.
\end{proof}

We conclude by recording a corollary of Theorem \ref{thm:HS}.

\begin{proposition} \label{p:nonconc-excess}
Let $\sigma \in \left( 0, \sfrac{1}
{40}\right)$ and 
$\kappa \in \left( 0, 1\right)$ be given. Then there are constants $ \Cl[eps]{e-binding}=\Cr{e-binding} (k,n,p,q,E_1,\sigma)\in (0,1)$ and $\Cl[con]{c-binding}=\Cr{c-binding}(k,n,p,q,E_1,\Cr{conslice},\kappa)\in (0,1)$ with the following property. Assume that $(\{V_t\},\{u(\cdot,t)\})\in \mathscr N_{\Cr{e-binding}}(U_5\times(-25,0])$ satisfies (A1)-(A6). Then, it holds
\begin{equation} \label{e:nonconc-cor}
   \iint_{P_{\frac12}}
   \frac{\dist^2(X,\bC)}{\max\{|x|,\sigma\}^{1-2\kappa}} \, d\|V_t\|(X)\,dt  \leq \Cr{c-binding}\,\max\{\|u\|,\mu\}^2.
\end{equation}
\end{proposition}

\begin{proof}
    Let $\sigma \in \left( 0, \sfrac{1}{40}\right)$. For $\Cr{e-binding}$ sufficiently small
    depending on $k,n,p,q,E_1$ and
    $\sigma$, we can apply Proposition \ref{p:NH_property} to conclude that for every $\zeta \in B_{\sfrac12}^{k-1}\subset \mathbf{S}(\bC)$ and for every $\tau \in [-\sfrac14,0]$
    there exists $\Xi=(\xi,\zeta) \in B^{n-k+1}_\sigma(0) \times \{\zeta\}$ such that $\Theta(\Xi,\tau) \geq \sfrac32$. If we further assume $\Cr{e-binding}\leq \Cr{e_HS}$, Theorem \ref{thm:HS} then implies that
    \begin{equation}\label{axiest1}
     |\xi| \leq \Cr{c_HS}\,\max\{\mu, \|u\|\}\,,
    \end{equation}
    and
    \begin{equation}\label{axiest2}
    \sup_{t\in [-1+\tau,\tau)} (\tau-t)^{-\kappa} \int_{B_1} \dist^2(X-\Xi,\bC)\,\rho_{(\Xi,\tau)}(X,t)\,d\|V_t\|(X) \leq \Cr{c_HS2}\,\max\{\|u\|,\mu\}^2\,.
    \end{equation}
    In particular, for any $(\zeta,\tau)$ as above it holds
    \begin{equation}\label{axiest3}
    \begin{split}
&\int_{\tau-2\sigma^2}^{\tau-\sigma^2} \int_{B_\sigma((0,\zeta))} \dist^2(X,\bC) \,d\|V_t\|\,dt \\ &\qquad\leq 2\int_{\tau-2\sigma^2}^{\tau-\sigma^2} \int_{B_{2\sigma}(\Xi)} \dist^2(X-\Xi,\bC) \,d\|V_t\|\,dt + 2\Cr{c_HS}^2\omega_k E_1\,\max\{\mu, \|u\|\}^2\,\sigma^{k+2}
\end{split}
\end{equation}
using ${\rm dist}^2(X,\bC)\leq 2\,{\rm dist}^2(X-\Xi,\bC)+2|\xi|^2$
if $X\in B_{\sigma}((0,\zeta))$,
$B_{\sigma}((0,\zeta))\subset
B_{2\sigma}(\Xi)$, \eqref{e:uniform-area-bound}
and \eqref{axiest1}. The first term of \eqref{axiest3} is
\begin{equation}\label{axiest4}
    \begin{split}
    &\qquad\leq c(k)\sigma^k \int_{\tau-2\sigma^2}^{\tau-\sigma^2} \int_{B_{2\sigma}(\Xi)} \dist^2(X-\Xi,\bC) \,\rho_{(\Xi,\tau)}(X,t) \,d\|V_t\|\,dt \\
    & \qquad \leq c(k)\Cr{c_HS2} \max\{\mu, \|u\|\}^2\,\sigma^{k+2+2\kappa}\,
    \end{split}
    \end{equation}
due to 
$\sigma^k\rho_{(\Xi,\tau)}(X,t)\geq c(k)$ if $(X,t)\in B_{2\sigma}(\Xi)\times[\tau-2\sigma^2,\tau-\sigma^2]$
and \eqref{axiest2}.
Combining \eqref{axiest3} and \eqref{axiest4}, we obtain with a constant $\Cl[con]{axiestcon}=\Cr{axiestcon}(n,k,p,q,E_1,\kappa,\Cr{conslice})$
\begin{equation}
   \int_{\tau-2\sigma^2}^{\tau-\sigma^2} \int_{B_\sigma((0,\zeta))} \dist^2(X,\bC) \,d\|V_t\|\,dt\leq \Cr{axiestcon}\sigma^{k+2}\max\{\mu,\|u\|\}^2.
\end{equation}
Denoting $B_\sigma(\bS(\bC))$ the $\sigma$-tubular neighborhood of the spine $\bS(\bC)$, we can then cover $B_{\sfrac12} \cap B_\sigma(\bS(\bC)) \times (-\sfrac14,-\sigma^2)$ with $\mathrm{O}(\sigma^{-k-1})$ cylinders $B_\sigma((0,\zeta_i))\times[\tau_i-2\sigma^2,\tau_i-\sigma^2]$ with finite intersection property to conclude that
\begin{equation} \label{nonconc-close}
    \int_{-\frac14}^{-\sigma^2}\int_{B_{\sfrac12}\cap B_\sigma(\bS(\bC))} \frac{\dist^2(X,\bC)}{\sigma} \,d\|V_t\|\,dt\leq \Cr{axiestcon}\max\{\mu,\|u\|\}^2\,.
\end{equation}
    
On the other hand, the estimates \eqref{e:HS_technical2} and \eqref{error estimates} immediately imply that
\begin{equation}\label{nonconc-close-ET}
     \int_{-\sigma^2}^0 \int_{B_{\sfrac12}} \dist^2(X,\bC)\,d\|V_t\|\,dt \leq C \sigma^{2} \max\{\|u\|,\mu\}^2 \,,
\end{equation}
so that it holds, in fact,
\begin{equation} \label{nonconc-close-final}
    \int_{-\frac14}^{0}\int_{B_{\sfrac12}\cap B_\sigma(\bS(\bC))} \frac{\dist^2(X,\bC)}{\sigma} \,d\|V_t\|\,dt\leq C\max\{\mu,\|u\|\}^2\,.
\end{equation}

Next, in order to obtain the estimate in the region away from $\bS(\bC)$, observe first that, upon possibly further reducing $\Cr{e-binding}$, we can make sure that $(B_{\sfrac12} \setminus B_\sigma(\bS(\bC))) \times (-\sfrac14,0)$ is contained in the region where the support of the flow coincides with $\mathrm{graph}\,f$. Arguing as in the proof of Theorem \ref{thm:graphical}, we can cover this region with (at most countably many) sets $T_{|x_i|,\sfrac12}(y_i,s_i)$ with the property that the number of intersections of $T_{|x_i|,1}(y_i,s_i)$ is bounded by a constant $c(n,k)$, whereas Proposition \ref{p:NH_property} guarantees that for every $i$ there exists $\Xi_i=(\xi_i,y_i) \in B_\sigma^{n-k+1}(0)\times\{y_i\}$ such that $\Theta(\Xi_i,s_i) \geq \frac32$. By Theorem \ref{thm:HS}, $|\xi_i| \leq \Cr{c_HS}\max\{\mu,\|u\|\}$, and thus we may argue as above. Precisely, setting 
\[
\begin{split}
\tilde T_{|x_i|,\sfrac12}(y_i,s_i) :&= T_{|x_i|,\sfrac12}(y_i,s_i) \cap \left\lbrace s_i- \frac{|x_i|^2}{16} < t<s_i - \frac{|x_i|^2}{32}\right\rbrace \\
&= \left\lbrace (x,y,t) \ \colon \, \left( |x| - |x_i| \right)^2 + |y-y_i|^2 < \frac{|x_i|^2}{16}\,, \quad s_i- \frac{|x_i|^2}{16} < t<s_i - \frac{|x_i|^2}{32} \right\rbrace\,,
\end{split}
\]
we may observe that for any $(x,y,t) \in \tilde T_{|x_i|,\sfrac12}(y_i,s_i)$, $\frac34 |x_i| < |x| < \frac54 |x_i| $, as well as that, since $T_{|x_i|,\sfrac12}(y_i,s_i)$ is contained in the graphicality region, $|\xi_i|\leq |x_i|$, and thus $|x-\xi_i|^2\leq 5|x_i|^2$. Combining this with $|y-y_i|^2 \leq |x_i|^2$ and $s_i-t \geq \frac{|x_i|^2}{32}$, we have that $\rho_{(\Xi_i,s_i)} \geq c |x_i|^{-k}$ in $T_{|x_i|,\sfrac12}(y_i,s_i)$. Therefore, we can estimate
\[
\begin{split}
&\iint_{\tilde T_{|x_i|,\sfrac12}(y_i,s_i)} \frac{\dist^2(X,\bC)}{|x|^{1-2\kappa}} \,d\|V_t\|\,dt \\
&\qquad \qquad  \leq C\iint_{\tilde T_{|x_i|,1}(y_i,s_i)} \frac{\dist^2(X-\Xi_i,\bC)}{|x|^{1-2\kappa-k}}\,\rho_{(\Xi_i,s_i)} \,d\|V_t\|\,dt + C \max\{\mu,\|u\|\}^2 |x_i|^{k+1+2\kappa} \\
&\qquad \qquad \leq C |x_i|^{k+4\kappa-1} \int_{s_i-\frac{|x_i|^2}{16}}^{s_i-\frac{|x_i|^2}{32}} (s_i-t)^{-\kappa}\int_{B_1}\dist^2(X-\Xi_i,\bC)\, \rho_{(\Xi_i,s_i)} \,d\|V_t\|\,dt \\
&\qquad\qquad\qquad\qquad\qquad\qquad+ C \max\{\mu,\|u\|\}^2 |x_i|^{k+1+2\kappa} \\
&\qquad\qquad\leq C |x_i|^{k+1+2\kappa} \max\{\mu,\|u\|\}^2\,,
\end{split}
\]
    where we have used \eqref{e:HS2}. Now, for every $j \in \mathbb N$ let $\mathcal{C}_j$ be the collection of tori $\tilde T_{|x_i|,1}(y_i,s_i)$ which intersect $\left(B_{\sfrac12} \setminus B_\sigma(\bS(\bC)) \right)\times \left(-\sfrac14,0\right)$ such that $2^{-j}\leq|x_i|\leq 2^{1-j}$: notice that the cardinality of $\mathcal C_j$ is at most $C (2^{-j})^{-k-1}$, where $C$ is independent of $j$. We can then estimate
    \begin{equation} \label{nonconc-away}
        \begin{split}
            & \int_{-\frac14}^{-\sigma^2}\int_{B_{\sfrac12}\setminus B_\sigma(\bS(\bC))} \frac{\dist^2(X,\bC)}{|x|^{1-2\kappa}} \,d\|V_t\|\,dt \\
            & \qquad \qquad \qquad\leq C \sum_{j\in\mathbb N} \sum_{\{i\,\colon\, \tilde T_{|x_i|,1}(y_i,s_i) \in \mathcal C_j\}} \iint_{\tilde T_{|x_i|,\sfrac12}(y_i,s_i)} \frac{\dist^2(X,\bC)}{|x|^{1-2\kappa}} \,d\|V_t\|\,dt \\
             & \qquad \qquad \qquad\leq C \,\max\{\mu,\|u\|\}^2 \sum_j (2^{-j})^{2\kappa} \\
             & \qquad \qquad \qquad\leq C \,\max\{\mu,\|u\|\}^2\,.
        \end{split}
    \end{equation}
Since, by estimates \eqref{e:HS_technical2} and \eqref{error estimates},
\begin{equation} \label{nonconc-away-et}
    \int_{-\sigma^2}^{0}\int_{B_{\sfrac12}\setminus B_\sigma(\bS(\bC))} \frac{\dist^2(X,\bC)}{|x|^{1-2\kappa}} \,d\|V_t\|\,dt \leq C \sigma^{1+2\kappa} \max\{\mu,\|u\|\}^2\,,
\end{equation}
the proof of \eqref{e:nonconc-cor} follows by combining \eqref{nonconc-close-final}, \eqref{nonconc-away}, and \eqref{nonconc-away-et}.

Finally, the same argument 
\end{proof}

\section{Blow-up and decay for the linear problem} \label{sec:BU}

Having established the main non-concentration estimates in the previous section, we now turn to their central application: the analysis of the asymptotic behavior of the flow at a singular point. The strategy is to perform a blow-up analysis: we consider a sequence of flows whose $L^2$-excess $\mu^{(m)}$ vanishes as $m \to \infty$, and we study the limit of the corresponding graphing functions $f^{(m)}$ after normalization by $(\mu^{(m)})^{-1}$.

The main goal of this section is to prove that the resulting limit function, $\tilde f$, is a classical solution to the heat equation on each of the three half-planes forming the triple junction. Furthermore, we will show that the geometric constraints on the original flow impose powerful symmetry conditions on the boundary values of $\tilde f$ at the spine. This detailed understanding of the linearized problem, obtained via a reflection argument, is the crucial input for deriving the excess decay in the next section.

\smallskip

Let $\{\sigma^{(m)}\}_{m \in \mathbb N}$ be a decreasing sequence in $\left(0, \sfrac{1}{40}\right)$ such that $\lim_{m \to \infty} \sigma^{(m)} =0$. For fixed $p \in [2,\infty)$, $q \in (2,\infty)$, and $E_1\in [1,\infty)$, we let $\varepsilon^{(m)}$ denote the threshold $\Cr{e-binding}(k,n,p,q,E_1,\sigma^{(m)})$.

\begin{definition} \label{def:BU}
    A blow-up sequence is a sequence of pairs $$\left( \{V^{(m)}_t\}_{t \in I}, \{u^{(m)}(\cdot, t)\}_{t\in I} \right) \in \mathcal N_{\varepsilon^{(m)}}(U_{R} \times I)$$ for $I=[-R^2,0]$ and $\varepsilon^{(m)}\in (0,1)$ so that assumptions (A1)-(A6) are satisfied for the above choice of parameters $p,q,E_1$ and for a standard triple junction $\bC$, and for which, additionally, it holds
    \begin{equation} \label{blow-up cond}
        \mu^{(m)} \to 0 \qquad \mbox{and} \qquad (\mu^{(m)})^{-1} \|u^{(m)}\| \to 0 \qquad \mbox{as $m \to \infty$}\,.
    \end{equation}
    Coordinates in the ambient space $\R^n$ are chosen so that all conventions and notation set forth in Section \ref{sec:results} are in place.
\end{definition}

Given that all estimates are scale invariant, in what follows we will assume without loss of generality that $R=5$. For every $m \in \mathbb N$, apply Theorem \ref{thm:graphical} and conclude the existence of open sets $U^{(m)}\subset \bC \cap (U_{5}\times(-25,0))$ satisfying \eqref{three}-\eqref{three2} with $U=U^{(m)}$ and $\sigma=\sigma^{(m)}$, as well as functions $f^{(m)}\in C^{1,\alpha}(U^{(m)};\bC^\perp)$ satisfying \eqref{e:bound}-\eqref{para1-ex} with $\{V_t\}_t, \{u(\cdot,t)\}_t,\sigma,U,f,\mu$ replaced by $\{V^{(m)}_t\}_t, \{u^{(m)}(\cdot,t)\}_t,\sigma^{(m)},U^{(m)},f^{(m)},\mu^{(m)}$, respectively. Furthermore, by \eqref{expara1}, we have that, for any $\sigma > 0$ and upon denoting $Q_\sigma:=\{(x,y,t)\in \bC \cap P_{2} \,\colon\,|x|>5\sigma \}$,
\begin{equation} \label{BU:C1a est}
    \|f^{(m)}\|_{C^{1,\alpha}(Q_\sigma)} \leq \Cr{nonsing1}(\sigma) \max\{\|u^{(m)}\|,\mu^{(m)}\}\,,
\end{equation}
for all $m$ such that $\sigma^{(m)} \leq \sigma$. In particular, by the Ascoli-Arzel\`a theorem the functions $\tilde f^{(m)} := (\mu^{(m)})^{-1}\,f^{(m)}$ converge in $C^1_{{\rm loc}}$ to a function $\tilde f \colon \bC \cap P_{2} \cap \{|x|>0\} \to \bC^\perp$.

Next, as a consequence of Proposition \ref{p:NH_property}, for all sufficiently large $m$ property \eqref{e:NH} holds with $\left(\{V_t^{(m)}\}_t, \{u^{(m)}(\cdot,t)\}_t\right)$ in place of $\left(\{V_t\}_t, \{u(\cdot,t)\}_t\right)$. Thus, upon fixing $\delta \in \left(0,\sfrac14\right)$, and upon denoting $\Theta^{(m)}$ the Gaussian density of the flow $\{V_t^{(m)}\}$, for every $z \in B^{k-1}_1$ and for every $\tau \in [-1,0)$ the set
\[
J_{z,\tau}^{(m)} := \{ \xi \in B^{n-k+1}_\delta(0) : \Theta^{(m)}((\xi,z),\tau) \ge 3/2 \}
\]
is not empty. Furthermore, since $\Theta^{(m)}$ is upper semi-continuous, $J^{(m)}_{z,\tau}$ is a compact subset of $\R^{n-k+1}$. We can then define a unique map $(z,\tau) \mapsto \xi^{(m)}(z,\tau) \in J^{(m)}_{z,\tau}$ according to the following selection procedure: for each $(z,\tau)$, we let $\xi^{(m)}(z,\tau)$ be the point in $J^{(m)}_{z,\tau}$ with the minimal Euclidean norm $|\xi|$; if multiple such points exist, then we choose the one which is first in the lexicographical ordering of $\R^{n-k+1}$. We shall call this map a \emph{binding function}, and we note explicitly that if $\Theta^{(m)}((0,z),\tau)\geq 3/2$ then $\xi^{(m)}(z,\tau)=0$. Let also $\Xi^{(m)}(z,\tau):=(\xi^{(m)}(z,\tau),z) \in \R^n$. Then, the maps $(z,\tau) \mapsto \xi^{(m)}(z,\tau)$ and $(z,\tau) \mapsto \Xi^{(m)}(z,\tau)$ are Borel measurable, and by Theorem \ref{thm:HS} it holds
\begin{equation} \label{e:pushing3jp}
|\xi^{(m)}(z,\tau)| \leq \Cr{c_HS} \max\{\|u^{(m)}\|,\mu^{(m)}\}\,,
\end{equation}
and
 \begin{equation}\label{e:binding}
 \begin{split}
     & \sup_{t\in[-1+\tau,\tau)}  
      (\tau-t)^{-\kappa-\frac{k}{2}}\int_{ \Omega^{(m)}_t} e^{-\frac{\left|X+f^{(m)}(X,t)-\Xi^{(m)}(z,\tau)\right|^2}{4(\tau-t)}} \times \\
      &\qquad \qquad \qquad \times|f^{(m)}(X,t)-\xi^{(m)}(z,\tau)^{\perp_X}|^2\,d\mathcal H^k(X) \leq \Cr{c_HS2}\,\max\{\mu^{(m)},\|u^{(m)}\|\}^2\,,
    \end{split}
    \end{equation}
    where $\Omega^{(m)}_t:=B_1 \cap \{X\,\colon\,(X,t) \in U^{(m)}\}$. The estimate \eqref{e:pushing3jp} guarantees that the sequence $\tilde\xi^{(m)}(z,\tau) := (\mu^{(m)})^{-1}\xi^{(m)}(z,\tau)$ is uniformly bounded in $m$. Hence, upon passing to a subsequence (which in principle may depend on $(z,\tau)$), it converges to a limit point $\tilde\xi(z,\tau)$. The estimate \eqref{e:binding}, in turn, produces, in the limit of this subsequence, that
   \begin{equation}\label{e:binding limit}
   \begin{split}
       \sup_{t \in [-1+\tau,\tau)} (\tau-t)^{-\kappa-\frac{k}{2}} \int_{\bC \cap B_1 \cap \{|x|>0\}}&\exp\left(-\frac{|x|^2+|y-z|^2}{4(\tau-t)}\right)\times \\
       &\times|\tilde f(x,y,t)-\tilde\xi(z,\tau)^{\perp_{(x,y)}}|^2\,d\Ha^{k}(x,y)\leq \Cr{c_HS2}\,.
       \end{split}
   \end{equation}
Recall that $\tilde\xi(z,\tau)^{\perp_{(x,y)}}$ is the projection of the vector $\tilde\xi(z,\tau)$ onto the orthogonal complement to $\bC$ at the point $(x,y)$. The validity of \eqref{e:binding limit} implies that for every $j \in \{1,2,3\}$ the projection $\mathbf P_{\mathbf P_j^\perp}(\tilde\xi(z,\tau))$ is  uniquely determined, and thus that the full sequence $\{\mathbf P_{\mathbf P_j^\perp}(\tilde\xi^{(m)}(z,\tau))\}_{m\in\mathbb N}$ converges to $\mathbf P_{\mathbf P_j^\perp}(\tilde\xi(z,\tau))$. We claim that then the vector $\tilde\xi(z,\tau)$ is uniquely determined, and thus that the full sequence $\{\tilde\xi^{(m)}(z,\tau)\}_{m\in\mathbb N}$ converges to $\tilde\xi(z,\tau)$. Indeed, suppose that $v\in\R^{n-k+1}$ is such that $\mathbf P_{\mathbf P_j^\perp}(v)=0$ for every $j$. As usual, let $w_j$ be the vectors such that $\mathbf H_j = {\rm span}^+(w_j) \oplus \R^{k-1}$, Then, $\mathbf P_{\mathbf P_j^\perp}$ is the operator $\mathbf P_{\mathbf P_j^\perp}=I_{n-k+1}-w_j\otimes w_j$, where $I_{n-k+1}$ is the orthogonal projection onto $\R^{n-k+1}$. Since $v \in \R^{n-k+1}$, the condition $\mathbf P_{\mathbf P_j^\perp}(v)=0$ translates into $v=(v \cdot w_j) \,w_j$, and this holds true for every $j\in\{1,2,3\}$. Since the $w_j$'s are linearly independent, it follows immediately that $v=0$, as we wanted. 

As a consequence of this argument, the pointwise limit function $(z,\tau) \mapsto \tilde\xi(z,\tau)$ is well defined, and it satisfies \eqref{e:binding limit}. We record the above conclusions in the following

\begin{proposition}\label{blow-up-basic}
    Let $\left\{\left(\{V^{(m)}_t\}_t, \{u^{(m)}(\cdot,t)\}_t\right)\right\}_{m\in\mathbb N}$ be a blow-up sequence, and let $U^{(m)}$ and $f^{(m)}$ be the corresponding domains of graphicality and parametrizations, respectively. Let $(z,\tau) \in B_1^{k-1} \times (-1,0) \mapsto \xi^{(m)}(z,t) \in \bS(\bC)^\perp$ be binding functions. Then, as $m \to \infty$:
    \begin{itemize}
        \item[(i)] the functions $\tilde f^{(m)}:=(\mu^{(m)})^{-1}f^{(m)}$ converge locally in $C^1$ to a function $\tilde f \colon \tilde U := \bC \cap P_1\cap \{|x|>0\} \to \bC^\perp$ satisfying
        \begin{equation} \label{e:reguarity inside}
          \sup_{(X,t) \in \tilde U} |x|^{\frac{k}{2}+2} \left( |x|^{-1} |\tilde f(X,t)| + \|\nabla \tilde f(X,t)\| \right)  \leq \Cr{nonsing1}\,.
        \end{equation}
        \item[(ii)] the functions $\tilde \xi^{(m)}:=(\mu^{(m)})^{-1}\xi^{(m)}$ converge to $\tilde \xi \colon B_1^{k-1}\times (-1,0) \to \bS(\bC)^\perp$ with $\|\tilde\xi\|_\infty \leq \Cr{c_HS}$;
        \item[(iii)] For every $(z,\tau) \in B^{k-1}_1 \times (-1,0)$, the following holds:
        \begin{equation}\label{e:bindinglim2}
        \sup_{t \in (-1+\tau,\tau))} (\tau-t)^{-\kappa-\frac{k}{2}} \int_{\bC \cap B_1 \cap \{|x|>0\}} e^{-\frac{|x|^2+|y-z|^2}{4(\tau-t)}}\,\left|\tilde f(x,y,t) - \tilde\xi(z,\tau)^{\perp_{(x,y)}}\right|^2\,d\mathcal H^k(x,y)\leq \Cr{c_HS2}\,.
        \end{equation}
    \end{itemize}
\end{proposition}

\begin{proposition} \label{blow-up-fine}
    Let $\{(\mathscr V^{(m)}, u^{(m)})\}_{m \in \mathbb N}$, $U^{(m)}$, $\tilde f^{(m)}$, $\tilde f$, $\tilde\xi^{(m)}$, and $\tilde \xi$ be as in Proposition \ref{blow-up-basic}. Then:
    \begin{itemize}
        \item[(i)] The convergence of $\tilde f^{(m)}$ to $\tilde f$ is strong up to the spine, in the sense that
        \begin{equation}\label{strong-L2}
           \iint_{\tilde U\cap P_{\sfrac12}} |\tilde f|^2 \, d\mathcal H^k \, dt =\lim_{m \to \infty} \iint_{U^{(m)}\cap P_{\sfrac12}} |\tilde f^{(m)}|^2 \, d\mathcal H^k \, dt\,.
        \end{equation}
        \item[(ii)] We have the estimate
        \begin{equation}\label{eq:bound-y-derivative}
           \limsup_{m \to \infty} \iint_{U^{(m)}\cap P_{\sfrac12}} |\nabla_y \tilde f^{(m)}|^2\,d\mathcal H^k dt\leq \Cr{er0}\,.
        \end{equation}
        \item[(iii)] We have the estimate
        \begin{equation}
            \limsup_{m \to \infty} (\mu^{(m)})^{-2} \iint_{P_{\sfrac12}}
            \dist^2(X,\bC) \,d\|V^{(m)}_t\|\,dt \leq \iint_{\bC\cap P_{\sfrac12}\cap\{|x|>0\}} |\tilde f(X,t)|^2\,d\mathcal H^k(X)\,.
         \end{equation}
         \item[(iv)] The function $\tilde f$ is locally smooth on $\bC\cap B_1 \cap \{|x|>0\}$, and on each half-plane $\mathbf H_j$ it is a solution to the heat equation $\partial_t \tilde f -\Delta \tilde f = 0$.
    \end{itemize}
\end{proposition}

\begin{proof}
    \textbf{Proof of (i).} Since the functions $\tilde f^{(m)}$ converge to $\tilde f$ locally uniformly away from the spine $\bS(\bC)$, it is enough to show that there is no concentration of the weighted $L^2$ norm at the spine. Let $r>0$, and let $B_r(\bS(\bC))$ denote the $r$-tubular neighborhood of the spine $\bS(\bC)$. With an analogous covering argument as in the second half of the proof of Proposition \ref{p:nonconc-excess}, but using \eqref{e:binding} instead of \eqref{e:HS2}, we conclude easily that
    \begin{equation} \label{e:nonconc-function}
        \iint_{U^{(m)} \cap P_{\sfrac12}} \frac{|f^{(m)}(X,t)|^2}{|x|^{1-2\kappa}} \, d\mathcal H^k(X) \, dt\leq \Cr{c_HS2}\,\max\{\|u\|,\mu\}^2\,.
    \end{equation}
    In particular, 
    \[
    \iint_{U^{(m)} \cap P_{\sfrac12} \cap B_r(\bS(\bC))}  |\tilde f^{(m)}|^2 \,d\mathcal H^k\,dt \leq \Cr{c_HS2}\, r^{1-2\kappa} \longrightarrow 0 \quad \mbox{as $r \to 0^+$}\,,
    \]
    as we needed.

    \medskip

    \textbf{Proof of (ii).}
    Let $\psi(y,t) \in C^\infty_c(U_1^{k-1} \times [-4,0])$ be non-negative, radially symmetric in $y$, and such that $\psi(\cdot,t)=1$ on $U_{\sfrac12}^{k-1}$ for every $t$, and let $\tilde\eta=\tilde\eta(|x|)$ be as in Proposition \ref{lateral}. Then, by \eqref{kore:0} we have, for all $m$ sufficiently large, the estimate
    \begin{equation} \label{the-perp-ineq}
    \Cr{er0} (\mu^{(m)})^2 \geq \int_{-\frac12}^0 \int_{B_{\frac12} \cap \mathrm{graph}\,f^{(m)}(\cdot,t)} \sum_{j=n-k+2}^n |S^\perp e_j|^2 \, dV^{(m)}_t\,dt\,.
    \end{equation}
    Now, pick a point $(X,t) \in \mathrm{graph}\,f^{(m)}(\cdot,t)$, and suppose without loss of generality that it projects onto $U^{(m)} \cap \mathbf{H}_1$, so that in the standard coordinates of $\R^n$ we have $X=(x_1,f^{(m)}(x_1,y,t),y,t)$ with $f^{(m)}(x_1,y,t)\in \mathbf{P}_1^\perp$. Then
    \[
    S=I_{\mathbf P_1} +\nabla f^{(m)}(\mathbf P_1,t)\,,
    \]
    so that
    \[
    S^\perp=I_{\mathbf P_1^\perp}-\nabla f^{(m)}(\mathbf{P}_1,t)\,.
    \]
    Since $e_j \in \mathbf{P}_1$ for every $j \in \{n-k+2,\dots,n\}$, we have that $|S^\perp e_j|^2= |\nabla_{e_j}f^{(m)}|^2$, and thus by the area formula \eqref{the-perp-ineq} implies
    \[
    \Cr{er0} (\mu^{(m)})^2 \geq \int_{-\frac12}^0 \int_{B_{\frac12} \cap U^{(m)}} |\nabla_y f^{(m)}|^2 \, d\mathcal H^k\,dt\,,
    \]
    from which \eqref{eq:bound-y-derivative} immediately follows.
    
    \medskip

    \textbf{Proof of (iii).} Let $r > 0$, and, as in the proof of (i), take $m$ so large that the space-time region $P_{\frac12} \setminus B_r(\bS(\bC))$ is contained in the graphicality region of $\mathscr{V}^{(m)}$. In particular, if $X \in \spt\|V^{(m)}_t\|$ then $X=z+f^{(m)}(z,t)$ for $z \in U^{(m)}$, and $\dist(X,\bC) \leq |f^{(m)}(z,t)|$. We can then estimate
    \[
    \begin{split}
    & \quad \limsup_{m \to \infty} (\mu^{(m)})^{-2} \iint_{P_{\frac12}\setminus B_r(\bS(\bC))} \dist^2(X,\bC) \,d\|V^{(m)}_t\|\,dt\\ & \leq \limsup_{m \to \infty} \left(1+C\,\Lip(f^{(m)};U^{(m)}\setminus B_r(\bS(\bC))) \right) \iint_{P_{\sfrac12}\cap U^{(m)}} |\tilde f^{(m)}|^2\,d\mathcal H^k\,dt\,,
    \end{split}
    \]
    where we have used the area formula and estimated the Jacobian of the graph map with the factor $(1+C\,\Lip)$. Since the Lipschitz constant of $f^{(m)}$ tends to zero in any region at positive distance from $\bS(\bC)$ by \eqref{e:reguarity inside}, we conclude from (i) that
    \begin{equation} \label{e:limsup_away}
    \begin{split}
    &  \quad  \limsup_{m \to \infty} (\mu^{(m)})^{-2} \iint_{P_{\frac12}\setminus B_r(\bS(\bC))} \dist^2(X,\bC) \,d\|V^{(m)}_t\|\,dt \\
    & \leq \iint_{\bC \cap P_{\sfrac12}\cap\{|x|>0\}} |\tilde f(z,t)|^2\,d\mathcal H^k(z)\, dt\,.
    \end{split}
    \end{equation}
On the other hand, we see as an immediate consequence of Proposition \ref{p:nonconc-excess} that
\begin{equation}\label{e:limsup_inside}
\begin{split}
    &\quad \limsup_{m \to \infty} (\mu^{(m)})^{-2} \iint_{P_{\sfrac12}\cap B_r(\bS(\bC))} \dist^2(X,\bC) \,d\|V^{(m)}_t\|\,dt \\
    &\leq \limsup_{m \to \infty} (\mu^{(m)})^{-2} \,r^{1-2\kappa}\iint_{P_{\sfrac12}\cap B_r(\bS(\bC))} \frac{\dist^2(X,\bC)}{\max\{|x|,\sigma ^{(m)}\}^{1-2\kappa}} \,d\|V^{(m)}_t\|\,dt \\
    &\leq \Cr{c-binding} r^{1-2\kappa}\,.
\end{split}
\end{equation}
By choosing any $\kappa \in \left( 0, \sfrac12\right)$, combining \eqref{e:limsup_away} and \eqref{e:limsup_inside} and letting $r \to 0$ one deduces (ii).

\medskip

\textbf{Proof of (iv).} This follows from the same argument as in \cite[Lemma 8.4]{Kasai-Tone}.
\end{proof}

In the following, for $\theta\in\mathbb R$,
we let ${\bf R}_{\theta}:\mathbb R^{n}\rightarrow\mathbb R^n$ be the rotation
\begin{equation}
    {\bf R}_{\theta}(x_1,x_2,z):=(x_1\cos\theta-x_2\sin\theta,x_1\sin\theta+x_2\cos\theta,z)
\end{equation}
for $(x_1,x_2,z)\in\mathbb R\times\mathbb R\times\mathbb R^{n-2}$, that is, ${\bf R}_\theta$ rotates the $\mathbb R^2\times\{0_{n-2}\}$ by
$\theta$ counterclockwise while fixing the other coordinates. With this notation, we 
often use the property
\begin{equation}\label{rot}
    ({\bf R}_0+{\bf R}_{\sfrac{2\pi}{3}}
    +{\bf R}_{\sfrac{4\pi}{3}})(x_1,x_2,z)=3(0,0,z).
\end{equation}
For each $j=1,2,3$, we define $f^{(m)}_j$
and $\tilde f_j$ defined on ${\bf H}_1
\cap P_1$ and having values in ${\bf P}_1^\perp=\{0_1\}\times\mathbb R^{n-k}\times\{0_{k-1}\}$
by
\begin{equation}\begin{split}
    &f^{(m)}_j(x,y,t):={\bf R}_{-2\pi(j-1)/3}(f^{(m)}({\bf R}_{2\pi(j-1)/3}(x,y),t)),\\
   & \tilde f_j(x,y,t):={\bf R}_{-2\pi(j-1)/3}(\tilde f({\bf R}_{2\pi(j-1)/3}(x,y),t)).
    \end{split}
\end{equation}
We also use the notation
\begin{equation}
    \tilde f_j(r,0,y,t)=(0,\tilde f_{j,2}(r,y,t),\cdots,
    \tilde f_{j,n-k+1}(r,y,t),0,\cdots,0)\in\{0_1\}\times\R^{n-k}\times\{0_{k-1}\}
\end{equation}
on $\{(r,y,t):r\in(0,1),\,r^2+|y|^2<1,\,t\in(-1,0)\}$. 
\begin{proposition} \label{p:odd}
The odd extensions with 
respect to $r$ of the following functions are solutions
of the heat equation
:
   \begin{itemize}
   \item[(1)] $\sum_{j=1}^3 \tilde f_{j,2}$.
   \item[(2)] If
   $n-k\geq 2$ and for each $j,j'\in\{1,2,3\}$, 
   $\ell\in \{3,\ldots,n-k+1\}$, $\tilde f_{j,\ell}-\tilde f_{j',\ell}$.
   \end{itemize}
\end{proposition}
\begin{proof} In the following,
we assume that $n-k\geq 2$ for
notational simplicity 
and note 
that the case $n-k=1$ proceeds verbatim. 
We write for $j\in\{1,2,3\}$
and $y\in\R^{k-1}$ and $t$
\begin{equation}
    \tilde\xi_j(y,t):={\bf P}_{{\bf P}_1^\perp}({\bf R}_{-2\pi(j-1)/3}(\tilde \xi(y,t)))
    =:(0,\tilde\xi_{j,2},\tilde\xi_{3},\cdots,\tilde\xi_{n-k+1})\times\{0_{k-1}\}.
\end{equation}
Note in particular ${\bf P}_{{\bf P}_1^\perp}$ is the 
identity map on 
$\{0_2\}\times\R^{n-k-1}\times\{0_{k-1}\}$ so that the dependence on $j$ is only on the second component of $\tilde\xi_j$. 
On the other hand, by \eqref{rot}, we have
\begin{equation}
\sum_{j=1}^3\tilde\xi_{j,2}=0.
\end{equation} 
Fix a small $\sigma>0$,
use \eqref{e:bindinglim2} with $\kappa=3/4$,
$\eta\in B_{1}^{k-1}$ and $t=\tau-\sigma^2$ so that we have 
\begin{equation}
    \int_{\bC\cap (B_\sigma^{n-k+1}\times B_{\sigma}^{k-1}(\eta))} e^{-1/2}|\tilde f(x,y,t)-\tilde\xi^\perp(\eta,t+\sigma^2)|^2\,d\mathcal H^k(x,y)\leq C\sigma^{k+3/2}.
\end{equation} 
Note that 
\begin{equation}
    \Big|\sum_{j=1}^3\tilde f_{j,2}(r,y,t)\Big|^2
    =\Big|\sum_{j=1}^3(\tilde f_{j,2}(r,y,t)-\tilde\xi_{j,2}(\eta,t+\sigma^2))\Big|^2
    \leq \sum_{j=1}^3|\tilde f(r w_j,y,t)-\tilde\xi^\perp(\eta,t+\sigma^2)|^2,
\end{equation}
and these prove that
\begin{equation}
    \sup_{t\in(-1,-\sigma^2)}
    \int_{{\bf H}_1\cap (B_\sigma^{n-k+1}\times B_1^{k-1})}\Big|\sum_{j=1}^3 \tilde f_{j,2}(r,y,t)\Big|^2\,d\mathcal H^k\leq C\sigma^{1/2}. 
\end{equation}
For $\ell\in\{3,\ldots,n-k+1\}$ and $j,j'\in\{1,2,3\}$,
since 
\begin{equation}
    |\tilde f_{j,\ell}(r,y,t)-\tilde f_{j',\ell}(r,y,t)|
\leq |\tilde f_{j,\ell}(x,y,t)-\tilde\xi_{\ell}(\eta,t+\sigma^2)|+
|\tilde f_{j',\ell}(x,y,t)-\tilde\xi_{\ell}(\eta,t+\sigma^2)|,
\end{equation}
we may similarly obtain
\begin{equation}
    \sup_{t\in(-1,-\sigma^2)}\int_{{\bf H}_1\cap (B_\sigma^{n-k+1}\times B_1^{k-1})}|\tilde f_{j,\ell}(r,y,t)-\tilde f_{j',\ell}(r,y,t)|^2\,d\mathcal H^k\leq C\sigma^{1/2}.
\end{equation}
Since these functions satisfy the heat equation away from the spine, a simple approximation argument shows that the odd reflection also 
satisfies the heat equation. 
\end{proof}
\begin{proposition} \label{p:even}
    The even extensions with respect to $r$ of the following functions are
    solutions of the heat equation:
    \begin{itemize}
    \item[(1)] For each $j,j'\in\{1,2,3\}$, $\tilde f_{j,2}-\tilde f_{j',2}$.
    \item[(2)]
    If $n-k\geq 2$ and 
    for each $\ell\in\{3,\ldots,n-k+1\}$,
    $\sum_{j=1}^3\tilde f_{j,\ell}$. 
    \end{itemize}
\end{proposition}
\begin{proof}
    Let $\psi(x,y,t)\in C^\infty_c (P_1)$ be a
    non-negative radially symmetric function with respect to $x$ and assume that
    \begin{equation}\label{psizero}
        \frac{\partial\psi}{\partial r}=0
    \end{equation}
    for $0\leq r=|x|\leq \sigma$, where $\sigma$
    is a fixed small number. In addition, let $\tilde\eta$ be as in Proposition \ref{lateral} and set $\tilde\psi(x,y):=\tilde\eta(|x|/2)\tilde\eta(|y|/2)$ so that $\tilde\psi$ has a compact support in $B_{2}^{n-k+1}\times B_2^{k-1}$ and equals to $1$ on $B_1^{n-k+1}\times B_1^{k-1}$. 
Let $a\in \R^{n-k+1}$ be an arbitrary unit vector, and 
define 
\begin{equation}
    \phi(x,y,t):=(a\cdot x)\psi(x,y,t)+2(\sup\psi)\tilde\psi(x,y).
    \end{equation}
    Note that $\psi$ has a compact support in $P_1$ so that 
    \begin{equation} \label{sma-ti0}
        \phi\geq (\sup \psi)\tilde\psi\geq 0.
    \end{equation} 
We use this $\phi$ in \eqref{e:Brakke-ineq}.
To make sure that the 
contribution coming from $\tilde\psi$ is small, we make a specific choice of time interval.
In fact, fix $\hat t\in (-1,0)$ so that
${\rm spt}\psi\subset\{t<\hat t\}$. Then
choose $t^{(m)}\in [-2,-1]$ such that
\begin{equation}\label{sma-ti}
    0\leq \liminf_{m\rightarrow\infty}\frac{1}{\mu^{(m)}}\Big(\int\tilde\psi \,d\|V_{\hat t}^{(m)}\|-\int \tilde\psi \,d\|V^{(m)}_{t^{(m)}}\|\Big).
\end{equation}
We may choose such $t^{(m)}$ by the following argument: 
By Proposition \ref{san-eq}, we have
\begin{equation}\label{sma-ti1}
    \int\tilde\psi \,d\|\bC\|-\int\tilde\psi\,d\|V_{\hat t}^{(m)}\|\leq \Cr{lmin}
    (\mu^{(m)})^2. 
\end{equation}
Moreover, by Proposition \ref{lat4} and by modifying
$\tilde\psi$ at $t=0,\,-4$ appropriately, we may
choose $t^{(m)}\in[-2,-1]$ such that
\begin{equation}\label{sma-ti2}
    \int\tilde\psi\,d\|V_{t^{(m)}}^{(m)}\|-\int\tilde\psi\,d\|\bC\|\leq \Cr{er1}(\mu^{(m)})^2.
\end{equation}
Then \eqref{sma-ti1} and \eqref{sma-ti2} prove \eqref{sma-ti}.
Next, we use $\phi$ in \eqref{e:Brakke-ineq} over
time interval of $[t^{(m)}, \hat t]$ and obtain
\begin{equation}\label{sma-ti3}
    \int \phi(\cdot,t)\,d\|V_t^{(m)}\|\Big|_{t=t^{(m)}}^{\hat t}\leq \int_{t^{(m)}}^{\hat t}\int (\nabla \phi-\phi h)\cdot (h+(u^{(m)})^\perp)+\phi_t\,dV_t^{(m)}dt .
\end{equation}
For the term involving $u^{(m)}$, we have
\begin{equation}\label{sma-ti4}
    (\nabla\phi-\phi h)\cdot (u^{(m)})^\perp
    \leq (\sup|\nabla\phi|)|u^{(m)}|+\frac{\phi}{2}\big(|u^{(m)}|^2+|h|^2\big),
\end{equation}
and for another term in \eqref{sma-ti3}, 
\begin{equation}\label{sma-ti5}
    \begin{split}
    \nabla\phi\cdot h&=\nabla ((a\cdot x)\psi)\cdot h
    +2(\sup\psi)\nabla\tilde\psi\cdot h \\
    &\leq \nabla ((a\cdot x)\psi)\cdot h+\frac{|h|^2}{2}(\sup\psi)\tilde\psi+\frac{2(\sup\psi)|(\nabla\tilde\psi)^\perp|^2}{\tilde\psi}\\
    &\leq \nabla ((a\cdot x)\psi)\cdot h+\frac{|h|^2}{2}\phi+\frac{2(\sup\psi)|(\nabla\tilde\psi)^\perp|^2}{\tilde\psi},
    \end{split}
\end{equation}
where we used \eqref{sma-ti0} in the last line.
Using \eqref{sma-ti4} and \eqref{sma-ti5}, we
obtain
\begin{equation}\label{sma-tt}
    (\nabla\phi-\phi h)\cdot(h+(u^{(m)})^\perp)
    \leq \nabla((a\cdot x)\psi)\cdot h+\Cl[con]{sma-con}\Big(|u^{(m)}|^2+|u^{(m)}|+\frac{|(\nabla\tilde\psi)^\perp|^2}{\tilde\psi}\Big),
\end{equation}
where $\Cr{sma-con}$ depends only on $\sup|\nabla\phi|$ and $\sup\psi$.
We next claim that
\begin{equation}\label{sma-ti6}
\int_{t^{(m)}}^{\hat t}\int_{P_2} |u^{(m)}|^2+|u^{(m)}|+\frac{|(\nabla\tilde\psi)^\perp|^2}{\tilde\psi}\,dV_t^{(m)}dt={\rm o}(\mu^{(m)}).
\end{equation}
The first two terms involving $u^{(m)}$ is 
${\rm o}((\mu^{(m)})^2)$ and ${\rm o}(\mu^{(m)})$,
respectively,
while $\tilde\psi^{-1}|(\nabla\tilde\psi)^\perp|^2$
may be bounded from above by a constant multiple of $|S^\perp(x)|^2$
on $\{2\geq|x|\geq 1\}$ and $|S^{\perp}(y)|^2$ on
$\{2\geq |y|\geq 1\}$. Then one can
proceed as in the argument in \eqref{Ses} and \eqref{kore},
and we may conclude that they are both ${\rm O}((\mu^{(m)})^2)$. This proves the claim
\eqref{sma-ti6}.
Combining \eqref{sma-ti}, \eqref{sma-ti3},
\eqref{sma-tt} and \eqref{sma-ti6},
we have
\begin{equation}\label{sma-a1}
    0\leq \liminf_{m\rightarrow\infty}\frac{1}{\mu^{(m)}}\int_{-1}^{\hat t}\int
    -\nabla^2((a\cdot x) \psi)\cdot S +(a\cdot x)\psi_t\,dV_t^{(m)}dt.
\end{equation}
We next examine the graphical part and 
non-graphical part of each terms. Recall 
that we have $U^{(m)}\subset\bC\cap P_1$
such that $\bC\cap P_1\setminus U^{(m)}
\subset \{|x|<\sigma^{(m)}\}$ given by
Theorem \ref{thm:graphical}. We may assume
that $\sigma^{(m)}<\sigma$, so that we have
\eqref{psizero} on the non-graphical part.
Since $\nabla\psi={\bf S}(\nabla\psi)$
for $|x|<\sigma$, we have
\begin{equation}\label{sma-a2}
\begin{split}
    \nabla^2((a\cdot x)\psi)\cdot S&=2\,{\bf S}^\perp(a)\otimes {\bf S}(\nabla\psi)\cdot(S-I)+(a\cdot x)\nabla^2\psi\cdot S\\
    &=-2a\cdot \sum_{j=1}^{k-1}\psi_{y_j}S^{\perp}(e_{n-k+j+1})+(a\cdot x)\nabla^2\psi\cdot S.
    \end{split}
\end{equation}
Thus 
\begin{equation}\label{sma-a}
\begin{split}
    &\int_{-1}^{\hat t}\int_{B_1\setminus {\rm graph}\,f^{(m)}(\cdot,t)}|\nabla^2((a\cdot x)\psi)\cdot S|\,dV_t^{(m)}dt \\ &\leq 4\Big(\int_{-1}^{\hat t}\int_{B_1}\sum_{j=1}^{k-1}|\psi_{y_j}S^\perp(e_{n-k+j+1})|^2\,dV_t^{(m)}dt\int_{-1}^0\|V_t^{(m)}\|(B_1\cap\{|x|\leq \sigma^{(m)}\})\,dt\Big)^{\frac12}\\
    &\,\,+\sup|\nabla^2\psi|\Big(\iint_{P_1\setminus{\rm graph}f^{(m)}} |x|^2\,d\|V_t^{(m)}\|dt\int_{-1}^0\|V_t^{(m)}\|(B_1\cap\{|x|\leq \sigma^{(m)}\})\,dt\Big)^{\frac12}.
\end{split}
\end{equation}
By \eqref{kore} and \eqref{para1}, as well
as \eqref{e:uniform-area-bound}, we may 
conclude that the right-hand side of
\eqref{sma-a} is $o(\mu^{(m)})$. The second
term of \eqref{sma-a1} can be handled exactly 
as the second term of \eqref{sma-a2}, and 
we may conclude that the non-graphical part
of \eqref{sma-a1} is $0$ in the limit, so that
we have
\begin{equation}\label{sma-a4}
      0\leq \liminf_{m\rightarrow\infty}\frac{1}{\mu^{(m)}}\int_{-1}^{\hat t}\int_{{\rm graph}f^{(m)}(\cdot,t)}
    -\nabla^2((a\cdot x) \psi)\cdot S +(a\cdot x)\psi_t\,dV_t^{(m)}dt.
\end{equation}
We write the above quantity as an integral over $U$.
We start by claiming that
\begin{equation}\label{sma-a3}
    \int_{-1}^{\hat t}\int_{U^{(m)}}|x||\nabla f^{(m)}|^2\,d\mathcal H^k dt=\mathrm{o}(\mu^{(m)}).
\end{equation}
For this, we note first that, for any point
$(X,t)\in \overline{U^{(m)}}\cap \{|x|\geq 2\sigma^{(m)}\}$, and
assuming without loss of generality that $P_1\cap \{|x|\geq \sigma^{(m)}\}\subset U^{(m)}$, the estimate in Remark \ref{rmk:C1a-deg} implies that
\begin{equation}\label{sma-a5}
\sup_{P_{\sigma^{(m)}/4}(X,t)}\Big((\sigma^{(m)})^{-1}|f^{(m)}|+|\nabla f^{(m)}|\Big)
\leq C(\sigma^{(m)})^{-\frac{k+4}{2}}\mu^{(m)}. 
\end{equation}
Thus, using also $|\nabla f^{(m)}|\leq \beta$,
\begin{equation}\begin{split}
    \int_{-1}^{\hat t}\int_{U^{(m)}}&|x||\nabla f^{(m)}|^2
    \leq \int_{-1}^{\hat t}\int_{U^{(m)}\cap
    \{|x|\geq 2\sigma^{(m)}\}}C^2(\sigma^{(m)})^{-k-4}(\mu^{(m)})^2 \\
    &+\Big(\int_{-1}^{\hat t}\int_{U^{(m)}\cap \{|x|<2\sigma^{(m)}\}}|x|^2|\nabla f^{(m)}|^2\Big)^{\frac12}c(k)(\sigma^{(m)})^{\frac12}\\ 
    &\leq C\Big((\sigma^{(m)})^{-k-4}(\mu^{(m)})^2+(\sigma^{(m)})^{\frac{1}{2}}\mu^{(m)}\Big)\,. 
    \end{split}\label{mat2}
\end{equation}
We may assume that $\lim_{m\rightarrow \infty}(\sigma^{(m)})^{-k-4}\mu^{(m)}=0$,
thus we proved \eqref{sma-a3}. 
For the second term of \eqref{sma-a4},
\begin{equation}\begin{split}
&\int_{-1}^{\tilde t}\int_{{\rm graph}f^{(m)}(\cdot,t)}  (a\cdot x)\psi_t\,dV_t^{(m)}dt  \\&
=\sum_{j=1}^3 \int_{-1}^{\hat t}\int_{{\bf H}_j\cap U^{(m)}} \Big((a\cdot w_j)r+a\cdot f^{(m)}\Big)\psi_t(\sqrt{r^2+|f^{(m)}|^2},y,t)J_{\nabla f^{(m)}}\,d\mathcal H^k dt.
\end{split}
\end{equation}
Since $J_{\nabla f^{(m)}}=1+\mathrm{O}(|\nabla f^{(m)}|^2)$ and $\psi_t(\sqrt{r^2+|f^{(m)}|^2},y,t)=\psi_t(r,y,t)+\mathrm{O}(|f^{(m)}|^2)$,
using \eqref{sma-a3}, the above may be 
continued as
\begin{equation}
    =\sum_{j=1}^3 \int_{-1}^{\hat t}\int_{{\bf H}_j\cap U^{(m)}} \Big((a\cdot w_j)r+a\cdot f^{(m)}\Big)\psi_t(r,y,t)\,d\mathcal H^k dt+\mathrm{o}(\mu^{(m)}).
\end{equation}
Since $\sum_{j=1}^3 w_j=0$, 
the first term vanishes. 
By the $L^2$ convergence of $f^{(m)}/\mu^{(m)}\rightarrow \tilde f$ proved in \eqref{strong-L2}, we obtain
\begin{equation}\label{mat5}
\lim_{m\rightarrow\infty}\frac{1}{\mu^{(m)}}\int_{-1}^{\tilde t}\int_{{\rm graph}f^{(m)}(\cdot,t)}  (a\cdot x)\psi_t\,dV_t^{(m)}dt 
=\int_{\bC} (a\cdot \tilde f)\,\psi_t\,d\mathcal H^k dt.    
\end{equation}
For the first term of \eqref{sma-a4}, 
we consider the terms
$(a\otimes\nabla\psi)\cdot S$
and $(a\cdot x)\nabla^2\psi\cdot S$ 
separately. 
By arguing similarly 
to \eqref{sma-a}, we may 
conclude that
\begin{equation}\label{mat6}
    \int_{-1}^{\hat t}\int_{{\rm graph}(f^{(m)}\mres_{U^{(m)}\cap\{|x|\leq 2\sigma^{(m)}\}})(\cdot,t)} (a\otimes\nabla\psi)\cdot S\,dV_t^{(m)}dt=\mathrm{o}(\mu^{(m)}).
\end{equation}
On $U^{(m)}\cap\{|x|>2\sigma^{(m)}\}$, 
using \eqref{sma-a5}, we may 
conclude that $J_{\nabla f^{(m)}}=
1+\mathrm{o}(\mu^{(m)})$ and
\begin{equation}
    \nabla\psi(x+f^{(m)}(x,y,t),y,t)
    =\nabla\psi(x,y,t)+f^{(m)}(x,y,t)\cdot\nabla^2\psi(x,y,t)+\mathrm{o}(\mu^{(m)}).
\end{equation}
Here on ${\bf H}_j\cap\{|x|>2\sigma^{(m)}\}$, choosing a coordinate such 
that $w_j$ point $x_1$-direction,
one can check by the direct 
calculation that
\begin{equation}\begin{split}
&\psi_{x_1 x_1}=\psi_{rr}+\mathrm{o}(\mu^{(m)}),\,
\psi_{x_1 x_l}=\mathrm{O}(\mu^{(m)} /(\sigma^{(m)})^{\frac{k}{2}+2}) \mbox{ for }l\in\{2,\ldots,n-k+1\}, \\
&\psi_{x_l x_{l'}}=\delta_{l\,l'}r^{-1}\psi_r+\mathrm{o}(\mu^{(m)}) \mbox{ for }l,l'\in\{2,\ldots,n-k+1\}.
\end{split}\label{mat1}
\end{equation}
Since $f^{(m)}\in {\bf H}_j^\perp$, we have
\begin{equation}
    \nabla\psi(x+f^{(m)}(x,y,t),y,t)=\nabla\psi(x,y,t)+\frac{\psi_r}{r}f^{(m)}(x,y,t)+\mathrm{o}(\mu^{(m)}).
\end{equation}
Since the projection to the tangent space to the graph of $f^{(m)}$ is 
\begin{equation}
    S={\bC}+\bC^\perp\circ\nabla f^{(m)}\circ{\bf C}+\bC\circ (\nabla f^{(m)})^T\circ\bC^\perp+\mathrm{O}(|\nabla f^{(m)}|^2), 
\end{equation}
and on $U^{(m)}\cap {\bf H}_j\cap\{|x|>2\sigma^{(m)}\}$,
we can compute that
\begin{equation}\begin{split}
   & (a\otimes\nabla\psi)\cdot \bC=(a\cdot w_j)\psi_r+\mathrm{o}(\mu^{(m)}), \\
   & (a\otimes\nabla\psi)\cdot 
(\bC^\perp\circ\nabla f^{(m)}\circ\bC)= (a\cdot w_j) \nabla_{w_j} f^{(m)}\cdot \nabla\psi +\mathrm{o}(\mu^{(m)})=\mathrm{o}(\mu^{(m)}),
\\
& (a\otimes\nabla\psi)\cdot(\bC\circ
(\nabla f^{(m)})^T\circ \bC^\perp)=
\nabla (f^{(m)}\cdot a)\cdot\nabla\psi+\mathrm{o}(\mu^{(m)}),
\end{split}
\end{equation}
where the second line used the 
fact that $\nabla_{w_j} f^{(m)}\in\bC^\perp$ and $v\cdot\nabla\psi=0$ for $v\in \bC^\perp$
on $\bC$. 
Thus using again $\sum_{j=1}^3 w_j=0$,
\begin{equation}
\begin{split}
    &\lim_{m\rightarrow\infty}\frac{1}{\mu^{(m)}}\int_{-1}^{\hat t}\int_{{\rm graph}(f^{(m)}\mres_{U^{(m)}\cap\{|x|> 2\sigma^{(m)}\}})(\cdot,t)} (a\otimes\nabla\psi)\cdot S\,dV_t^{(m)}dt
    \\
    &=\lim_{m\rightarrow\infty}\frac{1}{\mu^{(m)}}\int_{\bC\cap\{|x|>2\sigma^{(m)}\}}
    \nabla(f^{(m)}\cdot a)\cdot\nabla\psi\,d\mathcal H^{k}dt\\
    &=\int_{\bC} \nabla(\tilde f\cdot a)\cdot\nabla\psi\,d\mathcal H^k dt.\end{split}\label{mat7}
\end{equation}
The last line used the fact that
$\psi_r$ vanishes near the origin and that $\nabla_y f^{(m)}/\mu^{(m)}$ converges 
to $\nabla_y \tilde f$ weakly in $L^2$ as a consequence of \eqref{eq:bound-y-derivative}. 

For the term $(a\cdot x)\nabla^2\psi\cdot S$,
using \eqref{mat1},
we have on $U^{(m)}\cap {\bf H}_j\cap\{|x|>2\sigma^{(m)}\}$
\begin{equation}\label{mat3}
  \begin{split}  
  &(a\cdot x)\nabla^2\psi\cdot \bC=\{(a\cdot e_j)r+(a\cdot f^{(m)})\}(\psi_{rr}+\Delta_y \psi)+\mathrm{o}(\mu^{(m)}), \\
  & \nabla^2\psi\cdot (\bC^\perp\circ\nabla f^{(m)}\circ\bC)=\nabla^2\psi \cdot (\bC\circ
(\nabla f^{(m)})^T\circ \bC^\perp) =\mathrm{o}(\mu^{(m)}).
  \end{split}
\end{equation}
On $U^{(m)}\cap {\bf H}_j\cap\{|x|<2\sigma^{(m)}\}$,
since $\psi$ is independent of $x$, 
\begin{equation}
    (a\cdot x)\nabla^2\psi\cdot S
    =\{(a\cdot w_j)r+(a\cdot f^{(m)})\}
    (\Delta_y\psi+\mathrm{O}(|\nabla f^{(m)}|^2)).
\end{equation}
Thus 
\begin{equation}\begin{split}
    &\Big|\int_{-1}^{\hat t}\int_{{\rm graph}(f^{(m)}\mres_{U^{(m)} \cap\{|x|\leq 2\sigma^{(m)}\}})(\cdot,t)} (a\cdot x)\nabla^2\psi\cdot S\,dV_t^{(m)}dt
    \\
    &-\int_{U^{(m)}\cap\{|x|\leq 2\sigma^{(m)}\}}\{(a\cdot x)+(a\cdot f^{(m)})\}\Delta_y\psi\,d\mathcal H^k dt\Big| \\
    &\leq C\int_{U^{(m)}\cap \{|x|\leq 2\sigma^{(m)}\}}|x||\nabla f^{(m)}|^2\,d\mathcal H^k dt=\mathrm{o}(\mu^{(m)})
    \end{split}\label{mat4}
\end{equation}
where the last claim follows from \eqref{sma-a3}. Combining \eqref{mat3} and \eqref{mat4}, and using again 
$\sum_{j=1}w_j=0$, we obtain
\begin{equation}\label{mat8}
    \lim_{m\rightarrow\infty}\frac{1}{\mu^{(m)}}\int_{U^{(m)}} (a\cdot x)\nabla^2\psi\cdot S\,dV_t^{(m)}dt=\int_{\bC} (a\cdot\tilde f)\Delta\psi\,d\mathcal H^k\,dt.
\end{equation}
Since $\nabla^2((a\cdot x)\psi\cdot S
=2(a\otimes\nabla\psi)\cdot S+(a\cdot x)\nabla^2\psi\cdot S$, \eqref{sma-a4},
\eqref{mat5}, \eqref{mat6}, 
\eqref{mat7} and \eqref{mat8} show
\begin{equation}
    0\leq \int_{\bC}-2\nabla(a\cdot \tilde f)\cdot\nabla\psi-(a\cdot\tilde f)\Delta \psi+(a\cdot \tilde f)\psi_t\,d\mathcal H^k dt\,.
\end{equation}
We can perform the integration by parts 
for the second term since $\nabla_y \tilde f\in L^2$ and $\psi_r=0$ near $r=0$,
and since $a$ can be replaced by $-a$,
the inquality must hold with equality.
This finally proves
\begin{equation}
    0=\int_{\bC}\psi_t(a\cdot\tilde f)-\nabla(a\cdot \tilde f)\cdot \nabla\psi\,d\mathcal H^k dt
\end{equation}
for any $a\in \R^{n-k+1}\times\{0_{k-1}\}$ and any non-negative $\psi$ with \eqref{psizero}.
In terms of $\tilde f_j$, this implies
\begin{equation}
    0=\sum_{j=1}^3 \int_{{\bf H}_1} 
    \psi_t\,({\bf R}_{-2\pi(j-1)/3}(a)\cdot \tilde f_j)-\nabla({\bf R}_{-2\pi(j-1)/3}(a)\cdot\tilde f_j)\cdot\nabla\psi\,d\mathcal H^k dt.
\end{equation}
If $a\in \{0_2\}\times \R^{n-k-1}\times\{0_{k-1}\}$, then ${\bf R}_{-2\pi(j-1)/3}(a)=a$, so that
for any $\ell\in \{3,\cdots,n-k+1\}$,
we have
\begin{equation}\label{mat10}
    0=\int_{{\bf H}_1}\psi_t\big(\sum_{j=1}^3 \tilde f_{j,\ell}\big)-\nabla\big(\sum_{j=1}^3\tilde f_{j,\ell}\big)\cdot\nabla\psi\,d\mathcal H^k dt.
\end{equation}
If we take $a=w_1\in \R^1\times\{0_{n-1}\}$,
then $a\cdot \tilde f_1=0$, ${\bf R}_{-2\pi/3}(a)\cdot \tilde f_2=-\sqrt{3}\tilde f_{2,2}/2$ and
${\bf R}_{-4\pi/3}(a)\cdot \tilde f_3=\sqrt{3}\tilde f_{3,2}/2$.
This shows that
\begin{equation}\label{mat9}
    0=\int_{{\bf H}_1}\psi_t(\tilde f_{2,2}-\tilde f_{3,2})-
    \nabla(\tilde f_{2,2}-\tilde f_{3,2})\cdot\nabla\psi\,d\mathcal H^k dt.
\end{equation}
Similarly, by taking $a=w_2$ and $w_3$, we have \eqref{mat9} for 
$\tilde f_{1,2}-\tilde f_{3,2}$
and $\tilde f_{1,2}-\tilde f_{2,2}$,
respectively. If \eqref{mat10}
holds for $\psi$ satisfying
\eqref{psizero}, then by the
well-known argument, the 
function $\sum_{j=1}^3\tilde f_{j,\ell}$ can be extended evenly
and is the $C^\infty$ solution of the 
heat equation, and similarly for \eqref{mat9}. This ends the 
proof. 
\end{proof}

Propositions \ref{p:odd} and \ref{p:even} allow us to draw the following fundamental consequence on the limit function $\tilde f$. It will be convenient to rotate the three branches $\tilde f_j$ back and define the functions $\hat f_j := \mathbf R_{2\pi(j-1)/3}\tilde f_j$, for $j\in\{1,2,3\}$. Notice that these functions are defined on $\mathbf{H_1}\cap P_1$, which is given coordinates $(r,0,y,t)$, and take values in $\mathbf P_j^\perp$.

\begin{corollary} \label{cor:newcone}
For every $j \in \{1,2,3\}$, the function $\hat f_j$ admits a smooth extension, still denoted $\hat f_j$, to the parabolic cylinder $\mathbf{P}_1 \cap P_1$. Such extension is a solution to the heat equation, and it satisfies uniform interior estimates of the form \begin{equation} \label{e:u-estimates0}
\sup_{\mathbf{P}_1\cap P_{\sfrac12}} |\partial_r^a\partial_y^b\partial_t^c \hat f_j| \leq C_{a,b,c}\,\|\hat f\|_{L^2(\bC\cap P_1)}\,,
\end{equation}
for all indices $a,c$, and for every $(k-1)$-multiindex $b$. Moreover, there exist vectors $\tilde\xi_0 \in \bS(\bC)^\perp$, $v_j \in \mathbf{P}_j^\perp$, and a linear map $b \colon \bS(\bC) \to \bS(\bC)^\perp$ so that
\begin{equation} \label{newcone-stationary}
    v_1 + v_2 + v_3 = 0\,,
\end{equation}
and
\begin{equation} \label{newcone-decay}
    \begin{split}
  & \rho^{-(k+4)}\sum_{j=1}^3  \int_{-\rho^2}^0\int_{\{r^2+|y|^2 < \rho^2\}} \left| \hat f_j(r,y,t)-\mathbf{P}_j^\perp (\tilde\xi_0)-v_j \,r - \mathbf{P}_j^\perp (b(y))\right|^2\, d\mathcal H^k\,dt \\
  &\qquad \qquad \leq C\, \rho^2 \iint_{\bC\cap P_1} |\tilde f|^2\,d\mathcal H^k\,dt\,,
\end{split}
\end{equation}
for all $0 < \rho < \sfrac12$, where $C$ is a constant depending only on $k$ and $n$.

\end{corollary}

\begin{proof}
    First observe that for every $\ell \in \{2,\ldots,n-k+1\}$ the following identities hold:
    \[
    \begin{split}
        \tilde f_{1,\ell} &= \frac13 \sum_{j=1}^3 \tilde f_{j,\ell} + \frac13(\tilde f_{1,\ell}-\tilde f_{2,\ell}) + \frac13(\tilde f_{1,\ell}-\tilde f_{3,\ell}) \\
        \tilde f_{2,\ell} &= \frac13 \sum_{j=1}^3 \tilde f_{j,\ell} + \frac13(\tilde f_{2,\ell}-\tilde f_{1,\ell}) + \frac13(\tilde f_{2,\ell}-\tilde f_{3,\ell})\\
        \tilde f_{3,\ell} &= \frac13 \sum_{j=1}^3 \tilde f_{j,\ell} + \frac13(\tilde f_{3,\ell}-\tilde f_{1,\ell}) + \frac13(\tilde f_{3,\ell}-\tilde f_{2,\ell})\,.
    \end{split}
    \]
This shows that each function $\tilde f_{j,\ell}$ can be written as a finite sum of functions which admit an extension solving the heat equation across the spine of the cone. As a consequence, each $\tilde f_{j,\ell}$ is smooth up to $\overline{\mathbf{H}_1} \cap P_1$, with uniform estimates
\begin{equation} \label{e:u-estimates1}
\sup_{\overline{\mathbf{H}_1}\cap P_{\sfrac12}} |\partial_r^a\partial_y^b\partial_t^c \tilde f_j| \leq C\,\|\tilde f\|_{L^2(\bC\cap P_1)}\,,
\end{equation}
for any choice of indices $a,c$ and multi-index $b$. Furthermore, Proposition \ref{p:even} guarantees that 
\[
\begin{split}
&\left.\partial_r \tilde f_{1,2}\right|_{r=0}=\left.\partial_r \tilde f_{1,2}\right|_{r=0}=\left.\partial_r \tilde f_{1,2}\right|_{r=0} \\
\mbox{and} \quad &\sum_{j=1}^3 \left.\partial_r\tilde f_{j,\ell}\right|_{r=0}=0 \quad \mbox{for every $\ell\in\{3,\ldots,n-k+1\}$}\,.
\end{split}
\]
In particular, after rotation we have for the functions $\hat f_j$ that
\begin{equation} \label{stationarity0}
    \sum_{j=1}^3 \left. \partial_r \hat f_j\right|_{r=0} = 0\,,
\end{equation}
where the latter now is an identity between vectors in $\R^n$. This implies that the ``average'' map
\[
\omega := \sum_{j=1}^3 \hat f_j
\]
satisfies $\partial_r\omega=0$ on the spine $\bS(\bC)\cap P_1$, so that its even extension in the $r$-variable is a smooth solution to the heat equation, satisfying uniform estimates of the form
\begin{equation} \label{e:u-estimates2}
\sup_{\mathbf{P}_1\cap P_{\sfrac12}}|\partial_r^a\partial_y^b\partial_t^c\omega| \leq C\,\|\tilde f\|_{L^2(\bC\cap P_1)} \,.
\end{equation}
Notice also that $\omega$ takes values in the orthogonal complement $\bS(\bC)^\perp$ to the spine $\bS(\bC)$ of $\bC$. Now, it is an immediate consequence of Proposition \ref{blow-up-basic}(iii) that
\begin{equation} \label{e:spine motion}
     \hat f_j(0,y,t)=\mathbf{P}_j^\perp\tilde\xi_{(y,t)} \qquad \mbox{for all $(y,t) \in B^{k-1}_1 \times (-1,0)$}\,,
\end{equation}
and thus that
\begin{equation} \label{xi and the average}
    \omega (0,y,t) = \mathbf{P} \tilde\xi_{(y,t)}\,,
\end{equation}
where we have set $\mathbf{P}:=\sum_{j=1}^3 \mathbf{P}_j^\perp$. The operator $\mathbf P$ can be easily calculated, and in fact one immediately sees that, in the coordinates $(x_1,x_2,x_3,\ldots,x_{n-k+1},y_1,\ldots,y_{k-1})$ we have $\mathbf{P}={\rm diag}\left( \sfrac32,\sfrac32,3,\ldots,3,0,\ldots,0 \right)$. However, we will never use the precise form of the operator. What is important is that $\mathbf{P} \colon \R^n \to \R^n$ is self-adjoint, with kernel $\ker(\mathbf{P})=\bS(\bC)$ and image $W = \ker(\mathbf P)^\perp=\bS(\bC)^\perp$ that gets mapped bijectively onto intself. We denote by $\left.\mathbf{P}\right|_W^{-1}$ the inverse of the restriction $\left.\mathbf{P}\right|_W \colon W \to W$, and then we set $\mathbf{L}:=\left.\mathbf{P}\right|_W^{-1} \circ \mathbf{p}_W$, where $\mathbf{p}_W$ is the orthogonal projection of $\R^n$ onto $W$. Observe that, since $W = \bS(\bC)^\perp$, $\mathbf{L}$ maps $\R^n$ on $\bS(\bC)^\perp$. Moreover, if $w \in W$ then $\mathbf{P}(\mathbf{L}(w))=w$. In particular, thanks to \eqref{xi and the average} it holds
\[
\mathbf{P}(\mathbf{L}(\omega(0,y,t))) = \omega(0,y,t) = \mathbf{P} \tilde\xi_{(y,t)}\,,
\]
so that $\mathbf{L}(\omega(0,y,t))-\tilde\xi_{(y,t)} \in \bS(\bC)^\perp\cap\ker(\mathbf{P})=\{0\}$. We then conclude that
\[
\hat f_j (0,y,t)=\mathbf{P}_j^\perp(\mathbf{L}(\omega(0,y,t)))\,.
\]
Thus, the function
\[
u_j(r,y,t) := \hat f_j(r,y,t) - \mathbf{P}_j^\perp(\mathbf{L}(\omega(r,y,t)))
\]
has zero trace on the spine $\bS(\bC) \cap P_1$, so that its odd extension in the $r$-variable is a smooth solution to the heat equation satisfying uniform estimates of the form 
\begin{equation} \label{e:u-estimates3}
\sup_{\mathbf{P}_1\cap P_{\sfrac12}}|\partial_r^a\partial_y^b\partial_t^c u_j| \leq C\,\|\tilde f\|_{L^2(\bC\cap P_1)} \,,
\end{equation}
with the usual meaning for $a,b,c$. We have then established that
\[
\hat f_j = u_j + \mathbf{P}_j^\perp \circ \mathbf{L}(\omega) \qquad \mbox{on $\{(r,y,t)\,\colon\, r >0,\, r^2+|y|^2 < \sfrac14,\, t \in \left( -\sfrac14,0\right)\}$}\,,
\]
 where $u_j$ is the restriction of an odd function and $\mathbf{P}_j\circ \mathbf{L}(\omega)$ is the restriction of an even function. This allows one to extend $\hat f_j$ to the whole parabolic cylinder $\{(r,y,t)\,\colon\,  r^2+|y|^2 < \sfrac14,\, t \in \left( -\sfrac14,0\right)\}$ in such a way that the same identity is preserved. The estimates \eqref{e:u-estimates0} are then an immediate consequence of \eqref{e:u-estimates2}-\eqref{e:u-estimates3}. By Taylor's theorem and using again the estimates \eqref{e:u-estimates2}-\eqref{e:u-estimates3} combined with the fact that the $r$-derivative of $\omega$ and the $y$-derivative of $u_j$ vanish at $r=0$ we obtain \eqref{newcone-decay} upon setting 
\begin{align} \label{e:newcone-transl}
    \tilde \xi_0 :&= \mathbf{L} \,\omega(0) = \tilde\xi_{(0,0)} \in \bS(\bC)^\perp \\ \label{e:newcone-graph}
    v_j:&= \partial_r u_j(0) \in \mathbf{P}_j^\perp\,,
\end{align}
and where $b \colon \bS(\bC) \to \bS(\bC)^\perp$ is the linear map 
\begin{equation}\label{e:newcone-tilt}
    b(y) := \mathbf{L}(\nabla_y\omega(0) \cdot y)\,.
\end{equation}
Finally, \eqref{newcone-stationary} follows from \eqref{stationarity0} upon observing that $\left.\partial_r u_j \right|_{r=0}=\left.\partial_r f_j \right|_{r=0}$ because of the properties of $\omega$. 
\end{proof}

\section{Excess decay and proof of Theorem \ref{thm:main}} \label{sec:proof-main}

In this section we prove our main result, Theorem \ref{thm:main}. The key is the following excess decay theorem.

\begin{theorem} \label{thm:decay}
Corresponding to $n,k,p,q,E_1,\Cr{conslice}$ there exist $\theta_\star \in \left( 0, \sfrac{1}{20}\right)$ as well as $\Cl[eps]{e-decay}$ and $\Cl[con]{c-decay}$ so that the following holds. Let $I=(-R^2,0]$, and suppose that $\left( \{V_t\}_{t \in I}, \{u(\cdot,t)\}_{t \in I} \right) \in \mathscr{N}_{\Cr{e-decay}}(U_R\times I)$ and (A1)-(A6) are all satisfied in $U_R \times I$. Recall the definition of $\alpha$ in \eqref{e:the-Holder-exp}, $\|u\|$ in \eqref{smuterm}, and $\mu$ in \eqref{def:L2-excess}. Then, there exist a vector $a \in \bS(\bC)^\perp$ and a rotation $O \in \mathrm{O}(n)$ such that
\begin{equation}\label{newcone-is-close}
    |a| + \|O-{\rm Id}\| \leq \Cr{c-decay} \mu\,,
\end{equation}
and upon setting $\bC':=a+O(\bC)$
   \begin{equation}\label{e:exc-decay}
       \left((\theta_\star R)^{-k-4}  \iint_{P_{\theta_\star R}(a,0)}
       \dist(X,\bC')^2 \, d\|V_t\| (X) \, dt \right)^{\frac12} \leq \theta_\star^\alpha \max\{\mu, \Cr{c-decay} \|u\|\}\,.
   \end{equation} 
Furthermore, let $\iota_{a,\theta_\star}(X):=\theta_\star^{-1}(X-a)$, and define 
\[
V^\star_t := (O^{-1})_\sharp(\iota_{a,\theta_\star})_\sharp V_{\theta_\star^2 t}\qquad \mbox{and} \qquad u^\star(X,t):=\theta_\star\, O^{-1}u(a+\theta_\star O(X)\,,\,\theta_\star^2 t)\,.
\]
Then, the flow $\left(\{V^\star_t\}_{t\in I}, \{u^\star(\cdot,t)\}_{t \in I} \right)$ is also in $\mathscr{N}_{\Cr{e-decay}}(U_R\times I)$ and (A1)-(A6) are satisfied in $U_R \times I$.
\end{theorem}

\begin{proof}
    We may assume $R=5$ after a parabolic change of variables. The proof is by contradiction. Suppose the result is false. Then, we may consider sequences $\{V_t^{(m)}\}_{t \in I}$ and $u^{(m)}$ such that $(\{V_t^{(m)}\}_{t \in I}, \{u^{(m)}(\cdot, t)\}_{t\in I}) \in \mathcal{N}_{\sfrac1m}(U_5\times I)$, conditions (A1)-(A6) are all satisfied, and with the following additional property. For any triple junction $\bC'=a+O(\bC)$ with 
    \begin{equation} \label{e:constraint-cones}
    |a|+\|O-\mathrm{Id}\| \leq m \mu^{(m)}
    \end{equation}
    it holds 
    \begin{equation}\label{e:contradiction-decay}
       \left((5\theta_\star)^{-k-4} \iint_{P_{5\theta_\star}(a,0)} 
       \dist(X,\bC')^2 \, d\|V_t^{(m)}\| (X) \, dt \right)^{\frac12} > \theta_\star^\alpha \max\{\mu^{(m)}, m\|u^{(m)}\|\}\,.
   \end{equation}
   We will show that \eqref{e:contradiction-decay} is inconsistent for suitable choices of $\bC'=\bC'^{(m)}$ and of $\theta_\star$ depending, the latter, only on $n,k,p,q,E_1$. First, we claim that $(\{V_t^{(m)}\}_{t \in I}, \{u^{(m)}(\cdot, t)\}_{t\in I})$ is a blow-up sequence in the sense of Definition \ref{def:BU}, corresponding to the choice $\varepsilon^{(m)}=\sfrac{1}{m}$. The only condition to check is the validity of the second limit in \eqref{blow-up cond}. From \eqref{e:contradiction-decay} with $\bC'=\bC$ we see that
   \[
   \theta_\star^\alpha\, m\|u^{(m)}\| < (5\theta_\star)^{-\frac{k+4}{2}} \left(\iint_{P_1}\dist(X,\bC)^2\,d\|V_t^{(m)}\|(X)\,dt \right)^{\frac12} \leq (5\theta_\star)^{-\frac{k+4}{2}} \mu^{(m)}\,,
   \]
   from which it indeed follows that $(\mu^{(m)})^{-1}\|u^{(m)}\| \to 0$ as $m \to \infty$. Then, all the arguments of Section \ref{sec:BU} apply. Recall the definition of the numbers $\sigma^{(m)}$, of the domains $U^{(m)}$, of the functions $f^{(m)}$ and $\xi^{(m)}$, as well as of their limits upon renormalization by dividing by $\mu^{(m)}$. Consider now the vectors $\tilde\xi_0 \in \bS(\bC)^\perp$, $v_j \in \mathbf{P}_j^\perp$, and the linear map $b \colon \bS(\bC) \to \bS(\bC)^\perp$ from Corollary \ref{cor:newcone}. We now define the cone $\bC'^{(m)}$ leading to the contradiction. First, for every $j$ parameterize the half-plane $\mathbf{H}_j = \mathrm{span}^+(w_j) \oplus \bS(\bC)$ with coordinates $(r,y)$, where $r=|x|$ is the distance function from $\bS(\bC)$, and consider the graph of the map $l_j^{(m)} \colon \mathbf{H}_j \to \mathbf{P}_j^\perp$ defined by $l_j^{(m)}(r,y):=r\,\mu^{(m)}\, v_j$. Notice that each graph is a linear half-plane of dimension $k$ in $\R^n$ that contains $\bS(\bC)$. Furthermore, since $v_1+v_2+v_3=0$ the union of the three graphs is a standard triple junction $\bC_1^{(m)}$ with the spine $\bS(\bC)$, and clearly $\bC_1^{(m)}=O_1^{(m)}(\bC)$, where $O_1^{(m)}$ is an orthogonal transformation in $\R^n$ that keeps $\bS(\bC)$ fixed and satisfies $\|O_1^{(m)}-\mathrm{Id}\| \leq C\mu^{(m)}$. Next, consider the linear map $b$, and let $b^*\colon \bS(\bC)^\perp\to\bS(\bC)$ be its adjoint. Define $\tilde b \colon \bS(\bC)^\perp \oplus \bS(\bC) \simeq \R^n \to \R^n$ by $\tilde b(x,y) := b(y)-b^*(x)$. The vector field $\tilde b$ is then the infinitesimal generator of a one-parameter family of linear transformations of $\R^n$, indexed by the parameter $\tau$ and denoted $O(\tau,\cdot)$, namely such that
   \begin{equation} \label{ODE}
       \begin{cases}
           & O(0,x,y)=(x,y), \\
           & \partial_{\tau} O(0,x,y)=\tilde b(x,y)\,.
       \end{cases}
   \end{equation}
   If $B$ denotes the matrix that defines $\tilde b$, then $O(\tau,\cdot)$ is the linear transformation whose matrix is $\exp(\tau B)$. On the other hand, $B$ is skew-symmetric by definition, and thus $\exp(\tau B)$ is orthogonal. In other words, $O$ is a one-parameter family of rotations, and we set $O_2^{(m)}:=O(\mu^{(m)},\cdot)$. Again we see immediately that $\|O_2^{(m)}-\mathrm{Id}\| \leq C \mu^{(m)}$. We define $O^{(m)}:=O^{(m)}_2\circ O^{(m)}_1$: notice that $O^{(m)}(\bC)$ is a standard triple junction with spine $O^{(m)}_2(\bS(\bC))$, and that $\|O^{(m)}-\mathrm{Id}\| \leq C \mu^{(m)}$. Finally, define $a^{(m)}:=\mu^{(m)}\tilde\xi_0$, and set $\bC'^{(m)}=a^{(m)}+O^{(m)}(\bC)$. We now proceed to estimate the quantity
   \[
   (5\theta_\star)^{-k-4}  \iint_{P_{5\theta_\star}(a^{(m)},0)} \dist(X,\bC'^{(m)})^2 \, d\|V_t^{(m)}\| (X) \, dt
   \]
   with this choice of $\bC'^{(m)}$. We shall work separately on the region $P_{5\theta_\star}(a^{(m)},0) \cap \{|x| > 6 \sigma^{(m)}\}$, where $\{V_t^{(m)}\}$ can be represented as a graph over $\bC$, and on the complement region $P_{5\theta_\star}(a^{(m)},0)\cap \{|x| \leq 6 \sigma^{(m)}\}$. For the first part, let $t \in (-(5\theta_\star)^2,0)$, let $X \in \spt\|V_t\|\cap  B_{5\theta_\star}(a^{(m)}) \cap \{|x| > 6 \sigma^{(m)}\} $, and let $j \in \{1,2,3\}$ and $Z=(z,w) \in \mathbf H_j$ be such that $X = Z + f^{(m)}(Z,t)$. Consider now the points
   \begin{align*}
       \hat Z &:= Z-\mathbf{P}_j(a^{(m)}+\mu^{(m)}b(w)) + \mu^{(m)}b^*(z)\,,\\
       \hat X &:= \hat Z + l^{(m)}_j(\hat Z)\,,\\
       X'&:= a^{(m)}+O^{(m)}_2(\hat X)\,.
   \end{align*}
   Notice that $Z \in \mathbf H_j$ lies in the domain of the map $f^{(m)}$: since this domain has distance from $\bS(\bC)$ of order $\mathrm{O}(\sigma^{(m)})$, since $|\mathbf{P}_j(a^{(m)}+\mu^{(m)}b(w)) + \mu^{(m)}b^*(z)| = \mathrm{O}(\mu^{(m)})$, and since $\mu^{(m)}=\mathrm{o}(\sigma^{(m)})$, the point $\hat Z$ still belongs to $\mathbf H_j$, at a distance from $\bS(\bC)$ comparable to that of $Z$. Then, by definition, $\hat X \in O_1^{(m)}(\mathbf H_j)$, and $X' \in \bC'^{(m)}$, with $\dist(X,\bC'^{(m)}) \leq |X-X'|$. On the other hand, $X-X'=(X-\hat X)+(\hat X - X')$, and we proceed to calculate and estimate each vector in the sum separately. We have, recalling the definition of $a^{(m)}$, $l^{(m)}_j$, and $\tilde f^{(m)}$:
   \[
   \begin{split}
       X-\hat X &= \mathbf{P}_j(a^{(m)}) + \mu^{(m)}\mathbf{P}_j(b(w))-\mu^{(m)}b^*(z)+f^{(m)}(Z,t)-l_j^{(m)}(Z)+\mathrm{O}((\mu^{(m)})^2)\\
       &=\mu^{(m)} \left( \tilde f^{(m)}(z,w,t) - |z|v_j+\mathbf{P}_j(\tilde\xi_0+b(w))-b^*(z) \right) + \mathrm{O}((\mu^{(m)})^2)\,.
    \end{split}
   \]
    On the other hand, since $\mu^{(m)}\tilde b$ is the infinitesimal generator of $O_{2}^{(m)}$, we have that
    \[
    \begin{split}
    \hat X - X' &= \hat X - (a^{(m)}+\hat X + \mu^{(m)}\tilde b(\hat X) + \mathrm{O}((\mu^{(m)})^2))\\
    &=-\mu^{(m)} \left( \tilde\xi_0 + b(w) -b^*(z) \right) + \mathrm{O}((\mu^{(m)})^2)\,,
    \end{split}
    \]
   where we have used that $\tilde b(\hat X - Z) = {\rm O}(\mu^{(m)})$. By combining the two estimates, we then have that
   \[
   X-X' = \mu^{(m)} \left( \tilde f^{(m)}(z,w,t) - |z|v_j-\mathbf{P}_j^\perp(\tilde\xi_0+b(w)) \right) + \mathrm{O}((\mu^{(m)})^2). 
   \]
    This implies, using \eqref{strong-L2}, that 
    \[
    \begin{split}
    &(\mu^{(m)})^{-2} (5\theta_\star)^{-k-4} 
    \iint_{P_{5\theta_\star}(a^{(m)},0)\cap\{|x|>6\sigma^{(m)}\}} \dist(X,\bC'^{(m)})^2 \, d\|V_t^{(m)}\|(X)\,dt\\
    &\qquad  \leq C \theta_\star^{-k-4} 
    \iint_{P_{6\theta_\star}\cap U^{(m)}}\left| \tilde f^{(m)}(z,w,t) - |z|v_j-\mathbf{P}_j^\perp(\tilde\xi_0+b(w)) \right|^2 \,d\mathcal H^k \, dt + \mathrm{O}((\mu^{(m)})^2)\,,
    \end{split}
    \]
    where $C$ depends only on $k$. Using that $\tilde f$ is the limit of $\tilde f^{(m)}$ as $m \to \infty$, we can take advantage of \eqref{newcone-decay} to estimate the right-hand side and, since the space-time $L^2$ norm of $\tilde f$ is finite as a consequence of \eqref{strong-L2}, we can deduce that 
    \[
    \limsup_{m \to \infty}(\mu^{(m)})^{-2} (5\theta_\star)^{-k-4} 
    \iint_{P_{5\theta_\star}(a^{(m)},0)\cap \{|x|>\sigma^{(m)}\}}\dist(X,\bC'^{(m)})^2 \, d\|V_t^{(m)}\|(X)\,dt \leq C \theta_\star^2\,
    \]
    where $C$ depends only on $n,k,p,q,E_1$. Next, observe that on $B_1$ we have $\dist(X,\bC'^{(m)}) \leq \dist(X,\bC) + C \mu^{(m)}$. Fix $\sigma \in \left( 0,\sfrac{1}{40} \right)$ to be chosen momentarily depending only on $n,k,p,q,E_1$. For all $m$ sufficiently large, then, the flow $\left(\{V_t^{(m)}\}_{t\in I}, \{u^{(m)}(\cdot,t)\}_{t\in I} \right)$ satisfies all the hypotheses of Proposition \ref{p:nonconc-excess} with, say, $\kappa=\sfrac14$, and furthermore it holds $6 \sigma^{(m)} \leq \sigma$. In particular, we have
    \begin{equation*}
    \begin{split}
        &  \limsup_{m \to \infty}(\mu^{(m)})^{-2} (5\theta_\star)^{-k-4} \iint_{P_{5\theta_\star}(a^{(m)},0)\cap\{|x|\leq 6\sigma^{(m)}\}}\dist(X,\bC'^{(m)})^2 \, d\|V_t^{(m)}\|(X)\,dt \\
        &\quad \leq C \,
        \theta_\star^{-3}
        \sigma + \sigma^{\sfrac12}\limsup_{m \to \infty}(\mu^{(m)})^{-2} (5\theta_\star)^{-k-4} 
        \iint_{P_{5\theta_\star}(a^{(m)},0)\cap \{|x|\leq \sigma\}}
        \frac{\dist(X,\bC)^2}{\sigma^{\sfrac12}} \, d\|V_t^{(m)}\|(X)\,dt \\
        & \quad \leq C 
        \theta_\star^{-3}
        \sigma 
        + \Cr{c-binding} (5\theta_\star)^{-k-4} \sigma^{\sfrac12}.
    \end{split}
    \end{equation*}
   Here we used $\|V_t^{(m)}\|(B_{5\theta_\star}(a^{(m)})\cap \{|x|\leq \sigma\})\leq c(k)E_1 \sigma \theta_\star^{k-1}$. Upon choosing $\sigma$ small enough, depending on the given constants (note that $\Cr{c-binding}$ does not depend on $\sigma$) and on $\theta_\star$, and thus only on $n,k,p,q,E_1,\Cr{conslice}$, we can ensure that the right-hand side is bounded by $\theta_\star^2$. Combining the two estimates, we conclude that for all $m$ large enough it holds 
   \begin{equation} \label{e:the-end}
   \left((5\theta_\star)^{-k-4}  
   \iint_{P_{5\theta_\star}(a^{(m)},0)}
   \dist(X,\bC'^{(m)})^2 \, d\|V_t^{(m)}\| (X) \, dt \right)^{\frac12} \leq C\,\theta_\star\,\mu^{(m)}\,,
   \end{equation}
    where $C$ depends only on $n,k,p,q,E_1,\Cr{conslice}$. This estimate is in apparent contradiction with \eqref{e:contradiction-decay} for a suitable choice of $\theta_\star$ depending only on the same data. This proves the first part of the theorem. The second part is immediate, upon possibly further decreasing $\theta_\star$ to entail $\theta_\star^\alpha \Cr{c-decay} < 1$. Indeed, (A1)-(A5) remain true by the scale invariance property of Brakke flows and the fact that $\|u^\star\|_R = \|u\|_{\theta_\star R} \leq \theta_\star^\alpha \|u\|_R$; the latter scaling property also immediately implies \eqref{smuterm}, whereas \eqref{def:L2-excess} follows from \eqref{e:exc-decay}; \eqref{def:mass deficit} and \eqref{non0} are guaranteed because for any $\sigma > 0$, upon choosing $\Cr{e-decay}$ sufficiently small, $V_t$ is a $C^{1,\alpha}$ graph over $\bC$ in $U_{\frac{5}{2}\theta_\star}(a)\setminus B_{\sigma}(\bS(\bC))$. Finally, Assumption (A6), by its multiscale formulation, is satisfied by the rescaled flow $\left( \{V_t^\star\}_{t \in I}, \{u^\star(\cdot,t)\}_{t \in I}\right)$ by hypothesis, and the proof is complete.
\end{proof}

By the second claim of Theorem \ref{thm:decay}, excess decay can be iterated through dyadic scales, and we reach the following.

\begin{proposition}\label{prop:iteration}
Corresponding to $n,k,p,q,E_1,\Cr{conslice}$ there exist positive numbers $\Cl[eps]{eps-iter} < \Cr{e-decay}$ and $\Cl[con]{con-iter}$ with the following property. Let $I = (-R^2,0]$ and suppose that $(\{V_t\}_{t \in I}, \{u(\cdot,t)\}_{t \in I})\in \mathscr{N}_{\Cr{eps-iter}}(U_R\times I)$ so that (A1)-(A6) are all satisfied. Then there exists a sequence of triple junctions $\{\bC_m\}_{m=0}^\infty$ of the form $\bC_m = a_m + O_m(\bC)$, and a limit triple junction $\bC_\infty = a_\infty + O_\infty(\bC)$ with $O_m, O_\infty \in \mathrm{O}(n)$, such that for every $m \ge 0$:
\begin{itemize}
    \item[(i)] The excess at scale $\theta_\star^m R$ decays:
    \begin{equation}\label{eq:excess-decay-iter}
        \mu_m := \left((\theta_\star^m R)^{-k-4} 
        \iint_{P_{\theta_\star^m R}(a_m,0)}
        \dist(X,\bC_m)^2\,d\|V_t\|\,dt \right)^{\frac12} \leq (\theta_\star^m)^\alpha \,\max\{\mu, \Cr{c-decay}\|u\|\}\,.
    \end{equation}
    \item[(ii)] The cones $\bC_m$ converge geometrically to $\bC_\infty$, namely
    \begin{equation}\label{eq:cone-conv-iter}
        \max\{\theta_\star^{-m} |a_m-a_\infty|, \|O_m-O_\infty\|\} \le \Cr{con-iter} (\theta_\star^m)^\alpha \max\{\mu, \Cr{c-decay}\|u\|\}\,,
    \end{equation}
    and furthermore the distance between $\bC_\infty$ and $\bC$ is estimated by
    \begin{equation} \label{eq:close-to-standard}
        \max\{|a_\infty|, \|O_\infty-\mathrm{Id}\|\} \leq \Cr{con-iter} \max\{\mu, \Cr{c-decay}\|u\|\}\,.
    \end{equation}
\end{itemize}
\end{proposition}

\begin{proof}
The proof is by induction. We set the sequence of scales $R_m := \theta_\star^m R$, and we claim that there exist sequences of vectors $\{a_m\}$, rotations $\{O_m\}$, cones $\{\bC_m\}$, flows $\{V_t^{(m)}\}$, and forcing terms $\{u^{(m)}\}$ so that the following hold for every $m \geq 0$:
\begin{itemize}
    \item[$(1)_m$] $\bC_m=a_m+O_m(\bC)$;
    \item[$(2)_m$] $V_t^{(m)}=(O_m^{-1})_\sharp(\iota_{a_m,\theta_\star^m})_\sharp V_{\theta_\star^{2m}t}$ and $u^{(m)}(X,t) = \theta_\star^m \,O_m^{-1}u(a_m + \theta_\star^m O_m(X), \theta_\star^{2m}t)$;
    \item[$(3)_m$] equation \eqref{eq:excess-decay-iter} holds;
    \item[$(4)_m$] $(\{V^{(m)}_t\},\{u^{(m)}(\cdot,t)\})$ satisfies the assumptions of Theorem \ref{thm:decay}.
\end{itemize}

\noindent\textbf{Base Case (m=0):}
We set $a_0 = 0$, $O_0 = \mathrm{Id}$, so $\bC_0 = \bC$ and $(1)_0$ is satisfied. We define $V_t^{(0)} := V_t$ and $u^{(0)} := u$, so that $(2)_0$ is satisfied,. The excess $\mu_0$ satisfies \eqref{eq:excess-decay-iter} by definition, and the hypothesis $(\{V_t\},\{u(\cdot,t)\})\in \mathscr{N}_{\Cr{eps-iter}}(U_R\times I)$ with $\Cr{eps-iter} < \Cr{e-decay}$ ensures that the conditions of Theorem \ref{thm:decay} are met for this initial setup.

\noindent\textbf{Inductive Step:}
Assume that for some $m \ge 0$, we have constructed $a_m, O_m, \bC_m, V_t^{(m)}, u^{(m)}$ such that conditions $(1)_m$-$(4)_m$ are all satisfied. We aim at defining $a_{m+1}$, $O_{m+1}$, $V_t^{(m+1)}$, and $u^{(m+1)}$. We apply Theorem \ref{thm:decay} to the flow $(\{V_t^{(m)}\}, \{u^{(m)}(\cdot,t)\})$. The theorem provides a new vector $a' \in \bS(\bC)^\perp$ and a rotation $O' \in \mathrm{O}(n)$ such that $|a'| + \|O'-\mathrm{Id}\|$ is controlled by the $L^2$-excess of the flow $V_t^{(m)}$ at scale $R$. Using $(3)_m$, we then have
\begin{equation}\label{eq:one-step-cone-close}
|a'| + \|O'-\mathrm{Id}\| \le \Cr{c-decay}  \mu_m\,.
\end{equation}
Furthermore, the theorem gives the following one-step excess decay: for $\bC':=a'+O'(\bC)$ it holds
\begin{equation}\label{eq:one-step-decay}
\left((\theta_\star R)^{-k-4} 
\iint_{P_{\theta_\star}(a',0)}
\dist(X,\bC')^2 \, d\|V_t^{(m)}\| (X) \, dt \right)^{\frac12}\le \theta_\star^\alpha \max\{\mu_m, \Cr{c-decay}\|u^{(m)}\|\}\,,
\end{equation}
where, of course, the quantity $\|u^{(m)}\|$ is computed as in \eqref{smuterm} using $V^{(m)}$. We can now define the new cone in the sequence: precisely, we set 
\begin{align}
    a_{m+1} &:= a_m + \theta_\star^m\, O_m(a') \label{eq:iter-a} \\
    O_{m+1} &:= O_m \circ O' \label{eq:iter-O} \\
    \bC_{m+1} &:= a_{m+1} + O_{m+1}(\bC) \label{eq:iter-C}\,.
\end{align}
The condition $(1)_{m+1}$ is then satisfied by definition. We also set  $V^{(m+1)}_t:=(V^{(m)}_t)^\star$ and $u^{(m+1)}:=(u^{(m)})^\star$ as in Theorem \ref{thm:decay}, using the new $a'$ and $O'$: explicitly, using $(2)_m$ as well as \eqref{eq:iter-a} and \eqref{eq:iter-O} we see after a simple algebraic calculation that
\begin{align*}
V^{(m+1)}_t &= ((O')^{-1})_\sharp (\iota_{a',\theta_\star})_\sharp V^{(m)}_{\theta_\star^2t} = (O_{m+1}^{-1})_\sharp (\iota_{a_{m+1},\theta_\star^{m+1}})_\sharp V_{\theta_\star^{2(m+1)}t}\,, \\
\qquad u^{(m+1)}(X,t) &= \theta_\star (O')^{-1} u^{(m)}(a'+\theta_\star O'(X),\theta_\star^2 t)  \\&= \theta_\star^{m+1} \,O_{m+1}^{-1} u(a_{m+1}+\theta_\star^{m+1}O_{m+1}(X),\theta_\star^{2(m+1)}t)\,,
\end{align*}
namely $(2)_{m+1}$ is satisfied. Theorem \ref{thm:decay} also guarantees that $(4)_{m+1}$ holds. Finally, using $(2)_m$ and the definition of $\bC_{m+1}$ one immediately sees that $\mu_{m+1}$ equals the left-hand side of \eqref{eq:one-step-decay}. In turn, \eqref{eq:one-step-decay} together with $(3)_m$ and the trivial estimate $\|u^{(m)}\|\leq \theta_\star^{m\alpha}\|u\|$ gives $(3)_{m+1}$. This completes the proof of the inductive claim. In particular, it proves the validity of \eqref{eq:excess-decay-iter} for every $m \geq 0$.

\smallskip

It remains to prove the geometric convergence of the cones and the estimate \eqref{eq:cone-conv-iter}. From the iterative definitions \eqref{eq:iter-a} and \eqref{eq:iter-O}, and the one-step estimate \eqref{eq:one-step-cone-close}, we can bound the distance between successive cones. For the rotations, we have:
\[
\|O_{m+1}-O_m\| = \|O_m \circ O' - O_m\| = \|O_m(O'-\mathrm{Id})\| = \|O'-\mathrm{Id}\| \le \Cr{c-decay}\mu_m\,.
\]
For the translations, we have:
\[
|a_{m+1}-a_m| = |\theta_\star^m O_m(a')| = \theta_\star^m |a'| \le \Cr{c-decay} \theta_\star^m \mu_m\,.
\]
By \eqref{eq:cone-conv-iter}, we see that for any $p > m$:
\begin{align*}
\|O_p - O_m\| &\le \sum_{j=m}^{p-1} \|O_{j+1}-O_j\| \le \Cr{c-decay} \sum_{j=m}^{p-1} \mu_j \le \Cr{c-decay}\left( \sum_{j=m}^{p-1} (\theta_\star^\alpha)^j\right) \max\{\mu, \Cr{c-decay}\|u\|\}, \\
|a_p - a_m| &\le \sum_{j=m}^{p-1} |a_{j+1}-a_j| \le \Cr{c-decay} \sum_{j=m}^{p-1} \theta_\star^j \mu_j \le \Cr{c-decay} \left(  \sum_{j=m}^{p-1} (\theta_\star^{1+\alpha})^j\right) \max\{\mu, \Cr{c-decay}\|u\|\}\,.
\end{align*}
Since $\theta_\star < 1$ and $\alpha>0$, both right-hand sides are tails of convergent geometric series, which go to zero as $m \to \infty$. Thus, $\{O_m\}$ and $\{a_m\}$ are Cauchy sequences. They converge to limits $O_\infty \in \mathrm{O}(n)$ and $a_\infty \in \mathbb{R}^n$, respectively. The limit cone is $\bC_\infty = a_\infty + O_\infty(\bC)$.

Letting $p \to \infty$ in the inequalities above, we get the estimate
\begin{align*}
\max\lbrace \theta_\star^{-m}|a_m-a_\infty| , \|O_m - O_\infty\| \rbrace &\le \Cr{c-decay}\left( \sum_{j=m}^{\infty} (\theta_\star^\alpha)^j\right) \max\{\mu, \Cr{c-decay}\|u\|\} \\
&=\Cr{c-decay} \frac{\theta_\star^{m\alpha}}{1-\theta_\star^\alpha} \max\{\mu, \Cr{c-decay}\|u\|\}\,,
\end{align*}
whereas summing the whole series gives
\begin{align*}
\max\lbrace |a_\infty| , \|O_\infty-\mathrm{Id}\| \rbrace &\le \Cr{c-decay}\left( \sum_{j=0}^{\infty} (\theta_\star^\alpha)^j\right) \max\{\mu, \Cr{c-decay}\|u\|\} \\
&=\Cr{c-decay} \frac{1}{1-\theta_\star^\alpha} \max\{\mu, \Cr{c-decay}\|u\|\}\,,
\end{align*}
Choosing $\Cr{con-iter}=(1-\theta_\star^\alpha)^{-1}\Cr{c-decay}$ the desired estimates \eqref{eq:cone-conv-iter} and \eqref{eq:close-to-standard} follow. This completes the proof.
\end{proof}

The following is an immediate corollary of Proposition \ref{prop:iteration}: it finally existence of a point close to the origin at time $t=0$ with a static triple junction tangent flow, as well as uniqueness and decay.

\begin{proposition} \label{prop:tangent-flow}
    Under the same assumptions of Proposition \ref{prop:iteration}, the following holds. For every $0 < s < R$ there are points $a_s$ and rotations $O_s$, with corresponding cones $\bC_s=a_s+O_s(\bC)$, as well as a point $a_\infty$ and a rotation $O_\infty$ with corresponding cone $C_\infty=a_\infty+O_\infty(\bC)$ such that
    \eqref{eq:close-to-standard} holds and furthermore
    \begin{equation}\label{eq:cont-ed}
        \left(s^{-k-4} \iint_{P_s(a_s,0)} \dist(X,\bC_s)^2\,d\|V_t\|(X)\,dt \right)^{\frac12} \leq \left( \frac{s}{R}\right)^{\alpha} \max\{\mu,\Cr{c-decay} \|u\|\}
    \end{equation}
    and 
    \begin{equation} \label{eq:cont-conv}
        \max\{(s/R)^{-1}|a_s-a_\infty|,\|O_s-O_\infty\|\} \leq \Cr{con-iter} \left( \frac{s}{R}\right)^{\alpha} \max\{\mu, \Cr{c-decay}\|u\|\}\,.
    \end{equation}
    In particular, $a_\infty \in \spt\|V_0\|$, $O_\infty(\bC)$ is the unique tangent flow to $\{V_t\}_t$ at the point $(a_\infty,0)$, and the parabolic blow-ups $(\iota_{a_\infty,\lambda})_\sharp V_{\lambda^2s}$ converge to the static $O_\infty(\bC)$ with rate $\lambda^\alpha$ as $\lambda \to 0^+$. Furthermore, if the Gaussian density $\Theta(0,0)\geq \sfrac32$ then the same conclusion holds with $a_\infty=0$.
\end{proposition}

\begin{proof}
    Conclusions \eqref{eq:cont-ed} and \eqref{eq:cont-conv} are an immediate consequence of \eqref{eq:excess-decay-iter} and \eqref{eq:cone-conv-iter}, respectively, upon interpolating between dyadic scales. The fact that $O_\infty(\bC)$ is the unique tangent flow at $(a_\infty,0)$ follows then from \eqref{eq:cont-ed} and \eqref{eq:cont-conv}. Finally, if $\Theta(0,0) \geq \sfrac32$ then $\tilde\xi_0=0$ in Corollary \ref{cor:newcone}, which in turn forces $a=0$ in Theorem \ref{thm:decay} at every iteration across scales, and thus $a_\infty=0$.
\end{proof}

We are now in the position to prove the Main Theorem \ref{thm:main}.

\begin{proof}[Proof of Theorem \ref{thm:main}]
By scale invariance, we can assume $R=4$. We shall divide the proof into steps.

\smallskip
\noindent \textbf{Step 1.} Suppose first that the Gaussian density $\Theta(0,0) \geq \sfrac{3}{2}$. By Proposition \ref{prop:tangent-flow}, assuming $\Cr{e_main} < \Cr{eps-iter}$ there exists a unique static triple junction tangent flow at $(0,0)$, which we denote $\bC_{(0,0)}$, and thus $\Theta(0,0)=\sfrac32$. Without loss of generality, up to possibly rotating the flow and the forcing field, we can assume that $\bC_{(0,0)}=\bC$. We claim then that the conclusion of Theorem \ref{thm:main} is valid in a (backward in time) parabolic cylinder centered at $(0,0)$, with the following additional information on the functions $\xi$ and $f_i$. Recalling that
\[
\bC= \bigcup_{i=1}^3 \mathbf{H}_{i}\,,
\]
and that $\bS_{(0,0)}:=\bS(\bC)$ is the spine of $\bC$, the parabolic blow-ups of $f_i$ at $(0,0)$ converge to $\mathbf{H}_{i}$ and the parabolic blow-ups of $\xi$ at $(0,0)$ is $\bS_{(0,0)}$.

To see this, we notice first that, upon choosing $\Cr{e_main}$ sufficiently small, for any $(\Xi,\tau) \in U_3(0)\times (-9,0)$ the assumptions of Proposition \ref{prop:tangent-flow} are satisfied for the flow $V^{(\Xi,\tau)}_t\,, u^{(\Xi,\tau)}(\cdot,t)$, where 
\begin{align*}
    V^{(\Xi,\tau)}_t :&= (\iota_{\Xi,\frac14})_\sharp V_{\tau+\frac{t}{16}} \\
    u^{(\Xi,\tau)}(X,t) :&= (\sfrac14)\, u(\Xi + (\sfrac14) X, \tau + t/16)\,.
\end{align*}
In particular, we have the following alternative for every point $(\Xi,\tau)$ in $U_3(0)\times(-9,0)$:
\begin{itemize}
    \item[(a)] either $\Theta(\Xi,\tau) < \sfrac32$,
    \item[(b)]  or $\Theta(\Xi,\tau) \geq \sfrac32$, and in this case we are again in the position of applying Proposition \ref{prop:tangent-flow}, conclude that in fact $\Theta(\Xi,\tau)=\sfrac32$, and determine the existence of a rotation $O_{(\Xi,\tau)}$ and a corresponding triple junction $\bC_{(\Xi,\tau)}=O_{(\Xi,\tau)}(\bC)$ with spine $\bS_{(\Xi,\tau)}=O_{(\Xi,\tau)}(\bS(\bC))$ so that
    \begin{equation} \label{eq:decay-nearby}
       \left( r^{-k-4} \iint_{P_r(\Xi,\tau)}\dist(X-\Xi,\bC_{(\Xi,\tau)})^2\,d\|V_t\|\,dt  \right)^{\frac12}\leq \Cl[con]{con-dn} \, \max\{\mu, \Cr{c-decay} \|u\|\}\,r^\alpha\,,
    \end{equation}
    for every $r \in (0,1)$, and with 
    \begin{equation} \label{eq:change-cone}
        \| O_{(\Xi,\tau)} - \mathrm{Id} \| \leq \Cr{con-dn} \max\{\mu, \Cr{c-decay} \|u\|\}\,.
    \end{equation}
\end{itemize}
We have therefore a correspondence $(\Xi,\tau)\mapsto O_{(\Xi,\tau)}$, and corresponding triple junctions $\bC_{(\Xi,\tau)}=O_{(\Xi,\tau)}(\bC)$, for all points $(\Xi,\tau) \in U_3(0)\times (-9,0)$ such that $\Theta(\Xi,\tau) \geq \sfrac32$. Now let $(\Xi,\tau)$ and $(\Xi',\tau')$ be points in $U_3(0)\times (-9,0)$ such that $\Theta(\Xi,\tau) \geq \sfrac32$ and $\Theta(\Xi',\tau') \geq \sfrac32$, and call $r$ their parabolic distance, namely $r := |\Xi-\Xi'| + \sqrt{|\tau-\tau'|}$. By applying \eqref{eq:decay-nearby} to both $(\Xi,\tau)$ and $(\Xi',\tau')$ at scale $r$ we see as a consequence of the triangle inequality that the (Hausdorff) distance between the cone $\bC_{(\Xi,\tau)}$ and the cone $\tau_{r^{-1}(\Xi'-\Xi)} (\bC_{(\Xi',\tau')})$ (where $\tau_v$ denotes the translation by vector $v$) is bounded by $\Cr{con-dn} \, \max\{\mu, \Cr{c-decay} \|u\|\}\,r^\alpha$. In particular:
\begin{align} 
    \| O_{(\Xi,\tau)} - O_{(\Xi',\tau')} \| &\leq \Cr{con-dn} \, \max\{\mu, \Cr{c-decay} \|u\|\}\,r^\alpha\,, \label{eq:holder-con1} \\
    \dist(\Xi-\Xi',\bS_{(\Xi',\tau')}) + \dist(\Xi'-\Xi,\bS_{(\Xi,\tau)}) &\leq \Cr{con-dn} \, \max\{\mu, \Cr{c-decay} \|u\|\}\,r^{1+\alpha}\,. \label{eq:holder-con2}
\end{align}
We claim that \eqref{eq:holder-con1}-\eqref{eq:holder-con2} imply that, in $U_3(0) \times (-9,0)$, the set $\{\Theta \geq \sfrac32\}=\{\Theta = \sfrac32\}$ is contained in the graph of a $C^{1,\alpha}$ map $(y,t) \in (\bS\cap U_3(0)) \times (-9,0) \mapsto \xi(y,t) \in \bS^\perp$, so that every point $(\Xi,t)$ with $\Theta(\Xi,t) \geq \sfrac 32$ is of the form $(\Xi,t)=(\xi(y,t),y,t)$. The only thing we need to check is that for every $(y,t)$ as above there exists a unique point $(\Xi,t)$ with $\Theta(\Xi,t)\geq \sfrac32$ such that $\bS(\Xi)=y$: the claimed regularity will then be an immediate consequence of \eqref{eq:holder-con1}-\eqref{eq:holder-con2}. From the same estimates it also follows that the tangent to the graph of $\xi(\cdot,t)$ at the point $(\xi(y,t),y,t)$ is $\bS_{(\xi(y,t),y,t)}$. Fix then $\delta > 0$, and consider any $y_0 \in \bS \cap U_3(0)$ and $t_0 \in (-9,0)$. If $\Cr{e_main}$ is sufficiently small, Proposition \ref{p:NH_property} guarantees that $B^{n-k+1}_\delta\times\{y_0\}$ contains a point $\Xi_1=(\xi_1,y_0)$ so that $\Theta(\Xi_1,t_0) \geq \sfrac32$. Suppose by contradiction that this point is not unique, so that there exists $\Xi_2=(\xi_2,y_0)$ with $\Theta(\Xi_2,t_0) \geq \sfrac32$ and $r:=|\Xi_1-\Xi_2|=|\xi_1-\xi_2| > 0$. Choosing $\delta$ small (say $\delta < 1/8$), we have that $r < 1$. Let $\bC_1 = \bC_{(\Xi_1,t_0)}$ and $\bC_2 = \bC_{(\Xi_2,t_0)}$ be the corresponding unique tangent cones, with spines $\bS_1$ and $\bS_2$ respectively. By \eqref{eq:holder-con2}, we have that
\begin{equation} \label{unique-contr1}
    \dist(\Xi_1-\Xi_2, \bS_2) \le \Cr{con-dn} \, \max\{\mu, \Cr{c-decay} \|u\|\} \, r^{1+\alpha}\,.
\end{equation}
On the other hand, $\dist(\Xi_1-\Xi_2, \bS_2) = \dist (\xi_1-\xi_2,O_2(\bS)) = |\bS^\perp (O_2^{-1}(\xi_1-\xi_2))|$, where $O_2=O_{(\Xi_2,t_0)}$. Since $\xi_1-\xi_2 \in \bS^\perp$, and \eqref{eq:change-cone} holds, we have
\begin{equation} \label{unique-contr2}
    \dist(\Xi_1-\Xi_2, \bS_2) \geq r \left( 1 - C \|O_2-\mathrm{Id}\|^2\right) \geq r \left( 1 - C \Cr{con-dn}^2 \, \max\{\mu, \Cr{c-decay} \|u\|\}^2 \right)\,.
\end{equation}
Together, \eqref{unique-contr1} and \eqref{unique-contr2} give 
\[
1 - C \Cr{con-dn}^2 \, \max\{\mu, \Cr{c-decay} \|u\|\}^2 \leq \Cr{con-dn} \, \max\{\mu, \Cr{c-decay} \|u\|\} \, r^{\alpha}\,,
\]
a contradiction. This completes the proof of the existence of the map $\xi$ and its regularity.

\smallskip

Next, let $(X,t)$ be any point on the support of the flow in $\left( U_3(0)\times (-9,0) \right) \setminus \mathrm{graph}\,\xi$, and let $(\Xi(X,t),\tau(X,t))$ be a point in $\mathrm{graph}\,\xi$ with $\tau\ge t$ that minimizes the parabolic distance to $(X,t)$. If $r := |\Xi-X| + \sqrt{\tau-t}$, \eqref{eq:decay-nearby} guarantees that we can apply Theorem \ref{thm:graphical} to the flow
\begin{align*}
    V^{(\Xi,\tau),r}_s :&= (O_{(\Xi,\tau)}^{-1})_\sharp(\iota_{\Xi,r})_\sharp V_{\tau+r^2s} \\
    u^{(\Xi,\tau),r}(Y,s) :&= O_{(\Xi,\tau)}^{-1} \, r \, u(\Xi + r O_{(\Xi,\tau)}(Y), \tau + r^2 s)\,,
\end{align*}
and conclude that $(X,t)$ is contained in a toroidal region (of characteristic scale comparable to $r$) where the flow is a $C^{1,\alpha}$ graph over $\Xi+\bC_{(\Xi,\tau)}$ satisfying the estimates \eqref{expara1} corresponding to $\sigma=\sfrac18$. This shows that $(X,t)$ is a regular point, and, since it is arbitrary, that $\mathrm{graph}\,\xi$ coincides with the singular set. Furthermore, \eqref{eq:change-cone} implies that, upon choosing $\Cr{e_main}$ small, any local graphical region over $\Xi + \bC_{(\Xi,\tau)}$ such that its projection to $\mathbf P_i$ is contained in $\Omega_i$ can be written as a normal graph over $\Omega_i$. Since such graphs all agree with the support of the flow, they must agree on overlaps: hence, we obtain global functions $f_i \in C^{1,\alpha}(\Omega_i;\mathbf P_i)$ such that \eqref{e:main-thm-est} holds with $\Cr{c_main}=\Cr{nonsing1}+\Cr{con-dn}\Cr{c-decay}$. This completes the proof of Theorem \ref{thm:main} under the assumption that $\Theta(0,0) \geq \sfrac32$.

\medskip
\noindent\textbf{Step 2.} In the general case, we first apply Proposition \eqref{prop:tangent-flow} and identify a point $a_\infty$ and a rotation $O_\infty$ so that \eqref{eq:close-to-standard} holds and $O_\infty(\bC)$ is the unique tangent flow at $(a_\infty,0)$. Then, we apply step 1 to the translated, rotated, and slightly rescaled flow
\begin{align*}
    V'_t :&= (O_\infty^{-1})_\sharp(\iota_{a_\infty,\frac23})_\sharp V_{4t/9}\,,\\
u'(X,t) :&= \frac{2}{3}\,O_\infty^{-1}\,u\!\left(a_\infty+\tfrac{2}{3}O_\infty X,\ \tfrac{4t}{9}\right).
\end{align*}
to obtain parametrization for the flow and its singular set over $O_\infty(\bC)$ in $U_2(0)\times (-4,0)$. We then reparametrize over $\bC$, and the proof is complete.
\end{proof}

\section{Unconditional triple junction regularity} \label{lastapp}

In this section we discuss more in detail two classes of Brakke flows to which our main result apply: first, because triple junction singularities are naturally expected to occur; second, because the main structural condition, Assumption (A6), is automatically satisfied.

To begin with, we work in the case when $n=k+1$, and we introduce the notion of Brakke flow equipped with a ``cluster-like'' structure.
\begin{definition}\label{clusterdef}
    We say that a family $\{V_t\}_{t\in I}$ of $k$-varifolds in $U_R\subset \mathbb R^{k+1}$ is {\it cluster-like} if for some $N\in\mathbb N_{\geq 2}$, we have
    families $\{E_i(t)\}_{t\in I}$ ($i=1,\ldots,N$) of open sets in $U_R$ with the following properties.
    \begin{itemize}
\item[(i)] For each $t\in I$, $E_1(t),\ldots,E_N(t)$ 
are pairwise disjoint, and $\mathcal H^k(U_R\setminus\cup_{i=1}^N E_i(t))<\infty$. 
\item[(ii)] For a.e.~$t\in I$, $2\|V_t\|\geq \sum_{i=1}^N\|\partial^* E_i(t)\|$ in $U_R$ as Radon measures. 
\item[(iii)]
For some open set $O\subset U_R$ and 
interval $I'\subset I$, if $V_t$ is a unit-density varifold in $O$ for a.e.~$t\in I'$, then 
$2\|V_t\|= \sum_{i=1}^N\|\partial^* E_i(t)\|$ in $O$ for a.e.~$t\in I'$.
    \end{itemize}
    Here $\partial^* E$ denotes the reduced boundary of the set of finite perimeter 
    $E$, and $\|\partial^* E\|$ is the perimeter measure, so that $\|\partial^*E\|= \mathcal H^{k}\mres_{\partial^* E}$. 
\end{definition}
We note that $({\rm i})$ implies that each
$E_i(t)$ is a set of finite perimeter
(\cite[Proposition 3.62]{AFP}),
and by \cite[Proposition 29.4]{Maggitextbook}, 
\begin{equation}\label{fineqper}
\frac12\sum_{i=1}^N\|\partial^*E_i(t)\|=\sum_{1\leq i<j\leq N}\mathcal H^k\mres_{\partial^*E_i(t)\cap\partial^*E_j(t)}.
\end{equation}
\begin{remark}
    The Brakke flow constructed in \cite{ST_canonical}, with forcing $u \equiv 0$, is precisely cluster-like, see \cite[Theorem 2.11,\,2.12]{ST_canonical}.
The immediate corollary of the following 
Theorem \ref{profor1.1} is that
Theorem \ref{thm:main} is applicable to the flow in \cite{ST_canonical}.
\end{remark}
\begin{theorem}\label{profor1.1}
    Suppose that 
    $(\{V_t\}_{t\in I},\{u(\cdot,t)\}_{t\in I})$  satisfies (A1)-(A5) in $U_R\times I\subset \R^{k+1}\times I$, and further assume that $\{V_t\}_{t\in I}$ is cluster-like. 
    Then, the condition (A6) is automatically satisfied. In particular, Theorem \ref{thm:main}
    is applicable without assuming (A6)
    in this case.
\end{theorem}
\begin{proof}
    We need to check the existence of a
constant $\Cr{conslice}$ as stated in (A6). For any $P_r$ in which 
$(\{V_t\}_{t\in I'},\{u(\cdot,t)\}_{t\in I'})\in {\mathscr N}_{\Cr{onlyunit}}(P_r)$, by Proposition \ref{prouniden} corresponding to $r=3/4$, for a.e.~$t\in (-3r^2/4,0)$ the varifold $V_t$ is unit-density in $B_{3r/4}$.
By Definition \ref{clusterdef}(iii)
and \eqref{fineqper}, 
$V_t={\bf var}(\cup_{i=1}^N\partial^* E_i(t),1)$ in $B_{3r/4}$ for a.e.~$t\in I'$ and, in the notation of (A6), $M_t=\cup_{i=1}^N\partial^* E_i(t)$. By \cite[Theorem 18.11 and Remark 18.13]{Maggitextbook}, the slice of $E_i(t)$
by $\R^{2}\times\{y\}$ has the property
that
\begin{equation}
   \mathcal H^1\Big( \big((\R^{2}\times\{y\})\cap \partial^* E_i(t)\big)\triangle\big(\partial^*((\R^{2}\times\{y\})\cap E_i(t))\big)\Big)=0
\end{equation}
for $\mathcal H^{k-1}$-a.e.~$y$. Thus,
writing $E_i^y(t):= (\R^2\times\{y\})\cap E_i(t)$, in $B_{3r/4}$ 
and for $\mathcal H^{k-1}$-a.e~$y$ it holds
\begin{equation}\label{messon}
    \mathcal H^1(M_t^y\triangle \cup_{i=1}^N \partial^*(E_i^y(t)))=0\,.
\end{equation}
Also by Theorem \ref{thm:graphical} 
(which does not require (A6)), in
$P_{3r/4}\cap \{|x|>r/10\}$,
${\rm spt}\,\|V_t\|$ is represented 
as a $C^{1,\alpha}$ graph over
${\bf C}$. Thus, for all $y\in B_{r/2}^{k-1}$,
$M_t^y\cap \{r/10\leq |x|\leq 3r/4\}$ is 
represented as three graphical $C^{1,\alpha}$ curves over $\hat{\bC}$, and we have by \eqref{anasup}
\begin{equation}\label{grheit}
    {\rm dist}_{H}(M_t^y\cap \{r/8\leq |x|\leq r/2\},{\hat{\bC}}\cap\{r/8\leq |x|\leq r/2\})\leq rK. 
\end{equation}
In terms of $\mathcal H^1$-measure, $M_t^y$
and $\cup_{i=1}^N\partial^*(E_i^y(t))$ can
be identified by \eqref{messon}. In the following two lemmas, by letting $E_i=E_i^y(t)$ after a suitable change of variables, we conclude the validity of \eqref{cotriple} and \eqref{cotriple2}, respectively.
\end{proof}

\begin{lemma}\label{triplemin1}
    Suppose that $E_1,\ldots,E_N\subset B_1^2$ are 
    mutually disjoint open sets with finite perimeter
    such that $\mathcal L^2(B_1^2\setminus \cup_{i=1}^N E_i)=0$. Suppose that 
    $(B_1^2\setminus B_{1/2}^2)\cap (\cup_{i=1}^N \partial^* E_i)$ consists of three $C^1$ curves $\ell_1,\ell_2,\ell_3$ which are represented as 
    $C^1$ graphs over $\hat\bC$ with small $C^1$-norms, and
    assume that ${\rm dist}_H(\cup_{i=1}^3\ell_i,(B_1^2\setminus B_{1/2}^2)\cap \hat{\bf C})\leq K$.
    Then there exists an absolute constant $\Cl[con]{minleng}>0$
    such that for any $s\in[1/2,1)$, we have 
    \begin{equation}
        \frac{1}{s}\mathcal H^1(B_s^2\cap \cup_{i=1}^N\partial^* E_i)\geq 
        \frac{1}{s}\mathcal H^1(B_s^2\cap\hat\bC)-\Cr{minleng} K^2=3-\Cr{minleng}K^2.
        \end{equation}
\end{lemma}
\begin{proof}
    By the assumption, there are exactly three open sets, say $E_1,E_2,E_3$, which are not empty in $B_1^2\setminus B_s^2$. Given these sets, consider the perimeter 
    minimization 
    problem of $\sum_{i=1}^N\mathcal H^1(\partial^* \tilde E_i \cap B_1^2)$ among $\tilde E_1,\ldots,\tilde E_N\subset B_1^2$ with $(B_1^2\setminus B_s^2)\cap\tilde E_i=(B_1^2\setminus B_s^2)\cap E_i$ for $i=1,\ldots,N$ and with $\mathcal L^2(\tilde E_i\cap\tilde E_{i'})=0$ for $i\neq i'$ and $\mathcal L^2(B_1^2\setminus \cup_{i=1}^N \tilde E_i)=0$. By the
    standard compactness theorem of set of finite perimeter,
    there exists a minimizer which we call $\tilde E_1,\ldots,\tilde E_N$. One can prove that $B_s^2\cap \cup_{i=1}^N\partial^* \tilde E_i$ lies in the convex
    hull of the three points $\cup_{i=1}^3 \ell_i\cap \partial B_{s}^2$, and it is locally either a line segment or a
    triple junction of $120^\circ$. Then one can argue that 
    the line segment starting from $\ell_1\cap\partial B_{s}^2$ has another end point being a triple junction,
    from which two lines start and reach to $\partial B_s^2\cap\ell_2$ and $\partial B_s^2\cap\ell_3$ without having another triple junction. In other 
    words, $\tilde E_4,\ldots,\tilde E_N$ are empty and
    $\cup_{i=1}^3\partial^*E_i\cap B_{s}^2$ is a regular
    triple junction. If the triple junction intersects with $\partial B_{s}^2$ at three points which differs from that of $\hat\bC\cap \partial B_{s}^2$ at most by $K$, one can 
    estimate $\mathcal H^1(B_{s}^2\cap \cup_{i=1}^3\partial^* E_i)$ from below by $3s$ minus 
    some absolute constant times $K^2s$. 
    Thus, we proved the claim. 
\end{proof}
\begin{lemma}\label{triplemin3}
    Under the same assumption of Lemma \ref{triplemin1},
    there exists an absolute constant $\Cl[con]{hlengh}>0$
    such that if $K<\Cr{hlengh}$ then for any
    $s\in[1/2,1)$, we have
    \begin{equation}
        \frac{1}{s^3}\int_{B_s^2\cap \cup_{i=1}^N\partial^* E_i}|x|^2\,d\mathcal H^1(x)\geq \frac{1}{s^3}\int_{B_s^2}|x|^2\,d\|\hat\bC\|(x)=1.
    \end{equation}
\end{lemma}
\begin{proof}
    By arguing similarly as in the proof of Lemma \ref{triplemin1},
    we have a minimizer $\tilde E_1,\ldots,\tilde E_N$ which minimizes
    $\int_{B_s^2\cap\cup_{i=1}^N\partial^*\tilde E_i}|x|^2\,d\mathcal H^1(x)$ in 
    $B_s^2$, having $E_i=\tilde E_i$ on $B_1^2\setminus B_s^2$. We claim that $\tilde E_4,\ldots,\tilde E_N$ are empty and
    the boundary $\cup_{i=1}^3\partial^* \tilde E_i$ in $B_{s}^2$ consists of three straight line segments which connect
    $\partial B_{s}^2\cap \ell_i$ ($i=1,2,3$) and 
    the origin. By the minimizing property, we first note that $\cup_{i=1}^N\partial^* \tilde E_i$ 
    in $B_{s}^2$ is in the convex hull 
    of three points $\cup_{i=1}^3 \partial B_{s}^2\cap \ell_i$, and
    that the set is locally either a smooth curve or triple junction
    possibly except for the origin.
    By computing the first variation, one can also prove that the curve satisfies 
    \begin{equation}
        h=\frac{x^\perp}{|x|^2}\,,
    \end{equation}
    where $h$ is the curvature vector.
    This equation has an explicit 
    solution: suppose that the 
curve is given as a graph 
$x=(x_1,f(x_1))$ with $x_1\in\R$. Then the above equation reduces to 
\begin{equation}\label{ODEgeo}
    \frac{f''}{1+(f')^2}=\frac{-x_1f'+f}{f^2+(x_1)^2}.
\end{equation}
Let $f(x_1)=\sqrt{a+(x_1)^2}$ for any
$a>0$. It is straightforward to check by
direct computation that this $f$ satisfies the equation.
The curve behaves like $(x_1,|x_1|)$
for small $a$. We can use
this curve as a barrier function. 
Assume that $\partial B_{s}^2\cap \cup_{i=1}^3\ell_i$ is positioned so that
the line segments to the origin intersect
at the origin with angles bigger than, say, $100^\circ$. Then, if $B_{s}^2\cap \cup_{i=1}^N\partial^*\tilde E_i$ is
not three line segments as claimed above, 
by sliding the explicit curves above (rotated appropriately) with 
varying $a$ and from varying direction, 
one should be able to touch $B_{s}^2\cap \cup_{i=1}^N\partial^*\tilde E_i$ with 
this curve. The point of touching cannot be a triple junction singularity, and also it is not the origin. 
Then by the uniqueness of the solution of ODE \eqref{ODEgeo}, these two curves must coincide, which is a contradiction since
the opening angle of $f$ is $90^\circ$ at most. 
Thus the minimizer consists of three 
straight line segments in $B_s^2$,
and the claim follows immediately.
\end{proof}
We next discuss the case, in arbitrary codimension $n-k \geq 1$, when the flow is
equipped with a mod~3 current structure, which naturally allows for triple junction singularities in the interior.
\begin{theorem} \label{thm:main-mod3}
    Suppose that $(\{V_t\}_{t\in I},\{u(\cdot,t)\}_{t\in I})$ satisfies
    (A1)-(A5) in $U_R\times I$, and further
    assume that, for a.e.~$t\in I$, there exists a mod~3 integral current ${\mathscr S}_t$ whose mod~3 mass measure coincides with $\|V_t\|$ and $\partial {\mathscr S}_t=0$ mod~3 in $U_R$. Then assumption (A6) is automatically satisfied. In particular, Theorem \ref{thm:main}
    is applicable without assuming (A6)
    in this case. 
\end{theorem}
\begin{proof}
By the same argument as in the proof of Theorem \ref{profor1.1}, in $P_r$, we have
a unit-density varifold for a.e.~$t\in (-3r^2/4,0)$ in $B_{3r/4}$ and by assumption, we have a representative
integer rectifiable current (denoted with the same symbol) ${\mathscr S}_t$ with density function
equal to $1$ for a.e.~$t$. For $\mathcal H^{k-1}$-a.e.~$y\in B_{r/2}^{k-1}$, the slice of ${\mathscr S}_t$ by $\R^{n-k+1}\times\{y\}$
is a one-dimensional integer rectifiable current, supported on $M_t^y$ with the inherited orientation and with zero boundary mod~3. By the same argument, we also have \eqref{grheit}. 
If we set this current, after a suitable change of variables, as $P$ in the following two lemmas, we conclude the validity of \eqref{cotriple} and \eqref{cotriple2}, respectively. 
\end{proof}
\begin{lemma}\label{triplemin2}
    Suppose that $P$ is a unit-density
    one-dimensional mod~3 current satisfying $\partial P=0 \; {\rm mod}\, 3$ in $B_1^{n-k+1}$. Suppose that
    $(B_1^{n-k+1}\setminus B_{1/2}^{n-k+1})
    \cap {\rm spt}\,\|P\|$ consists of three $C^1$ curves $\ell_1,\ell_2,\ell_3$ which 
    are represented as $C^1$ graphs over 
    $\hat\bC\times\{0_{n-k-1}\}$ satisfying 
    ${\rm dist}_H(\cup_{i=1}^3\ell_i,(B_1^2\setminus B_{1/2}^2)\cap \hat{\bf C})\leq K$. Then there exists an absolute constant $\Cl[con]{concurr}>0$ such that for any $s\in[1/2,1)$, we have 
    \begin{equation}
    \frac{1}{s}\|P\|(B_s^{n-k+1})\geq 3-\Cr{concurr}K^2.
    \end{equation}
\end{lemma}
\begin{proof}
    The proof is similar to that of Lemma \ref{triplemin1}, except that one minimizes the mass functional among mod~3 one-dimensional integral currents $\tilde P$ with $\partial\tilde P=0$ mod~3 in $B_1^{n-k+1}$ and with $P=\tilde P$
    in $B_1^{n-k+1}\setminus B_s^{n-k+1}$. By the standard compactness theorem for such class of currents, the minimizer $\tilde P$ exists, and furthermore, one
    can argue that $\tilde P$ in $B_s^{n-k+1}$ is the triple junction
    with straight line segments connecting the    three points $\partial B_s^{n-k+1}\cap \cup_{i=1}^3\ell_i$. It lies in a two-dimensional
    affine plane in $\R^{n-k+1}$ away from $\R^2\times\{0_{n-k-1}\}$ by at most $K$. The total length
    of such triple junction can be estimate from below by $3s-\Cr{concurr}K^2s$, so that we have the stated claim as before.
\end{proof}
\begin{lemma}
    Under the same assumption of Lemma \ref{triplemin2},
    there exists an absolute constant $\Cl[con]{concurr2}>0$ such that if $K<\Cr{concurr2}$
    then for any $s\in[1/2,1)$, we have
    \begin{equation}
        \frac{1}{s^3}\int_{B_s^{n-k+1}}|x|^2\,d\|P\|(x)\geq 
        \frac{1}{s^3}\int_{B_s^{n-k+1}\cap (\hat\bC\times\{0_{n-k-1}\})}|x|^2\,d\mathcal H^1(x)-\Cr{concurr2}K^2=1-\Cr{concurr2}K^2.
    \end{equation}
\end{lemma}
\begin{proof}
    Similarly to the previous lemma, one minimizes in this case the mass weighted by $|x|^2$ in $B_s^{n-k+1}$ among
    mod~3 currents, and the minimizer $\tilde P$ 
    exists. One 
    can also conclude that, away from $x=0$, each 
    point of ${\rm spt}\,\|\tilde P\|$ is 
    locally 
    either a smooth curve or triple junction,
    with $\partial B_s^{n-k+1}\cap {\rm spt}\|\tilde P\|=\partial B_s^{n-k+1}\cap\cup_{i=1}^3\ell_i$. Since $\partial B_{s}^{n-k+1}\cap\cup_{i=1}^3\ell_i$ consists of three non-collinear points, there is a unique two-dimensional
    affine plane, denoted by $\hat A$, containing it, and let $\hat x\in \hat A$ be the closest point, in $\hat A$, to the origin. Suppose that $\hat x\neq 0$. Let $\hat S$ be the orthogonal projection map from $\R^{n-k+1}$ to the 3-dimensional subspace containing both $\hat x$ and $\hat A$, namely the subspace $(\hat A-\hat x) \oplus {\rm span}(\hat x)$, and consider the map $F:\R^{n-k+1}\rightarrow \hat S(\R^{n-k+1})$ defined by 
    \begin{equation}
        F(x):=\left\{\begin{array}{ll}
        \hat S(x) & \mbox{if }0\leq x\cdot \hat x\leq |\hat x|,\\
        \hat S(x)-\big(\hat S(x)\cdot\frac{\hat x}{|\hat x|}\big)\frac{\hat x}{|\hat x|}+\hat x
        & \mbox{if }x\cdot\hat x>|\hat x|, \\
        \hat S(x)-\big(\hat S(x)\cdot\frac{\hat x}{|\hat x|}\big)\frac{\hat x}{|\hat x|} &
        \mbox{if }x\cdot\hat x<0\,.
        \end{array}\right.
    \end{equation}
    The map $F$ is Lipschitz with Lipschitz constant equal to $1$, and $|F(x)|\leq |x|$ for all $x\in\R^{n-k+1}$.
    Then the pushforward of $P$ under the map
    $F$ is a non-increasing operation for $\int_{B_{s}^{n-k+1}}|x|^2\,d\|P\|(x)$
    while fixing the three boundary points,
    and we may therefore assume that ${\rm spt}\|\tilde P\|$ in
    $B_{s}^{n-k+1}$ is contained in the image 
    of $F$, which is a subset of the aforementioned  3-dimensional subspace. Let us identify the latter with $\R^3$, with a slight abuse of notation. The image of $F$ is then the region of $\R^3$ limited by $\partial B_s^{3}$, 
    the affine plane $\hat A$ and the subspace $\hat A-\hat x$, which we identify with $\R^2\times\{0_1\}\simeq\R^2$. The assumption implies that $|\hat x|\leq \Cr{concurr2}K s$
    for some absolute constant. Consider then 
    the pushforward of $\tilde P$ in $B_s^3$ 
    by orthogonal projection of $\R^3$ 
    to $\hat A-\hat x=\R^2$ denoted by $G$. 
    The map again reduces the weighted integral,
    that is, 
    \begin{equation}\label{P1eq}
        \int_{B_s^3}|x|^2\,d\|\tilde P\|(x)\geq
        \int_{B_s^2} |x|^2\,d\|G_\sharp\tilde P\|(x),
    \end{equation}
    and $G_\sharp(\tilde P)$ has three boundary points whose distance from 
    $\partial B_s^2$ is less than $\Cr{concurr2}K^2 s$ for some absolute constant and which are positioned close to
    $\hat\bC\cap \partial B_s^2$. By considering
    the minimization problem with mod~3 current setting again on the two-dimensional plane as in Lemma \ref{triplemin3} and using that the
    end points are $\Cr{concurr2}K^2s$ close to
    $\partial B_s^2$, one can conclude that the minimizer is achieved by the three straight lines to the origin and we obtain that
    \begin{equation}\label{P2eq}
    \int_{B_s^2}|x|^2\,d\|G_\sharp P\|(x)\geq \int_{B_s^2\cap\hat\bC}|x|^2\,d\mathcal H^1(x)-\Cr{concurr2}K^2 s^3\,.
    \end{equation}
    Now \eqref{P1eq} and \eqref{P2eq} prove the
    desired inequality. In the case that $\hat x=0$ (that is $\hat A$ is a subspace), then 
    we may use $F$ which is the orthogonal projection of $\R^{n-k+1}$ to $\hat A$ and we may argue similarly. This ends the proof.
    \end{proof}

\section{Concluding remarks} \label{sec:finalrmk}
We conclude this manuscript with some remarks on this result, its assumptions, as well as future research questions stemming from it.

First, if $(\{V_t\}_{t\in I},\{u(\cdot,t)\}_{t\in I})$ happens to be independent of time (so we have $(V,u)$ instead) and satisfies (A1)-(A5), it is natural to assume that condition \eqref{e:the-Holder-exp} reduces to $q=\infty$ and $p>k$. One can also argue (see \cite[Lemma 10.1]{Kasai-Tone}) that $h(x,V)=-u(x)^\perp$ for $\|V\|$-a.e., thus $h\in L^p(\|V\|)$ with $p>k$. In this case, Simon's result \cite{Simon_cylindrical} (where $h$ is assumed to be in $L^\infty(\|V\|)$, but $L^p(\|V\|)$ with $p>k$ should be handled similarly) shows without (A6) that
$(V,u)\in \mathscr N_{\varepsilon}(U_R)$ for 
sufficiently small $\varepsilon>0$ implies that ${\rm spt}\,\|V\|$ is a $C^{1,\alpha}$ perturbation of $\bC$. Thus, (A6) is not needed in the corresponding time-independent case. As already mentioned, (A6) is essential to control terms stemming from various cut-off in Brakke's inequality,
while Simon avoided creating similar terms by utilizing \eqref{simonineq}. We do not know if (A6) may be removed in general. Theorem \ref{thm:main} is establishing a dichotomy: given a flow satisfying (A1)-(A5) and belonging to $\mathcal{N}_\varepsilon(\bC)$ for a sufficiently small $\varepsilon$, \emph{either} the flow is a $C^{1,\alpha}$ perturbation of $\bC$ in a smaller parabolic neighborhood \emph{or} the flow presents certain significant topological degeneracies at the level of its one-dimensional slices, in that (A6) must fail. It would be very interesting to construct examples of Brakke flows with this pathological behavior. By our results, any such flow cannot have multi-phase structure, nor can it have an underlying structure of flow of currents mod~3, which poses a significant difficulty in devising an appropriate construction method.

Concerning future research directions, it would be interesting to explore whether higher regularity of the ``moving free-boundary'' $\mathrm{graph}\,\xi$ in the case $u$ is sufficiently regular or even $u=0$. As anticipated, when $u=0$ the result of Krummel \cite{Krummel} narrows the problem to establishing $C^{2,\alpha}$ regularity, but whether that holds remains open.

Finally, it would be interesting to study whether one may now leverage on having both an end-time regularity at multiplicity-one planes \emph{and} an end-time regularity at multiplicity-one triple junctions $\bC=\mathbf Y^1 \times \R^{k-1}$ to conclude some similar parabolic $\varepsilon$-regularity near cones splitting a codimension 2 Euclidean factor, such as the tetrahedral cone $\mathbf{T}^2 \times \R^{k-2}$, in the spirit of what \cite{CoEdSp} does in the elliptic framework. We remark that even in the setting of \cite{CoEdSp} an underlying cluster-like or current-like structure is assumed in order to enforce the validity of the no-hole property. We expect that parabolic regularity may be proved for analogous classes of Brakke flows as well.

\bibliographystyle{siam}
\bibliography{Large_biblio}

\end{document}